\documentclass[oneside, 11pt]{article}
\usepackage{setspace}
\usepackage{titlesec}
\usepackage[utf8]{inputenc}
\usepackage{tabularx}	
\usepackage{etoolbox}
\usepackage{mathtools}
\usepackage{amsmath}
\usepackage{amssymb} 
\usepackage{amsfonts}
\usepackage{mathrsfs} 
\usepackage{amsthm,pifont}
\usepackage{graphicx}
\usepackage{bbm}
\usepackage[colorlinks=true, pdfstartview=FitV, linkcolor=blue,
citecolor=blue, urlcolor=blue, pagebackref=false]{hyperref}
\usepackage[nottoc,numbib]{tocbibind}
\usepackage{xcolor}
\definecolor{truepurp}{RGB}{160,15,55}

\usepackage{tikz}
\usepackage{pgfplots,pgf}

\usepackage{cite}
\usepackage{braket}
\usepackage{leftidx,tensor}
\usepackage{dsfont}
\usepackage{calc}
\usepackage{caption}
\usepackage{subcaption}
\usepackage{color}
\usepackage[printonlyused,nohyperlinks]{acronym}
\usetikzlibrary{decorations.markings}
\usetikzlibrary{shapes.geometric}
\usetikzlibrary{patterns}
\tikzstyle{vertex}=[circle, draw, inner sep=0pt, minimum size=6pt]

\usepackage[utf8]{inputenc}
\usepackage[english]{babel}
\usepackage{enumitem}

\allowdisplaybreaks

\newtheorem{theorem}{Theorem}[section]
\makeatletter
\patchcmd{\ttlh@hang}{\parindent\z@}{\parindent\z@\leavevmode}{}{}
\patchcmd{\ttlh@hang}{\noindent}{}{}{}
\makeatother
\titleformat*{\section}{\large\bfseries}
\titleformat*{\subsection}{\small\bfseries}
\titleformat*{\subsubsection}{\small\bfseries}
\titleformat*{\paragraph}{\small\bfseries}
\titleformat*{\subparagraph}{\small\bfseries}

\newcommand{\N}{\mathbb{N}}
\newcommand{\R}{\mathbb{R}}
\newcommand{\Z}{\mathbb{Z}}
\newcommand{\E}{\mathbb{E}}
\newcommand{\p}{\mathbb{P}}
\newcommand{\md}{\ensuremath{\mathrm{d}}}
\newcommand{\dxp}{\theta}
\newcommand{\e}{\mathcal{E}}
\newcommand{\eps}{\varepsilon}
\newcommand{\dia}{\text{Diam}}

\newcommand{\mo}{\mathbf{1}}
\newcommand{\mz}{\mathbf{0}}

\newcommand{\cC}{\mathcal{C}}

\newcommand{\cN}{\mathcal{N}}
\newcommand{\cG}{\mathcal{G}}
\newcommand{\cH}{\mathcal{H}}

\newtheorem{lemma}[theorem]{Lemma}
\newtheorem{remark}[theorem]{Remark}

\newtheorem{definition}[theorem]{Definition}

\newtheorem{corollary}[theorem]{Corollary}
\newtheorem{claim}[theorem]{Claim}

\usepackage{geometry}
\begin{document}

	\title{\vspace{-1cm} Behavior of the distance exponent for $\frac{1}{|x-y|^{2d}}$ long-range percolation}

	\author{Johannes B\"aumler\footnote{ \textsc{Department of Mathematics, University of California, Los Angeles}.\\ \text{  \ \ \ \ \ \ } E-Mail: \href{mailto:jbaeumler@math.ucla.edu}{jbaeumler@math.ucla.edu}
	}
	}

	\maketitle
	
	\vspace{-3mm}
	
	\begin{center}
		\parbox{13cm}{ \textbf{Abstract.} 
			We study independent long-range percolation on $\mathbb{Z}^d$ where the vertices $u$ and $v$ are connected with probability asymptotic to $\frac{\beta}{\|u-v\|^{2d}}$ for $\|u-v\|_\infty\geq 2$ and with probability 1 for $\|u-v\|_\infty=1$, where $\beta \geq 0$ is a parameter. It is proven in \cite{baeumler2022distances} that there exists an exponent $\theta=\theta(d,\beta) \in \left(0,1\right]$ such that the graph distance between the origin $\mathbf{0}$ and $x \in \mathbb{Z}^d$ scales like $\|x\|^{\theta}$. We prove that this exponent $\theta(d,\beta)$ is continuous and strictly decreasing as a function in $\beta$. Furthermore, we show that $\theta(d,\beta)=1-\beta+o(\beta)$ for small $\beta$ in dimension $d=1$.
	}
	\end{center}
	
	\let\thefootnote\relax\footnotetext{{\sl MSC Class}: 05C12 ; 60K35 ; 82B43 ; 82B27 }
	\let\thefootnote\relax

	\hypersetup{linkcolor=black}
	\tableofcontents
	\hypersetup{linkcolor=blue}

\section{Introduction}

The study of the chemical distance, also called graph distance or hop-count distance, for long-range percolation is a topic of extensive research \cite{baeumler2022distances,benjamini2001diameter,benjamini2011geometry,biskup2004scaling,biskup2011graph,biskup2019sharp,coppersmith2002diameter,wu2022sharp}. In this paper, we study the so-called {\sl distance exponent} for $\frac{1}{\|x-y\|^{2d}}$-long-range percolation. Our setup is as follows.
Given $\beta \geq 0$ and a function $q:\Z^d \to \left[0,+\infty\right]$ satisfying $q(x)=q(-x)$ for all $x\in \Z^d$ and $q(x)=+\infty$ for $x\in \Z^d$ with $\|x\|=1$, we consider the random graph with vertex set $\Z^d$ where there is an undirected edge between $u$ and $v$ with probability 
\begin{equation*}
	\p(u\sim v) = 1 - \exp(-\beta q(u-v)),
\end{equation*}
independent of all other edges. Here, we use the conventions $0\cdot \infty = \infty$ and $e^{-\infty}=0$. If there is an edge between $u$ and $v$, we also say that $\{u,v\}$ is open, and closed otherwise. The chemical distance $D(u,v)$ between two points $u,v\in \Z^d$ is defined as the minimal number of edges that must be traversed to reach $v$ from $u$. The condition that $q(x)=+\infty$ for all $x\in \Z^d$ with $\|x\|=1$ directly guarantees that $D(u,v)\leq \|u-v\|_1$ almost surely, as all nearest-neighbor edges $\{x,y\}$ are open with probability $1-e^{-\beta \cdot \infty}=1$. In particular, this implies that the resulting random graph is almost surely connected. The asymptotic scaling of the chemical distance is a central characteristic of a random graph and has also been examined for many different models of percolation, see for example \cite{addario2012continuum, antal1996chemical, benjamini2001diameter, berger2004lower, biskup2004scaling, biskup2011graph, biskup2019sharp, coppersmith2002diameter, dembin2022variance, ding2018chemical, drewitz2014chemical, garet2007large, hao2021graph, hernandez2021chemical, nachmias2008critical}. We are interested in the case where the function $q$ decays polynomially, say $q(x)\sim \|x\|^{-s}$. The asymptotics of the chemical distance $D(u,v)$ depends heavily on the value $s$, with the transitions between different regimes happening at $s=d$ and $s=2d$:
\begin{itemize}
	\item For $s<d$, the graph distance between two points is at most $\lceil \frac{d}{d-s} \rceil$ \cite{benjamini2011geometry}.
	\item For $s=d$, the typical distance between two points $u,v$ with $\|u-v\|=n$ is of order $\frac{d \log(n)}{\log(\log(n))}$ \cite{coppersmith2002diameter, wu2022sharp}.
	\item For $s\in (d,2d)$, the typical distance between two points $u,v$ with $\|u-v\|=n$ grows poly-logarithmically like $\log(n)^\Delta$, where $\Delta^{-1} = \log_2 \left(\frac{2d}{s}\right)$ \cite{benjamini2001diameter, biskup2004scaling, biskup2011graph, biskup2019sharp}.
	\item For $s=2d$, the typical distance between two points $u,v$ grows polynomially like $\|u-v\|^\theta$, with distance exponent $\theta=\theta(d,\beta)\in (0,1)$ \cite{baeumler2022distances,coppersmith2002diameter,ding2013distances}.
	\item For $s>2d$, the typical distance between two points grows linearly in the Euclidean distance between these points \cite{baumler2023continuity, berger2004lower}.
\end{itemize}

In this paper, we are interested in the case $s=2d$. Here we assume that, given a parameter $\beta \geq 0$, there is an edge between $u$ and $v$ with probability
\begin{equation}\label{asymptotics}
	\p(u\sim v) = \frac{\beta}{\|u-v\|^{2d}} + \mathcal{O} \left(\frac{1}{\|u-v\|^{2d+1}}\right)
\end{equation}
for $\|u-v\|_\infty \geq 2$ and with probability $1$ for $\|u-v\|_\infty = 1$. As the nearest-neighbor edges are always open, the resulting graph is clearly connected. This simplifies the analysis of the metric structure of the graph, as one does not need to worry about connectivity. For all cases $s\neq 2d$, at least partial results are known when one removes this assumption and looks at a supercritical percolation cluster instead \cite{baumler2023continuity,benjamini2011geometry,biskup2004scaling,coppersmith2002diameter}. Contrary, for $s=2d$, there are no results known so far about the metric structure of the infinite cluster in the supercritical regime.

In \cite[Theorem 7.1]{baeumler2022distances} it is proven that the typical graph distance between the origin $\mz$ and a point $v \in \Z^d$ grows like $\|v\|^\dxp$ for some $\dxp = \dxp(d,\beta) \in \left(0,1\right]$, where $\dxp=1$ if and only if $\beta=0$. See also \cite{ding2013distances} for a proof of this fact in one dimension. Furthermore, the diameter of the box $\{0,\ldots,n\}^d$ also has the same asymptotic growth. More precisely, it is shown that 
\begin{align}\label{eq:baeumler1}
	\|x\|^\dxp \approx_P D(\mz,x) \approx_P \E_\beta\left[D(\mz,x)\right] 
\end{align}
and that
\begin{align}\label{eq:baeumler2}
	n^\dxp \approx_P \dia \left(\{0, \ldots ,n\}^d \right) \approx_P \E_\beta \left[\dia \left(\{0, \ldots ,n\}^d\right) \right]
\end{align}
where the notation $A(n)\approx_P B(n)$ means that for all $\eps >0 $ there exist $0<c<C<\infty$ such that $\p\left(cB(n) \leq A(n) \leq CB(n)\right)> 1-\eps$ for all $n \in \N$, respectively the same for all $x \in \Z^d$. The parameter $\dxp=\dxp(d,\beta)$ is called the {\sl distance exponent}. This quantity is an important characteristic of the random graph, as it is a universal object that only depends on the dimension $d$ and the parameter $\beta$, but not on the precise specifications of the model. Despite the distance exponent being a universal object, we mostly focus on the case where
\begin{align*}
	q(u-v)= \int_{u+\left[0,1\right)^d} \int_{v+\left[0,1\right)^d} \frac{1}{\|t-s\|^{2d}} \md t \md s .
\end{align*}
We explain in section \ref{subsec:cts} below why this is a particular natural choice for the function $q$. For a parameter $\beta \geq 0$, we write $\p_\beta$ for the measure of long-range percolation with respect to this function and we denote its expectation by $\E_\beta$.

The distance exponent is also directly related to the geometry of the limiting long-range percolation metric.
It is proven in \cite{ding2023uniqueness} that, under proper normalization, the set $\{0,\ldots,n\}^d$ considered as a metric space converges to a self-similar random metric space, and the geometry of the limiting metric space also depends on $\dxp(d,\beta)$. For example, the Hausdorff-dimension of this limiting metric space equals $\tfrac{d}{\theta(d,\beta)}$ \cite{baumler2023polynomial,ding2023uniqueness}. So Theorems \ref{theo:strictmonotonicity}, \ref{theo:continuity} and \eqref{eq:largebetabehavior} below show that the Hausdorff dimension of these metric spaces interpolates continuously between $d$ and $+\infty$ as $\beta$ goes from $0$ to $+\infty$. Determining the exact relation between $\beta$ and $\theta(d,\beta)$ is still a major open problem.

\subsection{Main results}
In this paper, we study how the distance exponent $\theta(d,\beta)$ depends on $\beta$. As we consider the dimension as fixed, while varying $\beta$, we also write $\dxp(\beta)$ instead of $\theta(d,\beta)$.
It is clear that the function $\dxp(\beta)$ is monotonically decreasing in $\beta$, as for $\beta_1 < \beta_2$ we can couple the respective measures in such a way that the set of edges resulting from $\p_{\beta_1}$ is a subset of the edge-set sampled from $\p_{\beta_2}$. This directly implies that $\E_{\beta_1} \left[D(u,v)\right] \geq \E_{\beta_2} \left[D(u,v)\right]$ for every $u,v \in \Z^d$, which, in turn, implies $\dxp(\beta_1) \geq \dxp(\beta_2)$. In section \ref{section:monotone}, we show that $\dxp(\beta)$ is even strictly decreasing. 
\begin{theorem}\label{theo:strictmonotonicity}
	The distance exponent $\dxp : \R_{\geq 0} \rightarrow \left(0,1\right]$ is strictly monotonically decreasing.
\end{theorem}

Theorem \ref{theo:strictmonotonicity} shows a unique feature of the case $s=2d$. For all other exponents $s\in (0,+\infty)\setminus\{2d\}$ of the decay of $q(x)\sim \|x\|^{-s}$, the first order of the chemical distance does not depend on $\beta$. Contrary to that, for $s=2d$ each positive $\beta$ influences the first-order growth of the chemical distance uniquely. This variation between different values of the distance exponent $\theta$ happens in a continuous way, which is the statement of the next theorem.

\begin{theorem}\label{theo:continuity}
	The distance exponent $\dxp : \R_{\geq 0} \rightarrow \left(0,1\right]$ is continuous in $\beta$.
\end{theorem}

The existence of the distance exponent was shown in \cite{baeumler2022distances} using a submultiplicative structure of the expected chemical distance, see Lemma \ref{lem:submultiplicativity} below. This submultiplicativity also directly implies the left-continuity of the distance exponent, see Lemma \ref{lem:leftcontinuity}.
The proof of the right-continuity is more difficult. For this, we introduce new measures that interpolate between $\p_{\beta_1}$ and $\p_{\beta_2}$ for $\beta_1,\beta_2 \geq 0$. The (left-)continuity of the distance exponent was also used in \cite[Section 7]{baeumler2022distances} to get the up-to-constants results of \eqref{eq:baeumler1} for all functions $(q(x))_{x\in \Z^d}$ that satisfy \eqref{asymptotics} and to understand the polynomial volume growth of balls in the long-range percolation metric \cite{baumler2023polynomial}.

In \cite{baeumler2022distances} it is also proven that for all dimensions $d$ there exist constants $0<c<C<\infty$ such that 
\begin{align}\label{eq:largebetabehavior}
	\frac{c}{\log(\beta)} \leq \dxp(\beta) \leq \frac{C}{\log(\beta)}
\end{align}
for all $\beta \geq 2$.
So in particular, Theorem \ref{theo:continuity} together with the facts that $\lim_{\beta \to 0} \dxp(\beta) = \dxp(0) = 1$ and $\lim_{\beta \to \infty} \dxp(\beta)=0$ 
show that $\dxp(\beta)$ interpolates continuously between $0$ and $1$, as $\beta$ goes from $+\infty$ to $0$. This also shows that the long-range percolation model with $s=2d$ interpolates continuously between the power-logarithmic growth observed for $s\in (d,2d)$ \cite{biskup2004scaling} and the linear growth observed for $s>2d$ \cite{baumler2023continuity,berger2004lower}. \\

In dimension $d=1$, it is well-known that $\dxp(\beta) \geq 1-\beta$, see for example \cite{coppersmith2002diameter,ding2013distances}. This lower bound arises by considering {\sl cutpoints}  for long-range percolation on the one-dimensional integer line $\Z$. We also review this inequality in section \ref{section:smallbeta} below.
In section \ref{section:smallbeta}, we also show that for small $\beta>0$ this lower bound is indeed a good approximation for $\dxp(\beta)$.

\begin{theorem}\label{theo:smallbeta}
	In dimension $d=1$, the right-hand derivative of the distance exponent $\frac{\md }{\md \beta}\dxp(\beta)$ exists at $\beta=0$ and furthermore one has $\frac{\md }{\md \beta}\dxp(\beta)\Big\vert_{\beta=0} = -1$. This yields that $\dxp(\beta)=1-\beta + o(\beta)$ as $\beta \to 0$.
\end{theorem}

\subsection{The continuous model}\label{subsec:cts}

One particularly interesting case where \eqref{asymptotics} is satisfied is the case when
$q(u-v)= \int_{u+\left[0,1\right)^d} \int_{v+\left[0,1\right)^d} \frac{1}{\|t-s\|^{2d}} \md t \md s$. The function $q$ defined like this satisfies $q(x)=q(-x)$ for all $x\in \Z^d$ and $q(x)=+\infty$ for all $x\in \Z^d$ with $\|x\|_\infty = 1$. We denote the measure of long-range percolation with respect to this function and parameter $\beta > 0$ by $\p_\beta$ and its expectation by $\E_\beta$. We write $p(\beta,\{u,v\})$ for the probability that there is an edge between $u$ and $v$, i.e., 
\begin{equation*}
	p(\beta,\{u,v\}) \coloneqq \p_\beta(u\sim v)  =  1- \exp \left( - \beta \int_{u+\left[0,1\right)^d} \int_{v+\left[0,1\right)^d} \frac{1}{\|t-s\|^{2d}} \md t \md s \right).
\end{equation*}
First note that the function $p(\beta,\{u,v\})$ decays asymptotically like $\frac{\beta}{\|u-v\|^{2d}}$ and satisfies 
\begin{equation*}
	p(\beta,\{u,v\}) = \frac{\beta}{\|u-v\|^{2d}} + \mathcal{O} \left(\frac{1}{\|u-v\|^{2d+1}}\right),
\end{equation*}
see \cite[Example 7.2]{baeumler2022distances}. The reason why the model with inclusion probability $p(\beta,\{u,v\})$ is particularly interesting is because the described discrete percolation model has a self-similarity for all $\beta \geq 0$. This self-similarity comes from a coupling with an underlying continuous model, that we will now describe for any dimension. This will also explain the, at first sight complicated, choice of the function $(q(x))_{x\in \Z^d}$. Consider a Poisson point process $\tilde{\e}$ on $\R^d \times \R^d$ with intensity $\frac{\beta}{2 \|t-s\|_2^{2d}}$. Define the symmetrized version $\e$ by $\e \coloneqq \{(t,s) \in \R^d \times \R^d : (s,t) \in \tilde{\e}\} \cup \tilde{\e}$. We define $\cC\coloneqq \left[0,1\right)^d$. We now define a random graph with vertex set $\Z^d$. For $u,v \in \Z^d$ with $\|u-v\|_\infty \geq 1$ we put an undirected edge between $u$ and $v$ if and only if $\left(\left(u+\cC\right) \times \left(v+\cC\right)\right) \cap \e \neq \emptyset$. The cardinality of $\left(\left(u+\cC\right) \times \left(v+\cC\right)\right) \cap \tilde{\e}$ is a random variable that follows a Poisson distribution with parameter $\int_{u+\cC} \int_{v+\cC} \frac{\beta}{2 \|t-s\|^{2d}} \md s \md t$. As the random variables $\left|\left(\left(u+\cC\right) \times \left(v+\cC\right)\right) \cap \tilde{\e}\right|$ and $\left|\left(\left(v+\cC\right) \times \left(u+\cC\right)\right) \cap \tilde{\e}\right|$ are independent and identically distributed, we get that
\begin{align*}
	&\p\big(\left(\left(u+\cC\right) \times \left(v+\cC\right)\right) \cap \e = \emptyset\big) = \p\left(\left(\left(u+\cC\right) \times \left(v+\cC\right)\right) \cap \tilde{\e} = \emptyset\right)^2\\
	& = \left(e^{-\int_{u+\cC} \int_{v+\cC} \frac{\beta}{2 \|t-s\|^{2d}} \md s \md t}\right)^2
	= e^{-\int_{u+\cC} \int_{v+\cC} \frac{\beta}{ \|t-s\|^{2d}} \md s \md t}
	= 1 - p(\beta, \{u,v\})
\end{align*}
which is exactly the probability that $u \nsim v$ under the measure $\p_\beta$. Note that for $\{u,v\}$ with $\|u-v\|_\infty=1$ we have $\int_{u+\cC} \int_{v+\cC} \frac{\beta}{ \|t-s\|^{2d}} \md s \md t = \infty$. So we really get that all edges of the form $\{u,v\}$ with $\|u-v\|_\infty=1$ are present almost surely. The construction with the Poisson process also implies that the presence of different edges is independent and thus the resulting measure of the random graph constructed above equals $\p_\beta$.
The chosen connection probabilities, respectively the function $q$, have many advantages. First of all, the resulting model is invariant under translation and invariant under the reflection of coordinates, i.e., $q((x_1,\ldots,x_d)) = q((y_1,\ldots,y_d))$ for all $x_1,\ldots,x_d,y_1,\ldots,y_d \in \Z$ with $|x_i|=|y_i|$ for all $i\in \{1,\ldots,d\}$. Additionally, the model is invariant under the permutation of the different coordinate directions, i.e., $q((x_1,\ldots,x_d)) = q((x_{\sigma(1)},\ldots,x_{\sigma(d)}))$ for all permutations $\sigma:\{1,\ldots,d\} \to \{1,\ldots,d\}$. However, the unique characteristic of this model as defined above is the following self-similarity: 
The integral of the form $\int_{A} \int_{B} \frac{1}{ \|t-s\|^{2d}} \md s \md t$ is invariant under scaling of the space $A\times B \subseteq \R^d \times \R^d$ in the sense that for any $\alpha>0$ one has
\begin{align*}
	& \int_{\alpha A} \int_{\alpha B} \frac{1}{ \|t-s\|^{2d}} \md s \md t
	=
	\int_{\alpha A} \int_{B} \alpha^d \frac{1}{ \|t-\alpha s\|^{2d}} \md s \md t\\
	&
	=
	\int_{A} \alpha^d \int_{B} \alpha^d \frac{1}{ \|\alpha t - \alpha s\|^{2d}} \md s \md t
	=
	\int_{A} \int_{B}  \frac{1}{ \| t -  s\|^{2d}} \md s \md t,
\end{align*}
where the first and the second equality follow via integration by substitution. Note that the exponent $2d$ in the denominator of $\frac{1}{ \| t -  s\|^{2d}}$ is the only exponent for which this holds. This is also one of the reasons why one transition between different regimes in the metric structure for long-range percolation happens at $s=2d$.

For some vector $u= (u_1,\ldots,u_d)\in \Z^d$ and $n\in \N_{>0}$ we define the translated box  $V_u^n \coloneqq \prod_{i=1}^{d} \{u_i n , \ldots, (u_i+1)n-1\} = nu  + \prod_{i=1}^{d} \{0,\ldots, n-1\}$.
Using the scaling-invariance of the integral, we see that for all points $u,v \in \Z^d$, and all $n\in \N_{>0}$ one has
\begin{align}\label{scale invariance}
	\notag &\p_\beta\left(V_u^n \nsim V_v^n\right) = \prod_{x\in V_u^n} \prod_{ y \in V_v^n} \p_\beta \left( x \nsim y \right)
	= \prod_{x\in V_u^n} \prod_{ y \in V_v^n} 
	e^{-\int_{x+\cC} \int_{y+\cC} \frac{\beta}{ \|t-s\|^{2d}} \md s \md t}\\
	\notag &
	= e^{- \sum_{x\in V_u^n} \sum_{ y \in V_v^n} \int_{x+\cC} \int_{y+\cC} \frac{\beta}{ \|t-s\|^{2d}} \md s \md t}
	= e^{-\int_{nu+\left[0,n\right)^d} \int_{nv+\left[0,n\right)^d} \frac{\beta}{ \|t-s\|^{2d}} \md s \md t}\\
	&
	= e^{-\int_{u+\cC} \int_{v+\cC} \frac{\beta}{ \|t-s\|^{2d}} \md s \md t}
	= \p_\beta \left( u \nsim v \right)
\end{align}
which shows the self-similarity of the model. Also observe that for any $\alpha \in \R_{> 0}$ the process $\alpha \tilde{\e} \coloneqq \left\{(x,y) \in \R^d \times \R^d : \left(\frac{1}{\alpha}x , \frac{1}{\alpha}y \right)\in \tilde{\e} \right\}$ is again a Poisson point process with intensity $\frac{\beta}{2\|x-y\|^{2d}}$.

\subsection{Notation} 

We use the notation $\mz$ and $\mo$ for the vectors in $\R^d$ that contain only zeros, respectively ones. For $x\in \Z^d$ and $r\geq 0$ we write $B_r(x)=\left\{y\in \Z^d : \|x-y\|_\infty \leq r\right\}$ for the ball of radius $r$ in the $\infty$-norm around $x$. We use the symbol $\cC$ for the box $\left[0,1\right)^d$.
For a vector $u= (u_1,\ldots,u_d)\in \Z^d$ and $n\in \N_{>0}$, we define the box 
\begin{align*}
	V_u^n \coloneqq \prod_{i=1}^{d} \{u_i n , \ldots, (u_i+1)n-1\} = nu  + \prod_{i=1}^{d} \{0,\ldots, n-1\} \text .
\end{align*}For a graph $G=(V,E)$ and $u,v\in V$, we write $D(u,v)$ for the graph distance between $u$ and $v$. 
For some subset $A\subset \Z^d$, we write $D_{A}(u,v)$ for the distance between $u$ and $v$ when we restrict to the edges with both endpoints inside $A$. When it is not clear, we will often write on which graph we work. When we write $\dia\left(A\right) $ we always mean the diameter inside this set, i.e., $\dia(A)=\max_{u,v \in A} D_{A}(u,v)$. We write $\p_\beta$ for the measure of independent long-range percolation with inclusion probabilities $p(\beta,\{u,v\}) = 1-e^{-\int_{u+\cC} \int_{v+\cC} \frac{\beta}{\|x-y\|^{2d} } \md y \md x}$, and $\E_\beta$ for the expectation under this measure. For some edge $e=\{u,v\}$, we write $|e|=|\{u,v\}|=\|u-v\|_\infty$ for the distance in the $\infty$-norm between the endpoints. We use the notation $\log(x)$ for the logarithm with base $e$. We identify the percolation configuration with a random element $\omega \in \{0,1\}^E$, where we say that the edge $e$ exists or is open or present if $\omega(e)=1$. For $\omega \in \left\{0,1\right\}^E$ and $e \in E$, we define the elements $\omega^{e+}, \omega^{e-} \in \{0,1\}^E$ by
\begin{equation*}
	\omega^{e+}(\tilde{e}) = \begin{cases}
		1 & \text{if } \tilde{e} = e \\
		\omega(e) & \text{if }\tilde{e} \neq e
	\end{cases}  \ \text{ and } \ 
	\omega^{e-}(\tilde{e}) = \begin{cases}
		0 & \text{if } \tilde{e} = e \\
		\omega(e) & \text{if } \tilde{e} \neq e
	\end{cases} \text ,
\end{equation*}
so this are the edge sets when we insert, respectively delete, the edge $e$. When we consider the chemical distance in the percolation environment defined by $\omega \in \{0,1\}^E$, we also write $D(u,v;\omega)$ for the graph distance between $u$ and $v$ in the environment represented by $\omega$.  We define the indirect distance $D^\star(A,B)$ between the sets $A,B\subset \Z^d$ as the graph distance in the environment where we removed all edges between $A$ and $B$, which is the minimal length of a path between $A$ and $B$ that does not use an edge $e=\{u,v\}$ with $u\in A, v\in B$.

\subsection{Related work}

Another line of research is to investigate the random graph when one drops the assumption that $q(u-v)=\infty$ for all nearest neighbor edges $\{u,v\}$. For the metric structure of the graph, at least partial results are known for all cases $s\neq 2d$ \cite{baumler2023continuity,benjamini2011geometry,biskup2004scaling,coppersmith2002diameter}. For $s=2d$, it is not known up to now how the typical distance in long-range percolation grows, when $q(\{u,v\})<\infty$ for nearest neighbor edges $\{u,v\}$. We also conjecture a polynomial growth in the Euclidean distance, whenever an infinite cluster exists.

Regarding the existence of an infinite cluster in dimension $d=1$, there is a change of behavior at $s=2$. As proven by Aizenman, Newman, and Schulman in \cite{aizenman1986discontinuity, schulman1983long, newman1986one}, an infinite cluster cannot emerge for $s > 2$ and for $s=2, \beta \leq 1$, no matter how small $\p\left(k \nsim k+1\right)$ is \cite{aizenman1986discontinuity}. On the other hand, an infinite cluster can emerge for $s<2$ and $s=2,\beta > 1$ \cite{newman1986one}. See also \cite{duminil2020long} for a new proof of these results. In \cite{aizenman1986discontinuity} the authors proved that there is a discontinuity in the percolation density for $s=2$, contrary to the situation for $s<2$, as proven in \cite{berger2002transience,hutchcroft2021power}. For models, for which an infinite cluster exists, the behavior of the percolation model at and near criticality is also a well-studied question (cf. \cite{baumler2022isoperimetric, barsky1991percolation, berger2002transience, borgs2005random, damron2017chemical, hutchcroft2021power, hutchcroft2021critical,hutchcroft2022sharp}). 

Also the Ising model on $\Z$ with interaction energy $J\left(\{x,y\}\right)=|x-y|^{-s}$ is a well-studied object, and several results regarding the existence and continuity of phase transitions are known \cite{aizenman1988discontinuity,dyson1969existence,frohlich1982phase,imbrie1988intermediate}.

In dimension $d=1$, the behavior of the mixing time is also a property that exhibits a transition at $s=2$, as proven in \cite{benjamini2008long}. On the line segment $\{0,\ldots,n\}$ the mixing time grows quadratically in $n$ for $s>2$ and is of order $n^{s-1}$ for $1<s<2$. The behavior of the mixing time for $s=2$ is still open, but we conjecture a similar behavior to that of the chemical distance, namely that the mixing time interpolates between $n$ and $n^2$, as $\beta$ goes from $+\infty$ to $0$. A better understanding of the mixing time is useful to study the heat kernel and understand the long-time behavior of the simple random walk on the cluster. In \cite{crawford2012simple, crawford2013simple} Crawford and Sly give bounds on the heat kernel and prove a scaling limit for the case $s \in \left(d,d+1\right)$.

\section{Technical preliminaries}\label{section:Russo}

In this chapter, we provide a few technical preliminaries that will be used later in the proofs of the main theorems; the first one is a version of Russo's formula. The classical Russo's formula, also called Russo-Margulis lemma, see for example \cite[Section 1.3]{heydenreich2017progress} or \cite{margulis1974probabilistic,russo1981critical} for the original papers, is a formula for i.i.d. bond percolation. It states that for any finite graph $(V,E)$ and any increasing event $A$ 
\begin{align}\label{eq:classicalrusso}
	\frac{\md }{\md p} \p_p\left(A\right) = \sum_{e\in E} \p_p\left( e \text{ is pivotal for } A \right) \text ,
\end{align}
where we say that an edge $e$ is pivotal for an event $A$ when changing the status of $e$ also changes the occurrence of the event $A$. Note that it does not depend on the occupation status of the edge $e$ whether $e$ is pivotal for $A$. Russo's formula \eqref{eq:classicalrusso} tells us how the probability of an event changes for i.i.d. percolation when varying the connection probability $p$. We modify this formula in two ways. First of all, we adapt it to long-range percolation, where the inclusion probabilities of the edges are not identical. Secondly, we give a formula that determines the derivative of the expectation of  a general function rather than just the probability of a given event.

\begin{lemma}[Russo's formula for expectations]\label{lem:russoexp}
	Let $G=(V,E)$ be a finite graph with a set of inclusion probabilities $\left(p(\beta, e)\right)_{e\in E ,\beta \geq 0}$, where $\beta \mapsto p(\beta, e)$ is continuously differentiable on $\R_{\geq 0}$ for all $e \in E$. By $\p_\beta$ we denote the Bernoulli product measure on $\left\{0,1\right\}^E$ with inclusion probabilities $\left(p(\beta, e) \right)_{e \in E}$. We write $\E_\beta$ for the expectation under the measure $\p_\beta$. We write $p^\prime(\beta,e)=\frac{\md}{\md \beta} p(\beta,e)$ for the derivative of $p(\beta,e)$ with respect to $\beta$. Let $f: \left\{0,1\right\}^E \rightarrow \R$ be a function. Then
	\begin{equation}\label{eq:russoexp}
		\frac{\md }{\md \beta} \E_\beta\left[f(\omega)\right] = \sum_{e \in E} p^\prime(\beta, e) \E_\beta\left[f(\omega^{e+})-f(\omega^{e-})\right] .
	\end{equation}
\end{lemma}
The lemma is stated for any set of continuously differentiable functions $p(\beta, e)$, but one can also always think of the case where $V$ is a finite subset of $\Z^d$ and where $p(\beta, \{u,v\}) = 1-e^{-\beta \int_{u+\cC} \int_{v+\cC} \frac{1}{\|x-y\|^{2d}}\md x \md y }$ for $u,v\in V$, as we only apply it to this case. The proof is almost identical to the proof of the classical Russo's formula, so we omit it here.\\

\noindent
We now consider the case where $p(\beta,\{u,v\})= 1-e^{-\int_{u+\cC} \int_{v+\cC} \frac{\beta}{\|x-y\|^{2d}} \md x \md y}$. Note that $p(\beta,e)$ decays like $\frac{\beta}{|e|^{2d}}$ as $|e|$ tends to infinity. By the triangle inequality we have for all $x\in u+\cC, y\in v+\cC$ that
\begin{equation*}
	\|x-y\| = \|u-v+(x-u)-(y-v)\| \leq \|u-v\| + \|(x-u)-(y-v)\| 
	\leq \|u-v\| + \sqrt{d} 
\end{equation*}
and
\begin{equation*}
	\|u-v\| = \|x-y+(u-x)-(v-y)\| \leq \|x-y\| + \|(u-x)-(v-y)\| 
	\leq \|x-y\| + \sqrt{d} .
\end{equation*}
Combining these inequalities we see that
\begin{align*}
	\|u-v\|-\sqrt{d} \leq \|x-y\|\leq \|u-v\|+\sqrt{d}
\end{align*}
and thus, for $u,v\in \Z^d$ with $\|u-v\| > \sqrt{d}$, one can bound $q(u-v)$ from above and below by
\begin{align}\label{eq:upperandlowerbound}
	\frac{1}{\left(\|u-v\| + \sqrt{d} \right)^{2d}} \leq \int_{v+\cC}\int_{u+\cC} \frac{1}{\|x-y\|^{2d}} \md y \md x
	\leq \frac{1}{\left(\|u-v\| - \sqrt{d} \right)^{2d}} \text .
\end{align}
Also note for all $u,v\in \Z^d$ with $\|u-v\|_\infty \geq 2$ and all $x\in u+\left[0,1\right)^d, y\in v+\left[0,1\right)^d$ one has $\|x-y\| \geq 1$. Integrating this over all $x\in u+\left[0,1\right)^d, y\in v+\left[0,1\right)^d$ yields that
\begin{align}\label{eq:integralboundedby1}
	1 \geq \int_{v+\cC}\int_{u+\cC} \frac{1}{\|x-y\|^{2d}} \md y \md x .
\end{align}
Next, we consider the derivative of $p(\beta,e)$ for edges $e = \{u,v\}$ with $\|u-v\|_\infty \geq 2$. By the chain rule we have
\begin{align*}
	\frac{\md }{\md \beta} p(\beta, e) = p^\prime(\beta, e) = \int_{u+\cC} \int_{v+\cC} \frac{1}{\|x-y\|^{2d}} \md x \md y \ e^{-\int_{u+\cC} \int_{v+\cC} \frac{\beta}{\|x-y\|^{2d}} \md x \md y}.
\end{align*}
Using \eqref{eq:upperandlowerbound} and \eqref{eq:integralboundedby1}, we get that for all $u,v \in \Z^d$ with $\|u-v\|_\infty > \sqrt{d}$
\begin{equation}\label{eq:derivativebound1}
	p^\prime(\beta, \{u,v\}) 
	\overset{\eqref{eq:upperandlowerbound}}{\geq} 
	\frac{e^{-\int_{u+\cC} \int_{v+\cC} \frac{\beta}{\|x-y\|^{2d}} \md x \md y}}{\left(\|u-v\| + \sqrt{d} \right)^{2d}}
	\overset{\eqref{eq:integralboundedby1}}{\geq} 
	\frac{e^{-\beta}}{\left(\|u-v\| + \sqrt{d} \right)^{2d}} .
\end{equation}
Further, for all $u,v\in \Z^d$ with $\|u-v\| > \sqrt{d}$, inequality \eqref{eq:upperandlowerbound} also implies that
\begin{equation}\label{eq:derivativebound2}
	p^\prime(\beta, \{u,v\}) \overset{\eqref{eq:upperandlowerbound}}{\leq} \frac{e^{-\int_{u+\cC} \int_{v+\cC} \frac{\beta}{\|x-y\|^{2d}} \md x \md y}}{\left(\|u-v\| - \sqrt{d} \right)^{2d}}
	\leq \frac{1}{\left(\|u-v\| - \sqrt{d} \right)^{2d}} .
\end{equation}
Inequalities \eqref{eq:derivativebound1} and \eqref{eq:derivativebound2} imply that for fixed $\beta >0$ the derivative $p^\prime(\beta,\{u,v\})$ satisfies $p^\prime(\beta,\{u,v\})=\Theta\left(\frac{1}{\|u-v\|^{2d}}\right)$ as $\|u-v\|\to \infty$. So in particular, the derivative $p^\prime(\beta,\{u,v\})$ is of the same order as $p(\beta,\{u,v\})$.

Above we bounded the derivative $p^\prime (\beta, e)$ from above and below. We also want to bound the connection probability $p(\beta,e)$ from above and below. As $1-e^{-s}\geq \frac{s}{2}\wedge\frac{1}{2}$ for all $s \geq 0$ one has that
\begin{align}\label{eq:jensensbound}
	p(\beta, \{u,v\})\geq \int_{u+\cC} \int_{v+\cC} \frac{\beta}{2 \|x-y\|^{2d}} \md x \md y \wedge \frac{1}{2} 
	\overset{\eqref{eq:upperandlowerbound}}{\geq} \frac{\beta}{2\left( \|u-v\| + \sqrt d \right)^{2d}} \wedge \frac{1}{2}  
\end{align}
for all $u,v \in \Z^d$ with $\|u-v\|>\sqrt{d}$.
On the other hand one has $1-e^{-s}\leq s$ and thus one can upper bound the connection probability by
\begin{align}\label{eq:connection probability upper bound}
	p(\beta, \{u,v\}) \leq \int_{u+\cC} \int_{v+\cC} \frac{\beta}{ \|x-y\|^{2d}} \md x \md y \overset{\eqref{eq:upperandlowerbound}}{\leq} \frac{\beta}{\left( \|u-v\| - \sqrt d \right)^{2d}}
\end{align}
for $\|u-v\| > \sqrt{d}$.

\section{Asymptotic behavior of $\dxp(\beta)$ for small $\beta$ and $d=1$}\label{section:smallbeta}

In this section, we prove Theorem \ref{theo:smallbeta}, i.e., that $\dxp(\beta)=1-\beta+o(\beta)$ for $\beta \to 0$ in dimension $d=1$.
Determining the asymptotic behavior of $\dxp(\beta)$ for $\beta\to 0$ in dimension $d\geq 2$ is more difficult for several reasons. First, there is no lower bound on $\dxp(\beta)$ that arises from considering cut points or something similar. The notion of cut points and its implication on the distance exponent $\dxp(\beta)$ in dimension $d=1$ will be explained below. Secondly, it is not clear which pair of vertices $x,y \in V_\mz^n$ maximizes the expected distance $\E_\beta \left[D_{V_\mz^n}(x,y)\right]$ in dimension two or higher, i.e., whether a similar statement of equation \eqref{eq:maximum of pairs} below holds in dimension $d\geq 2$. However, for all dimensions $d$ there exists a constant $c>0$ such that $\dxp(\beta) \leq 1-c\beta$ for $\beta$ small enough. This can already be shown with the exact same technique that was used in \cite{coppersmith2002diameter} to prove that $\theta(\beta) < 1$.

But now let us consider dimension $d=1$ again.
Here we have $\dxp(0)=1$ and it is well known that $\dxp(\beta)\geq 1-\beta$ (see \cite{coppersmith2002diameter, ding2013distances}).  
For the sake of completeness, we start with a short sketch of the proof of this lower bound.
For this, we define the notion of a cut point. We say that the vertex $w\in\{1,\ldots,n-2\}$ is a cut point $\big($for $\{0,\ldots,n-1\}\big)$ if there exists no open edge $\{u,v\}$ with $0\leq u < w <v \leq n-1$. By the independence of all these edges, we have that
\begin{align}\label{eq:individual cutpoint}
	\notag \p_\beta\left(w \text{ is a cut point}\right) & = \prod_{0\leq u < w} \ \prod_{w < v \leq n-1} e^{-\beta \int_{u}^{u+1} \int_{v}^{v+1} \frac{1}{|x-y|^2} \md x \md y }
	= e^{-\beta \int_{0}^{w} \int_{w+1}^{n} \frac{1}{|x-y|^2} \md x \md y } \\
	& \geq  e^{-\beta \int_{0}^{w} \int_{w+1}^{\infty} \frac{1}{|x-y|^2} \md x \md y } 
	= e^{-\beta \int_0^w \frac{1}{w+1-y} \md y}
	= e^{-\beta\log(w+1)} \geq n^{-\beta}.
\end{align}
As the distance between $0$ and $n-1$ is lower bounded by the number of cut points between $0$ and $n-1$ we get, by linearity of expectation, that
\begin{align}\label{eq:cutpoint lowerbound}
	\notag \E_\beta\left[D_{[0,n-1]}(0,n-1)\right] & \geq \E_\beta \left[ \left|\{w : w \text{ is a cut point}\} \right| \right] = \sum_{w=1}^{n-2} \p_\beta\left(w \text{ is a cut point}\right)\\
	& \geq (n-2) n^{-\beta} = \Omega\left(n^{1-\beta}\right) \ 
\end{align}
which shows that $\dxp(\beta) \geq 1-\beta$.  The emergence of cut points for $\beta < 1$ is also related to the fact that the phase transition is discontinuous for one-dimensional long-range percolation with quadratic decay of the connection probability \cite{aizenman1986discontinuity}. So Theorem \ref{theo:smallbeta} shows that for small $\beta$ the cut points really contribute significantly to the chemical distance.\\

As a first step towards the proof of Theorem \ref{theo:smallbeta}, we remind ourselves about the submultiplicativity of the expected distance and the definition of $\theta(\beta)$ using the submultiplicativity.  This was proven in \cite[Lemma 2.3]{baeumler2022distances}.

\begin{lemma}\label{lem:submultiplicativity}
	For all dimensions $d$ and all $\beta \geq 0$, the sequence
	\begin{equation}\label{eq:Lambda}
		\Lambda(n) = \Lambda(n,\beta)\coloneqq\max_{ u,v \in \left\{0, \ldots ,n-1\right\}^d} \E_\beta \left[D_{V_\mz^n}(u,v)\right] +1
	\end{equation}
	is submultiplicative and furthermore one has
	\begin{align}\label{eq:limitsinf}
		\dxp(\beta) = \lim_{n \to \infty} \frac{\log\left( \Lambda(n, \beta) \right) }{\log(n)} = \inf_{n\geq 2} \frac{\log\left( \Lambda(n, \beta) \right) }{\log(n)}  \text .
	\end{align}
\end{lemma}

\begin{remark}\label{remark:super}
	The sequence $\Lambda(n,\beta)$ is not just submultiplicative, but also {\sl super}multiplicative in the following sense. For all $\beta \geq 0$ there exists a constant $c(\beta) \in \left(0,1\right)$ such that $c(\beta) \Lambda\left(m,\beta\right) \Lambda\left(\ell,\beta\right) \leq \Lambda\left(m\ell,\beta\right)$ for all $m,\ell \in \N$, see \cite[Lemma 5.5]{baeumler2022distances}. This directly implies that for every $m\in \N_{\geq 2}$ the sequence $b_{k} \coloneqq \log(c(\beta)\Lambda \left(m^k, \beta\right))$ is superadditive so that by Fekete's lemma one has 
	\begin{equation*}
		\theta(\beta) = \lim_{k\to \infty} \frac{\log\left(\Lambda(m^k,\beta)\right)}{\log(m^k)} = \lim_{k\to \infty} \frac{\log\left(c(\beta)\Lambda(m^k,\beta)\right)}{\log(m^k)} = \sup_{k \in \N} \frac{\log\left(c(\beta)\Lambda(m^k,\beta)\right)}{\log(m^k)} .
	\end{equation*}
	As this holds for all $m \geq 2$ and $k\in \N$, we get together with \eqref{eq:limitsinf} that
	\begin{equation*}
		\inf_{n\geq 2} \frac{\log\left( \Lambda(n, \beta) \right) }{\log(n)} = \dxp(\beta) = \sup_{n\geq 2} \frac{\log\left( c(\beta) \Lambda(n, \beta) \right) }{\log(n)},
	\end{equation*}
	which directly implies that
	\begin{align}\label{eq:subsuperimpli}
		n^{\theta(\beta)} \leq \Lambda(n,\beta) \leq c(\beta)^{-1} n^{\theta(\beta)} .
	\end{align}
	for all $n \geq 2$.\\
\end{remark}

The inequality $\theta(\beta)\geq 1-\beta$ directly implies that $\liminf_{\beta \to 0} \frac{\dxp(\beta)-\dxp(0)}{\beta} \geq -1$. So in order to prove Theorem \ref{theo:smallbeta}, it suffices to show that $\limsup_{\beta \to 0} \frac{\dxp(\beta)-\dxp(0)}{\beta} \leq -1$.
Our main tools for this are Lemma \ref{lem:submultiplicativity} and Russo's formula for expectations \eqref{eq:russoexp}.

\begin{proof}[Proof of Theorem \ref{theo:smallbeta}]
	From the definition of $\Lambda(n,\beta)$ in \eqref{eq:Lambda}, it is not clear which pair of points $(u,v) \in \{0,\ldots,n-1\}^2$ attains the maximum. We first show that for fixed $n$ and $\beta$ very small $\big($compared to $\tfrac{1}{n}\big)$ we have that
	\begin{equation*}
		\Lambda(n,\beta) = \E_\beta\left[D_{V_0^n}(0,n-1)\right] + 1.
	\end{equation*}
	To see this, note that on the one hand for any $u,v \in \{0, \ldots ,n-1\}$ we have
	\begin{align*}
		\E_\beta\left[D_{V_0^n}(u,v)\right] \leq |u-v| \text ,
	\end{align*}
	whereas on the other hand we have
	\begin{align*}
		\E_\beta\left[D_{V_0^n}(0,n-1)\right] \geq (n-1)\p_\beta\left( \bigcap_{e\in E} \{e \text{ closed}\} \right) \text,
	\end{align*}
	where $E$ is the set of all edges of length at least $2$ in the graph with vertex set $\left\{0, \ldots ,n-1\right\}$.
	As the probability of the event $\bigcap_{e\in E} \{e \text{ closed}\}$ tends to 1 for $\beta \rightarrow 0$, we see that
	\begin{align}\label{eq:maximum of pairs}
		\E_\beta\left[D_{V_0^n}(0,n-1)\right] = \max_{ u,v \in \left\{0, \ldots ,n-1\right\}} \E_\beta\left[D_{V_0^n} (u,v)\right]
	\end{align}
	for small enough $\beta$, where small enough of course depends on $n$. 
	This also already implies that $\Lambda(n,\beta)$ is  differentiable at $\beta=0$, as the inclusion probabilities $p(\beta,\{u,v\})$ are differentiable in $\beta$, and hence $\E_\beta\left[D_{V_0^n} (0,n-1)\right]$ is also differentiable in $\beta$.\\

	Note that $\Lambda(n,0)=n$ and thus $\frac{\log(\Lambda(n,0))}{\log(n)}=1$. So for fixed $\beta>0$ one has
	\begin{align*}
		&\frac{\theta(\beta)-\theta(0)}{\beta} = \frac{\left(\inf_{n\geq 2} \frac{\log\left( \Lambda(n, \beta) \right) }{\log(n)}\right)-1}{\beta} 
		=
		\inf_{n\geq 2} \left( \frac{\log\left( \Lambda(n, \beta) \right) }{\beta\log(n)} - \frac{1}{\beta} \right)\\
		&
		=
		\inf_{n\geq 2} \left( \frac{\log\left( \Lambda(n, \beta) \right) }{\beta\log(n)} - \frac{\log\left(\Lambda(n,0)\right)}{\beta\log(n)} \right)
		=
		\inf_{n\geq 2} \left( \frac{\log\left( \Lambda(n, \beta) \right) - \log\left(\Lambda(n,0)\right)}{\beta\log(n)} \right).
	\end{align*}
	Taking the limit superior for $\beta \searrow 0$ at this formula and interchanging the limit superior with the infimum, we get that 
	\begin{align}\label{eq:derivative}
		& \limsup_{\beta \searrow 0} \frac{\dxp(\beta)-\dxp(0)}{\beta} = \limsup_{\beta \searrow 0} \inf_{n \geq 2} \frac{\log(\Lambda(n,\beta))- \log(\Lambda(n,0))}{\beta \log(n)} \notag \\
		& \leq \inf_{n \geq 2} \limsup_{\beta \searrow 0}  \frac{\log(\Lambda(n,\beta))- \log(\Lambda(n,0))}{\beta \log(n)}   = \inf_{n \geq 2} \frac{1}{\log(n)} \frac{\md }{\md \beta} \log(\Lambda(n,\beta)) \Big\vert_{\beta=0} \notag \\
		& = \inf_{n \geq 2} \frac{1}{\log(n) \Lambda(n,0)}  \frac{\md }{\md \beta} \Lambda(n,\beta) \Big\vert_{\beta=0}.
	\end{align}
	So we have to calculate $\frac{\md }{\md \beta} \Lambda(n,\beta) \Big\vert_{\beta=0}$.  For $e\in E$, remember that $\omega^{e+}$ was defined to be the environment, where we added the edge $e$ (or do nothing in case it already existed before).  Using \eqref{eq:maximum of pairs}, we see that
	\begin{equation*}
		\Lambda(n,\beta) - \Lambda(n,0) = \E_\beta\left[D_{V_0^n} (0,n-1)\right] - \E_0\left[D_{V_0^n}(0,n-1)\right]
	\end{equation*}
	for all small enough $\beta>0$. So Russo's formula for expectations \eqref{eq:russoexp} gives that
	\begin{align}\label{eq:derivativestart}
		\notag	\frac{\md }{\md \beta} \Lambda(n,\beta) \Big\vert_{\beta=0} & =
		\lim_{\beta \searrow 0} \frac{\Lambda(n,\beta) - \Lambda(n,0)}{\beta} \overset{\eqref{eq:maximum of pairs}}{=} \lim_{\beta \searrow 0} \frac{\E_\beta\left[D_{V_0^n} (0,n-1)\right] - \E_0\left[D_{V_0^n}(0,n-1)\right] }{\beta} \\
		& \notag =	\sum_{e\in E} p^\prime(0,e)  \E_0 \left[ D_{V_0^n} (0,n-1;\omega^{e+}) - D_{V_0^n} (0,n-1;\omega^{e-})  \right] \\
		&
		=
		\sum_{e\in E} p^\prime(0,e)  \E_0 \left[ D_{V_0^n} (0,n-1;\omega^{e+}) - (n-1)  \right] .
	\end{align}
	In the environment $\omega^{e+}$ sampled by $\p_0$, where only the nearest-neighbor edges and the edge $e$ are present, the shortest path from $0$ to $n-1$ will also take the edge $e$. By taking the edge $e=\{u,v\}$, the distance between $0$ and $n-1$ decreases by $|e|-1$, and thus equals $D(0,n-1;\omega^{e+}) = n-1-(|e|-1)=n-|e|$.
	For $d=1$, we get from \eqref{eq:derivativebound1} that $p^\prime(0,\{u,v\}) \geq \frac{1}{(|u-v|+1)^2}$ for $|u-v|>1$. With this we can upper bound the derivative computed in \eqref{eq:derivativestart} and obtain that
	\begin{align}\label{eq:onedimensional derivative}
		\notag	\frac{\md }{\md \beta} \Lambda(n,\beta) \Big\vert_{\beta=0} & = \sum_{e\in E} p^\prime(0,e) \E_0 \left[ D_{V_0^n}(0,n-1;\omega^{e+}) - (n-1)  \right] = - \sum_{e\in E} p^\prime(0,e) (|e|-1) \\
		& \leq - \sum_{e\in E} \frac{1}{(|e|+1)^2} (|e|-1) =  \sum_{k=0}^{n-3} \sum_{j=k+2}^{n-1} \frac{1-|j-k|}{(|j-k| + 1)^2}	
		=  \sum_{k=0}^{n-3} \ \sum_{\ell=2}^{n-1-k} \frac{1-\ell}{(\ell+1)^2} \text .
	\end{align}
	For $\ell \in \N$, we have $\frac{-\ell}{(\ell+1)^2} \leq \frac{2}{\ell^2}- \frac{1}{\ell}$, as we will show now. One has
	\begin{align*}
		& \frac{-\ell}{(\ell+1)^2} \leq \frac{2}{\ell^2}- \frac{1}{\ell} \Leftrightarrow
		-\ell^3 \leq (\ell+1)^2 (2-\ell) = (\ell^2+2\ell+1)(2-\ell)\\
		\Leftrightarrow \ & \ell^3 \geq (\ell^2+2\ell+1) (\ell-2)= \ell^3 -2\ell^2 +2\ell^2 -4\ell +\ell -2 
		\Leftrightarrow \  0 \geq -3\ell -2
	\end{align*}
	and the last line is clearly true. Using that $\frac{-\ell}{(\ell+1)^2} \leq \frac{2}{\ell^2}- \frac{1}{\ell}$  we also get that $\frac{1-\ell}{(\ell+1)^2} \leq \frac{1}{(\ell+1)^2} + \frac{2}{\ell^2} - \frac{1}{\ell} \leq \frac{3}{\ell^2} -\frac{1}{\ell}$. Inserting this into \eqref{eq:onedimensional derivative}, we get that
	\begin{align*}
		\notag	\frac{\md }{\md \beta} \Lambda(n,\beta) \Big\vert_{\beta=0} 
		& =   \sum_{k=0}^{n-3} \ \sum_{\ell=2}^{n-1-k} \frac{1-\ell}{(\ell+1)^2}  \leq \sum_{k=0}^{n-3} \ \sum_{\ell=2}^{n-1-k} \left(\frac{3}{\ell^2} -\frac{1}{\ell}\right) \leq 
		\sum_{k=0}^{n-3} \sum_{\ell=2}^{\infty} \frac{3}{\ell^2} + \sum_{k=0}^{n-3} \ \sum_{\ell=2}^{n-1-k} \frac{-1}{\ell}
		\\
		& \leq \sum_{k=0}^{n-3} 6 + \sum_{k=0}^{n-3} \ \sum_{\ell=2}^{n-k-1} \frac{-1}{\ell}
		\leq
		7n + \sum_{k=0}^{n-3} \ \sum_{\ell=1}^{n-k-1} \frac{-1}{\ell}\\
		&
		\leq 7n + \sum_{k=0}^{n-3}  \int_{1}^{n-k} \frac{-1}{s} \md s = 7n - \sum_{k=0}^{n-3} \log(n-k)  
		= 7n - \sum_{k=3}^{n} \log(k) \\
		& \leq 7n - \int_{2}^{n} \log(s) \md s 
		= 7n - \Big[- s + s\log(s)\Big]_2^n \leq 8n + 4 - n \log(n) \text .
	\end{align*}
	Inserting this into \eqref{eq:derivative} gives
	\begin{align*}
		& \limsup_{\beta \searrow 0} \frac{\dxp(\beta)-\dxp(0)}{\beta}  \leq \inf_{n \geq 2} \frac{1}{\Lambda(n,0)\log(n)}  \frac{\md }{\md \beta} \Lambda(n,\beta) \Big\vert_{\beta=0} \\
		& =  \inf_{n \geq 2} \frac{1}{n\log(n)}  \frac{\md }{\md \beta} \Lambda(n,\beta) \Big\vert_{\beta=0}
		\leq \inf_{n \geq 2} \frac{8n + 4 - n\log(n)}{n\log(n)} \leq -1  ,
	\end{align*}
	where the infimum is achieved when taking $n \to \infty$. As $\dxp(\beta)\geq 1-\beta$, and thus $\liminf_{\beta \searrow 0} \frac{\dxp(\beta)-\dxp(0)}{\beta} \geq -1$, this finishes the proof of Theorem \ref{theo:smallbeta}.
\end{proof}

\section{Strict monotonicity of the distance exponent}\label{section:monotone}

In this chapter, we prove Theorem \ref{theo:strictmonotonicity}, i.e.,  that the function $\dxp(\beta)$ is strictly decreasing in $\beta$. It was known before, see \cite{baeumler2022distances,coppersmith2002diameter, ding2013distances}, that $\dxp(\beta)$ is strictly decreasing at $\beta=0$, which is equivalent to saying that $\dxp(\beta) < 1=\dxp(0)$ for all $\beta>0$. With the {\sl Harris coupling} (cf.  \cite{heydenreich2017progress}) one can see that the function $\dxp(\beta)$ is non-increasing. For this coupling, let $\left(U_e\right)_{e\in E}$ be a collection of i.i.d. random variables that are uniformly distributed on $\left[0,1\right]$ and then set $\omega(e) \coloneqq \mathbbm{1}_{\{U_e \leq p(\beta,e)\}}$. The environment $\omega\in \{0,1\}^E$ is distributed according to the law of $\p_\beta$ and for $\beta \leq \beta^\prime$ one has $\omega(e) = \mathbbm{1}_{\{U_e \leq p(\beta,e)\}} \leq \mathbbm{1}_{\{U_e \leq p(\beta^\prime,e)\}} \eqqcolon \omega^\prime(e)$. So in particular the environment defined by $\omega^\prime$ contains all edges that are open in the environment defined by $\omega$, and thus $D\left(u,v;\omega^\prime\right) \leq D\left( u,v ; \omega \right)$ for all $u,v\in \Z^d$. Taking expectations on both sides of this inequality and letting $\|u-v\|\to\infty$ already shows that $\dxp(\cdot)$ is non-increasing. 

Before going into the details of the proof of the strict monotonicity, we discuss the main idea.  We know that the chemical distance between $\mz$ and $(2^n-1)\mo$ is of order $2^{n\dxp(\beta)}$, as proven in \cite{baeumler2022distances}, and thus $\dxp(\beta) = \lim_{n\to \infty} \frac{\log\left(\E_\beta\left[D_{V_\mz^{2^n}}(\mz, (2^n-1)\mo)\right]\right)}{\log(2^n)}$. For fixed $n \in \N$, one can calculate the derivative of $\frac{\log\left(\E_\beta\left[D_{V_\mz^{2^n}}(\mz,(2^n-1)\mo)\right]\right)}{\log(2^n)}$ with Lemma \ref{lem:russoexp}. Doing this, one gets that
\begin{align}\label{eq:derivative2}
	\frac{\md }{\md \beta} & \frac{\log\left(\E_\beta\left[D_{V_\mz^{2^n}}(\mz,(2^n-1)\mo)\right]\right)}{\log(2^n)}  
	\notag\\
	\notag & = \frac{1}{\log(2^n)\E_\beta\left[D_{V_\mz^{2^n}}(\mz,(2^n-1)\mo)\right]} \frac{\md }{\md \beta} \E_\beta\left[D_{V_\mz^{2^n}}(\mz,(2^n-1)\mo)\right]\\
	& = \frac{\sum_{e \in E} p^\prime(\beta,e) \E_\beta  \left[ D_{V_\mz^{2^n}}(\mz,(2^n-1)\mo;\omega^{e+}) -  D_{V_\mz^{2^n}}(\mz,(2^n-1)\mo;\omega^{e-}) \right]}{\log(2^n)\E_\beta\left[D_{V_\mz^{2^n}}(\mz,(2^n-1)\mo)\right]} ,
\end{align}
where we define $E=\left\{\{u,v\}: u,v\in V_\mz^{2^n}, \|u-v\|_\infty \geq 2\right\}$ to be the set of edges with both endpoints in $V_\mz^{2^n}$ whose length is at least $2$. Our goal is to show that for each $\beta>0$, there exists a constant $c(\beta) < 0$ such that 
\begin{equation}\label{eq:condition derivative}
	\frac{\md }{\md \beta} \frac{\log\left(\E_\beta\left[D_{V_\mz^{2^n}}(\mz,(2^n-1)\mo)\right]\right)}{\log(2^n)}  < c(\beta)
\end{equation} uniformly over $n$. For this, it suffices to consider $n$ large enough, as the bound clearly holds for small $n$. Provided that \eqref{eq:condition derivative} holds, we get that
\begin{align*}
	\dxp(\beta+\eps)-\dxp(\beta) & = \lim_{n\to \infty} \left\{ \frac{\log\left(\E_\beta\left[D_{V_\mz^{2^n}}(\mz,(2^n-1)\mo)\right]\right)}{\log(2^n)} - 
	\frac{\log\left(\E_{\beta + \eps}\left[D_{V_\mz^{2^n}}(\mz,(2^n-1)\mo)\right]\right)}{\log(2^n)} \right\} \\
	& = \lim_{n\to \infty} \int_{\beta}^{\beta + \eps } \frac{\md }{\md \alpha} \frac{\log\left(\E_\alpha\left[D_{V_\mz^{2^n}}(\mz,(2^n-1)\mo)\right]\right)}{\log(2^n)} \md \alpha < 0 \ 
\end{align*}
as we will show in the end of section \ref{subsec:monotonproof} in detail.
This implies strict monotonicity of the function $\beta \mapsto \dxp(\beta)$.\\

In order to show $\frac{\md }{\md \beta} \frac{\log\left(\E_\beta\left[D_{V_\mz^{2^n}}(\mz,(2^n-1)\mo)\right]\right)}{\log(2^n)} < c(\beta)$, we divide the graph into several levels according to a dyadic decomposition. The $i$-th level consists of all edges $e\in E$ for which $c2^{i} < |e| \leq C 2^i$ for some constants $0<c<C<\infty$. Note that the different levels are not necessarily disjoint, but the number of levels to which an edge can belong is bounded. For an edge $e \in E$, we define
\begin{align*}
	\Delta D_{V_\mz^{2^n}} (e) \coloneqq  D_{V_\mz^{2^n}}(\mz,(2^n-1)\mo;\omega) - D_{V_\mz^{2^n}}(\mz,(2^n-1)\mo;\omega^{e+})
\end{align*}
as the difference of the distance between $\mz$ and $(2^n-1)\mo$ when the edge $e$ is inserted. As the environment $\omega^{e-}$ is contained in $\omega$, we observe that
\begin{align*}
	\sum_{e \in E} p^\prime(\beta,e) \E_\beta  \left[ D_{V_\mz^{2^n}}(\mz,(2^n-1)\mo;\omega^{e+}) -  D_{V_\mz^{2^n}}(\mz,(2^n-1)\mo;\omega^{e-}) \right]
	\leq -
	\sum_{e \in E} p^\prime(\beta,e) \Delta D_{V_\mz^{2^n}} (e) .
\end{align*}
Inserting this inequality into equation \eqref{eq:derivative2}, we see that
\begin{align}\label{eq:derivative3}
	\frac{\md }{\md \beta} & \frac{\log\left(\E_\beta\left[D_{V_\mz^{2^n}}(\mz,(2^n-1)\mo)\right]\right)}{\log(2^n)}
	\leq - \ 
	\frac{\sum_{e \in E} p^\prime(\beta,e) \Delta D_{V_\mz^{2^n}} (e)}{n\log(2)\E_\beta\left[D_{V_\mz^{2^n}}(\mz,(2^n-1)\mo)\right]} .
\end{align}
So in order to show that there exists  $c(\beta) < 0$ such that $\frac{\md }{\md \beta} \frac{\log\left(\E_\beta\left[D_{V_\mz^{2^n}} (\mz,(2^n-1)\mo)\right]\right)}{\log(2^n)}  < c(\beta)$ uniformly over $n$, it suffices to show that
\begin{align}\label{eq:self-simi-strict}
	\sum_{e \in E : c 2^k < |e| \leq C2^k} p^\prime(\beta,e) \E_\beta \left[ \Delta D_{V_\mz^{2^n}} (e) \right] \geq c^\prime(\beta) \E_\beta \left[D_{V_\mz^{2^n}} (\mz,(2^n-1)\mo)\right]
\end{align}
for some $c^\prime(\beta)>0$ and all $k \in \{1, \ldots ,n\}$ with $k,n-k$ large enough. Summing this inequality over the different values of $k\in \{1,\ldots,n\}$ then gives - up to a multiplicative constant - the numerator of the right-hand side of \eqref{eq:derivative3}. 

We first give a heuristic argument why both sides of the inequality in \eqref{eq:self-simi-strict} are of the same order. Suppose that the geodesic between $\mz$ and $(2^n-1)\mo$ enters a box $V_u^{2^k}$ at a vertex $x$ and leaves it at a vertex $y$. Then we might assume that the Euclidean distance between $x$ and $y$ is of order $2^k$, whereas the chemical distance between $x$ and $y$ is of order $2^{k\dxp}$. This is certainly true for typical points $x,y \in V_u^{2^k}$, but similar statements can also be proven for the entry- and exit-points of the geodesic. Then, we put small boxes $A_x, A_y$ around $x$, respectively $y$, such that $\dia(A_x),\dia(A_y) < \tfrac{1}{3} D_{V_u^{2^k}}(x,y)$. The side length of the boxes $A_x, A_y$ will still be of order $c 2^{k}$ for some small $c$. Inserting an edge $\{w,z\}$ with $w \in A_x , z \in A_y$ thus decreases the chemical distance between $\mz$ and $(2^n-1)\mo$ by $\tfrac{1}{3} D_{V_u^{2^k}}(x,y)\approx 2^{k\dxp}$, which implies that $\Delta D_{V_\mz^{2^n}}(\{w,z\}) \approx 2^{k\dxp}$. For each edge $e=\{w,z\}$ we have $p^\prime(\beta,e)\approx \|w-z\|^{-2d} \approx 2^{-2kd}$. However, there are approximately $|A_x|\cdot|A_y| \approx 2^{2kd}$ many such pairs $(w,z)\in A_x\times A_y$. This implies that
\begin{equation*}
	\sum_{w \in A_x , z \in A_y} p^\prime(\beta,\{w,z\}) \Delta D_{V_\mz^{2^n}}(\{w,z\}) \approx 2^{k\dxp}.
\end{equation*}

So far, we only looked at one box $V_u^{2^k}$, but we can do this for all boxes with side length $2^k$ through which the geodesic traverses. By the self-similarity of the model, there need to be around $2^{(n-k)\dxp}$ such boxes. Thus, the sum on the left-hand side of \eqref{eq:self-simi-strict} should be of order $2^{(n-k)\dxp} 2^{k\dxp} = 2^{n\dxp}$, which is also the magnitude of the right-hand side of \eqref{eq:self-simi-strict}.

\subsection{Block contraction and the graph $G^\prime$}\label{sec:contraction}

In this section, we introduce a technique and a notation that is heavily used in sections \ref{section:monotone} and \ref{sec:continuity}.
Consider long-range percolation on $\Z^d$.
We split the long-range percolation graph into blocks of the form $V_v^n$, where $v\in \Z^d$. For each $v\in \Z^d$, we contract the block $V_v^n \subset \Z^d$ into one vertex, named $r(v)$. If there are multiple edges between $r(u)$ and $r(v)$, we collapse these multiple edges into one edge. We call the graph that results from contracting all these blocks $G^\prime = (V^\prime, E^\prime)$. It follows directly from the scale invariance that the random graph $G^\prime=(V^\prime,E^\prime)$ has the same distribution as long-range percolation on $\Z^d$, because
\begin{equation*}
	\p_\beta\left(r(u)\sim r(v)\right) = \p_{\beta}\left(V_u^n \sim V_v^n\right) \overset{\eqref{scale invariance}}{=} \p_\beta(u\sim v) = 1-e^{-\int_{u+\cC} \int_{v+\cC} \frac{\beta}{ \|t-s\|^{2d}} \md s \md t},
\end{equation*}
and the connections between different boxes, i.e., the events of the form $\left\{V_u^n \sim V_v^n\right\}_{\{u,v\}\subset \Z^d}$ are independent.
In analogy to the $\infty$-metric, we also write $\|r(u)-r(v)\|_\infty=\|u-v\|_\infty$ and we write $D_{G^\prime}(r(u),r(v))$ for the graph distance between $r(u)$ and $r(v)$ in this random graph.
For a vertex $r(v)\in V^\prime$, we define its neighborhood $\cN\left(r(v)\right)$ by
\begin{align*}
	\cN\left(r(v)\right) = \left\{r(u) \in V^\prime : \|v-u\|_\infty \leq 1 \right\} \text ,
\end{align*}
and we define the neighborhood-degree of $r(v)$ by 
\begin{align}\label{eq:nbh degree}
	\deg^{\cN}(r(v)) = \sum_{r(u) \in \cN\left(r(v)\right)} \deg (r(u)) \text .
\end{align}
We also define these quantities in the same way when we start with long-range percolation on the graph $V_\mz^{mn}$, and contract boxes of the form $V_v^n$ for all $v \in V_\mz^m$. In this situation, the graph $G^\prime$ has the exact same distribution as long-range percolation with measure $\p_\beta$ on $V_\mz^m$.

\subsection{Distances of certain points}
\noindent
For the proof of Theorem \ref{theo:strictmonotonicity}, we import a few results from \cite{baeumler2022distances}:
Fix the three blocks $V_u^n, V_w^n$ and $V_{\mz}^n$ with $\|u\|_\infty \geq 2$. The next lemma deals with the distance in the infinity metric between points $x,y \in V_{\mz}^n$ with $x \sim V_u^n, y\sim V_w^n$ and is proven in \cite[Lemma 5.1]{baeumler2022distances}.
\begin{lemma}\label{lem:linearspacingcompanion}
	For all $\frac{1}{n} <\eps \leq \frac{1}{4}$ and $u,w \in \Z^d \setminus \{\mz\}$ with $\|u\|_\infty \geq 2$ and $u\neq w$ one has
	\begin{align}\label{eq:graphspacingcompanion}
		\p_\beta \big( \exists x,y \in V_{\mz}^n : \|x-y\|_\infty \leq \eps n, x \sim V_u^n, y \sim V_w^n \ \big|  \ V_\mz^n \sim V_u^n, V_\mz^n \sim V_w^n \big) \leq C_d^\prime \eps^{1/2} \lceil\beta \rceil^2
	\end{align}
	where $C_d^\prime$ is a constant that depends only on the dimension $d$.
\end{lemma}
Lemma \ref{lem:linearspacingcompanion} tells us that for a block $V_\mz^n$, the vertices $x,y \in V_\mz^n$ that are connected to different boxes $x \sim V_u^n, y\sim V_w^n$ are typically far apart in terms of Euclidean distance, whenever $\|u\|_\infty \geq 2$. The next lemma shows that such points $x,y$ are typically also not too close in terms of the graph distance inside $V_\mz^n$. It was already proven in \cite[Lemma 5.2]{baeumler2022distances}.
\begin{lemma}\label{lem:graphspacing}
	For all dimensions $d$ and all $\beta\geq 0$, there exists a function $g_1: (0,1) \to \left[0,1\right]$ with $\lim_{\eps \to 0} g_1(\eps) = 1$ such that for all $u,w \in \Z^d \setminus \{\mz\}$ with $\|u\|_\infty \geq 2$ and all large enough $n\geq n(\eps)$
	\begin{align*}
		\p_\beta \big(  D_{V_{\mz}^n}(x,y) > \eps \Lambda(n,\beta) \text{ for all } x,y \in V_{\mz}^n \text{ with } x\sim V_u^n, y\sim V_w^n \ \big| \ V_u^n \sim V_\mz^n \sim V_w^n \big) \geq g_1(\eps) \text .
	\end{align*}
\end{lemma}
\noindent
By the result of \cite{baeumler2022distances} we know that 
\begin{align}\label{point to box}
	D\left(\mz,B_n(\mz)^c \right) \approx_P \E_\beta \left[D(\mz,n\mo)\right] 
\end{align}
under the measure $\p_\beta$, see \cite[Lemma 4.10]{baeumler2022distances}. As we also have $\E_\beta \left[D(\mz,n\mo)\right] \approx_P D(\mz, n \mo)$ (by \eqref{eq:baeumler1}), we see that the point-to-point distance $D(\mz,n\mo)$ and the point-to-set distance $D\left(\mz,B_n(\mz)^c \right)$ are typically of the same order. One can ask whether the same statement is true for the distance between two sets that are separated by a Euclidean distance of $n$, for example $D\left(B_n(\mz),B_{2n}(\mz)^c \right)$. However, the same statement is not true for this case, as there is a uniform (in $n$) positive probability of a direct edge between the sets $B_n(\mz)$ and $B_{2n}(\mz)^c$. But if we condition on the event that there is no direct edge, then we can get such a result, as proven in \cite[Lemma 4.11 and Corollary 4.12]{baeumler2022distances}.

\begin{lemma}\label{lem:all quantiles of boxtobox}
	Let $\mathcal{L}$ be the event that there is no direct edge between $B_n(\mz)$ and $B_{2n}(\mz)^c$.	For all $\beta \geq 0$ and all $\eps>0$, there exist $0<c<C<\infty$ such that
	\begin{align}\label{eq:all quantiles of boxtobox}
		\p_\beta \left( c\Lambda(n,\beta) \leq D\left(B_n(\mz),B_{2n}(\mz)^c \right) \leq C\Lambda(n,\beta) \ \big| \ \mathcal{L} \right) > 1-\eps
	\end{align}
	for all $n \in \N$. Let $\mathcal{L}^\prime$ be the event that there is no direct edge between $V_\mz^n$ and $\bigcup_{u \in \Z^d : \|u\|_\infty \geq 2} V_u^n$.	For all $\beta \geq 0$ and all $\eps>0$, there exist $0<c<C<\infty$ such that
	\begin{align}\label{eq:all quantiles of boxtobox2}
		\p_\beta \left( c\Lambda(n,\beta) \leq D\left( V_\mz^n , \bigcup_{u \in \Z^d : \|u\|_\infty \geq 2} V_u^n \right) \leq C\Lambda(n,\beta) \ \Big| \ \mathcal{L}^\prime \right) > 1-\eps
	\end{align}
	for all $n \in \N$. So in particular there exists a function $g_2:(0,1) \to \left[0,1\right]$ with $\lim_{\eps \to 0}g_2(\eps)=1$ such that
	\begin{align*}
		\p_\beta \left( \eps \Lambda(n,\beta) < D\left( V_\mz^n , \bigcup_{u \in \Z^d : \|u\|_\infty \geq 2} V_u^n \right)  \ \Big| \ \mathcal{L}^\prime \right) \geq g_2(\eps)
	\end{align*}
	for all large enough $n\geq n(\eps)$.
\end{lemma}

\subsection{The geometry inside blocks}\label{subsec:insidegeometry}

In \cite[Theorem 6.1]{baeumler2022distances} it is proven that $\frac{\dia \left(V_0^n\right)}{n^{\theta(\beta)}}$ has exponential moments uniformly in $n$, i.e., that
\begin{equation}
	\sup_{n\in \N} \E_\beta \left[ \exp \left( \frac{\dia \left(V_0^n\right)}{n^{\theta(\beta)}} \right) \right] < \infty.
\end{equation}
From this, we can directly deduce that for all $k \in \N$ and $\beta >0$ there exists a constant $C_k$ such that for all $n\in \N$
\begin{align}\label{eq:dia_momentbound}
	\E_\beta \left[\dia\left(V_\mz^n\right)^k\right] \leq C_k n^{k\dxp(\beta)}.
\end{align}
Let $\delta \in \left(0,\frac{1}{3}\right)$. We define a family of sets $\mathcal{CO}_n^\delta \subset V_\mz^n$ with the following two properties: 
\begin{itemize}
	\item $\bigcup_{x \in \mathcal{CO}_n^\delta} B_{\delta n}(x) = V_\mz^n $, and 
	\item $\left|\mathcal{CO}_n^\delta \right| \leq C_\mathcal{CO} \delta^{-d} $ for all $\delta\in \left(0,\frac{1}{3}\right)$ and $n\in \N$,
\end{itemize}
where $C_\mathcal{CO}$ is a constant that depends only on the dimension $d$, but not on $\delta$ or $n$. The abbreviation $\mathcal{CO}$ stands for cover. Such a cover can be constructed by choosing the points in $\mathcal{CO}_n^\delta$ at a distance of approximately $\max(\delta n, 1)$. The next lemma deals with the local structure around the points $x \in \mathcal{CO}_n^\delta$. In particular, we prove that the diameters of balls around such points $x$ aref not too large. This already follows by \eqref{eq:dia_momentbound} and a union bound over all points $x\in \mathcal{CO}_n^\delta$.

\begin{lemma}\label{lem:epsnearness}
	For $\eps \in \left(0,\frac{1}{3}\right)$, let $\mathcal{DL}(\eps)$ be the event
	\begin{align*}
		\mathcal{DL}(\eps) = \bigcap_{x \in \mathcal{CO}_n^{\eps^2}} \left\{ \dia \left( B_{\eps^2 n}(x) \right) <  \frac{\left(\eps^{1.5}n\right)^{\dxp}}{3} \right\} \text .
	\end{align*}
	Then there exists a function $h_1:\left(0,\frac{1}{3}\right) \to \left[0,1\right]$ with $\lim_{\eps \to 0} h_1(\eps) = 1$ such that
	\begin{align*}
		\p_\beta \left( \mathcal{DL}(\eps) \right) \geq h_1(\eps)
	\end{align*}
	for all $n \geq n(\eps)$ large enough. If the event $\mathcal{DL}(\eps)$ holds, we say that $V_\mz^n$ is $\eps$-{\sl near}.
\end{lemma}

\begin{proof}
	By a union bound we have that 
	\begin{align}\label{eq:bound1}
		\p_\beta \left( \mathcal{DL}(\eps)^c \right)& 
		\leq 
		\sum_{x \in \mathcal{CO}_n^{\eps^2}} \p_\beta \left(  \dia \left( B_{\eps^2 n}(x) \right) \geq  \frac{\left(\eps^{1.5}n\right)^{\dxp}}{3} \right) \notag \\
		&
		\leq
		C_{\mathcal{CO}}\eps^{-2d} \p_\beta \left(  \dia \left( B_{\eps^2 n}(\mz) \right) \geq \frac{\left(\eps^{1.5}n\right)^{\dxp}}{3} \right) \text .
	\end{align}
	Using Markov's inequality we get that for any $k \in \N$ and $n\geq \eps^{-2}$
	\begin{align*}
		&\p_\beta \left(  \dia \left( B_{\eps^2 n}(\mz) \right) \geq \frac{\left(\eps^{1.5}n\right)^{\dxp}}{3} \right) 
		= 
		\p_\beta \left(  \dia \left( B_{\eps^2 n}(\mz) \right)^k \geq \left(\frac{\left(\eps^{1.5}n\right)^{\dxp}}{3}\right)^k \right)\\
		&
		\leq 
		\E_\beta \left[ \dia \left( B_{\eps^2 n}(\mz) \right)^k \right]  \left(\frac{\left(\eps^{1.5}n\right)^{\dxp}}{3}\right)^{-k}
		\overset{\eqref{eq:dia_momentbound}}{\leq} C_k (2\eps^2n+1)^{k\dxp} \left(\frac{\left(\eps^{1.5}n\right)^{\dxp}}{3}\right)^{-k}
		\leq C^\prime \eps^{0.5k\dxp}
	\end{align*}
	for some constant $C^\prime=C^\prime(k)<\infty$. So using $k =  6d \lceil \theta^{-1} \rceil$ and inserting this into \eqref{eq:bound1} we get that
	\begin{align*}
		\p_\beta \left( \mathcal{DL}(\eps)^c \right) \leq
		C_{\mathcal{CO}}\eps^{-2d} C^{\prime}
		\eps^{0.5\cdot 6d \lceil \theta^{-1} \rceil \dxp } \leq \tilde{C} \eps^{d}
	\end{align*}
	for some constant $\tilde{C}<\infty$, which finishes the proof.	
\end{proof}

The next lemma concerns the indirect distance between two sets, conditioned on the graph $G^\prime$, and is identical to \cite[Lemma 5.4]{baeumler2022distances}. For this, remember that we defined $D^\star(A,B)$ as the length of the shortest path between $A$ and $B$ that does not use a direct edge from $A$ to $B$. Lemma \ref{lem:all quantiles of boxtobox} above implies that $D^\star \left( V_\mz^{n} , \bigcup_{u \in \Z^d : \| u \|_\infty \geq 2} V_u^{n} \right)$ is of the same order as $\Lambda(n,\beta)$. We want to have a similar statement conditioned on the graph $G^\prime$. However, this conditional bound involves the neighborhood-degree $\deg^\cN(r(\mz))$ defined in \eqref{eq:nbh degree} as one of its terms. This also makes intuitive sense. If the set $\bigcup_{u\in \Z^d : \|u\|_\infty =1} V_u^n$ is connected to a lot of boxes $V_w^n \subset \bigcup_{u\in \Z^d : \|u\|_\infty \geq 2} V_u^n$, then there are more edges with an endpoint in $\bigcup_{u\in \Z^d : \|u\|_\infty =1} V_u^n$ through which a path starting at $V_\mz^n$ can leave the set $\bigcup_{u\in \Z^d : \|u\|_\infty =1} V_u^n$. Also note that the neighborhood-degree $\deg^{\cN}(r(v))$ is measurable with respect to the $\sigma$-algebra generated by $G^\prime = (V^\prime, E^\prime)$.

\begin{lemma}(\cite[Lemma 5.4]{baeumler2022distances})\label{lem:goodeventW}
	Let $\mathcal{W}(\eps)$ be the event
	\begin{align*}
		\mathcal{W}(\eps) \coloneqq \left\{ D^\star \left( V_v^{n} , \bigcup_{u \in \Z^d : \| u - v\|_\infty \geq 2} V_u^{n} \right) > \eps \Lambda(n,\beta) \right\} \text .
	\end{align*}
	For all large enough $n\geq n(\eps)$ one has
	\begin{equation*}
		\p_{\beta} \left( \mathcal{W}(\eps)^c \ | \ G^\prime \right) \leq 3^d \deg^{\cN} \left(r(v)\right)  \left(1-g_1(\eps)\right) + \left(1-g_2(\eps)\right),
	\end{equation*}
	where $g_1$ and $g_2$ are the functions defined in Lemma \ref{lem:graphspacing}, respectively Lemma \ref{lem:all quantiles of boxtobox}.
\end{lemma}

Roughly speaking, the term $\left(1-g_2(\eps)\right)$ corresponds to the probability of the existence of a path of length at most $\eps \Lambda(n,\beta)$ consisting only out of short edges from $V_v^{n} $ to $ \bigcup_{u \in \Z^d : \| u - v\|_\infty \geq 2} V_u^{n}$ and the term $3^d \deg^{\cN} \left(r(v)\right)  \left(1-g_1(\eps)\right)$ corresponds to the probability that there exists a path of length at most $\eps \Lambda(n,\beta)$ from $V_v^{n} $ to $ \bigcup_{u \in \Z^d : \| u - v\|_\infty \geq 2} V_u^{n}$ that uses a long edge between $ \bigcup_{u \in \Z^d : \| u - v\|_\infty =1} V_u^{n}$ and $ \bigcup_{u \in \Z^d : \| u - v\|_\infty \geq 2} V_u^{n}$.\\

Furthermore, before going to the proof of Theorem \ref{theo:strictmonotonicity}, we introduce the following lemma that considers the structure of connected sets in the long-range percolation graph. For a finite set $Z\subset \Z^d$, we define its average degree by $\overline{\deg}(Z)= \frac{1}{|Z|} \sum_{v\in Z} \deg(z)$. It is easy to show that for a fixed set $Z\subset \Z^d$, its average degree is of order $\E_\beta \left[\deg(\mz)\right]$, and for every $\alpha>1$ it is exponentially unlikely in $|Z|$ that the average degree of $Z$ is larger than $\alpha \E_\beta \left[\deg(\mz)\right]$. The next lemma shows that a similar statement also holds when not considering a fixed set $Z$, but all connected sets of a certain size containing the origin.
It is proven in \cite[Lemma 3.2 and (24)]{baeumler2022distances}.

\begin{lemma}\label{lem:connnectedsetsinLRP}
	Let $\mathcal{CS}_k= \mathcal{CS}_k\left(\Z^d\right)$ be all connected subsets of the long-range percolation graph with vertex set $\Z^d$ of size $k$ that contain the origin $\mz$.
	Define $\mu_\beta \coloneqq \E_\beta \left[\deg(\mz)\right]$. Then for all $\beta>0$ 
	\begin{align*}
		\p_\beta \left(\exists Z \in \mathcal{CS}_k (\Z^d) : \overline{\deg}(Z) \geq 20 \mu_\beta\right) \leq e^{-4k\mu_\beta}
	\end{align*}
	and
	\begin{align*}
		\E_\beta \left[ \left|\mathcal{CS}_k\left(\Z^d\right)\right| \right] \leq 4^k \mu_\beta^k \text.
	\end{align*}
\end{lemma}
The statement above also holds when one considers the graph $V_\mz^m$ instead of $\Z^d$. Indeed, restricting to the set $V_\mz^m$ only restricts the possible space of connected sets $Z$ and also lowers the average degree of such sets.\\

When contracting the sets $(V_v^n)_{v\in \Z^d}$, respectively the sets $(V_v^n)_{v\in V_\mz^m}$, and constructing the graph $G^\prime=(V^\prime,E^\prime)$ as described above, we also use the same notation and define $\mathcal{CS}_k(G^\prime)$ as the set containing all connected subsets $Z\subset V^\prime$ with $|Z|=k$ and $r(\mz)\in Z$. Further, for $t\in \R_{\geq 0}$, we also define
\begin{equation*}
	\mathcal{CS}_{\geq t}\left(G^\prime\right) = \bigcup_{i\geq t} \mathcal{CS}_{i}\left(G^\prime\right)
\end{equation*}
as the set of all connected sets of $V^\prime$ with size at least $t$ containing the origin $r(\mz)$. As the long-range percolation graphs with vertex set $\Z^d$, respectively $V_\mz^m$,  and the graph $G^\prime=(V^\prime, E^\prime)$ have the same distribution under the measure $\p_\beta$, the results of Lemma \ref{lem:connnectedsetsinLRP} also hold for $G^\prime$.

\begin{corollary}\label{coro:connnectedsetsinLRP}
	Consider long-range percolation with measure $\p_\beta$ on $\Z^d$, respectively $V_\mz^{mn}$.
	Let $G^\prime=(V^\prime,E^\prime)$ be the graph that results from contracting the sets $(V_v^n)_{v\in \Z^d}$, respectively the sets  $(V_v^n)_{v\in V_\mz^m}$.
	Then for all $\beta>0$ 
	\begin{align}\label{coroeq48}
		\p_\beta \left(\exists Z \in \mathcal{CS}_k(G^\prime) : \overline{\deg}(Z) \geq 20 \mu_\beta\right) \leq e^{-4k\mu_\beta}
	\end{align}
	and
	\begin{align}\label{coroeq482}
		\E_\beta \left[ \left|\mathcal{CS}_k\left(G^\prime\right)\right| \right] \leq 4^k \mu_\beta^k \text.
	\end{align}
\end{corollary}

\subsection{The proof of Theorem \ref{theo:strictmonotonicity}}\label{subsec:monotonproof}

With the knowledge from the previous subsections, we are now ready to go to the proof of Theorem \ref{theo:strictmonotonicity}. The proof consists of five main parts: In part (A), we describe how to extract certain subpaths from a path $P$ in $G^\prime$. In part (B), we define a notion of {\sl influential} paths in a renormalized graph $G^\prime$. Then, in part (C), we show that all paths of length $t$ are influential with probability approaching $1$ exponentially fast, as $t$ increases. In part (D), we show how this implies the uniform upper bound on the derivative of the expected distance (inequality \eqref{eq:condition derivative}), and in part (E) we show that this implies the strict monotonicity of the distance exponent.

\begin{proof}[Proof of Theorem \ref{theo:strictmonotonicity}] \textbf{Part (A)}.
	Consider the random graph with vertex set $V_\mz^{2^n}$. For $k\leq n$, define the graph $G^\prime=(V^\prime,E^\prime)$ by contracting all blocks of the form $V_u^{2^k}$, where $u\in V_\mz^{2^{n-k}}$. We define $r(u)\in V^\prime$ as the vertex that results from contracting $V_u^{2^k}$. By the scale invariance \eqref{scale invariance}, the graph $G^\prime$ has the same distribution as long-range percolation with measure $\p_\beta$ on $V_\mz^{2^{n-k}}$. In analogy to $\Z^d$, we call the vertex $r(\mz)$ the origin of $G^\prime$. Now consider a self-avoiding path $P=\left(r(u_0),r(u_1),\ldots, r(u_t)\right) \subset G^\prime$, where $u_0=\mz$, and $t \gg 10^{6d} \E_\beta \left[\deg(\mz)\right]$. 
	We divide the path into blocks of length $K \coloneqq 3^d+1$: For $j \leq \lfloor \frac{t}{K} \rfloor -1$, we define $R_j = \left( r(u_{jK}) ,\ldots, r(u_{jK+3^d})\right)$. Note that the subpath $R_{j+1}$ comes directly after $R_j$ in the original path $P$, but $R_j$ and $R_{j+1}$ do not overlap. So in particular, for $j^\prime \leq \lfloor \frac{t}{K} \rfloor -1$, one can concatenate the subpaths $R_0,R_1,\ldots,R_{j^\prime}$ and obtains a path in $G^\prime$ that contains the origin $r(\mz)$. \\
	
	For each such $j \in \{0,\ldots, \lfloor \tfrac{t}{K} \rfloor -1\}$ and $R_j$, we define a set $\widetilde{R}_j$ as follows: If there exist $i \in \{jK,\ldots, jK+3^d-1\}$ with $\|u_i-u_{i+1}\|_\infty \geq 2$, we simply set $\widetilde{R}_j = R_j$. If $\|r(u_i)-r(u_{i+1})\|_\infty = 1$ for all $i \in \{jK , \ldots, jK +3^d -1 \}$, then we set $\widetilde{R}_j = R_j \cup \cN(r(u_{jK}))$.
	The set $\bigcup_{j=0}^{\lfloor \frac{t}{K} \rfloor -1 } \widetilde{R}_j$ is a connected set and its cardinality is bounded from below by
	\begin{equation}\label{eq:t halbe}
		\left|\bigcup_{j=0}^{\lfloor \frac{t}{K} \rfloor -1 } \widetilde{R}_j\right| \geq \left|\bigcup_{j=0}^{\lfloor \frac{t}{K} \rfloor -1 } R_j\right| \geq K \Big\lfloor \frac{t}{K} \Big\rfloor \geq \frac{t}{2} \text ,
	\end{equation}
	where we used $t\gg 10^{6d} \E_\beta \left[\deg(\mz)\right]$ for the last inequality. Further, we can bound the size of the set $\bigcup_{j=0}^{\lfloor \frac{t}{K} \rfloor -1 } \widetilde{R}_j$ from above by
	\begin{equation}\label{eq:t doppelte}
		\left|\bigcup_{j=0}^{\lfloor \frac{t}{K} \rfloor -1 } \widetilde{R}_j\right|
		\leq
		\left|\bigcup_{j=0}^{\lfloor \frac{t}{K} \rfloor -1 }\left( {R}_j \cup \cN(r(u_{jK})) \right)\right|
		\leq
		\sum_{j=0}^{\lfloor \frac{t}{K} \rfloor -1 }
		(K+3^d)
		\leq
		\Big \lfloor \frac{t}{K} \Big \rfloor 2K
		\leq 2 t \text .
	\end{equation}
	From now on, we will always work on the event
	\begin{align}\label{eq:Ht assumption}
		\mathcal{H}_t \coloneqq \left\{ \overline{\deg}(Z) < 20 \mu_\beta \text{ for all } Z \in \mathcal{CS}_{\geq t/2}\left(G^\prime\right) \right\} .
	\end{align}
	The event $\mathcal{H}_t$ is very likely for large $t$, as
	\begin{align}\label{eq:Hbound}
		\p_\beta \left(\mathcal{H}_t^c \right) \overset{\eqref{coroeq48}}{\leq} 
		\sum_{t^\prime = \lceil t/2 \rceil}^\infty e^{-4t^\prime \mu_\beta}
		\leq 2 e^{-2t \mu_\beta} \leq 2^{-t}
	\end{align}
	by Corollary \ref{coro:connnectedsetsinLRP}, where the last and second to last inequalities hold since $\mu_\beta = \E_\beta \left[\deg(\mz)\right]\geq 2$ and $t \gg 10^{6d} \E_\beta \left[\deg(\mz)\right]$. We define the degree of $\widetilde{R}_j$ by 
	\begin{align*}
		\deg\left(\widetilde{R}_j\right) = \sum_{r(u) \in \widetilde{R}_j} \deg(r(u)) \text .
	\end{align*}
	The sets $(\widetilde{R}_j)_{j\in \{0,\ldots,\lfloor\tfrac{t}{K}\rfloor -1\}}$ satisfy $\widetilde{R}_j \subseteq R_j \cup \cN(r(u_{jK}))$. A fixed vertex $r(u) \in V^\prime$ might be included in more than one of the sets $\widetilde{R}_j$, so it does not necessarily hold that
	\begin{equation*}
		\sum_{j=0}^{\lfloor \frac{t}{K} \rfloor -1 } \deg\left(\widetilde{R}_j\right) = \sum_{r(u)\in  \bigcup_{j=0}^{\lfloor \frac{t}{K} \rfloor -1 } \widetilde{R}_j } \deg(r(u)) \text ,
	\end{equation*}
	as some vertices $r(u)$ can  be included in both $\widetilde{R}_j$ and $\widetilde{R}_i$ for different values $j\neq i$. If $r(u)\in \widetilde{R}_j$, then $\cN(r(u)) \cap R_j \neq \emptyset$ (where we identify the path-segment $R_j=(r(u_{jK}),\ldots,r(u_{jK+3^d}))$ with the set $\{r(u_{jK}),\ldots,r(u_{jK+3^d})\}$). Since $K=3^d+1$ and $R_j=(r(u_{jK}),\ldots,r(u_{jK+3^d}))$, the different segments $(R_j)_{j\in\{0,\ldots,\lfloor \tfrac{t}{K} \rfloor -1\}}$ are disjoint. So in particular, for each vertex $r(u)$ there can be at most $\left|\cN(r(u))\right|\leq 3^d$ many indices $j$ with $\cN(r(u))\cap R_j \neq \emptyset$. So each vertex $r(u)$ can be included in at most $3^d$ many sets of the form $\widetilde{R}_j$. Thus we have
	\begin{align}\label{drei hoch d}
		&\notag \sum_{j=0}^{\lfloor \frac{t}{K} \rfloor -1 } \deg\left(\widetilde{R}_j\right) 
		=
		\sum_{j=0}^{\lfloor \frac{t}{K} \rfloor -1 } \
		\sum_{r(u) \in \widetilde{R}_j} \deg(r(u))\\
		& \notag
		= 
		\sum_{r(u)\in  \bigcup_{j=0}^{\lfloor \frac{t}{K} \rfloor -1 } \widetilde{R}_j } \deg(r(u)) \left|\left\{j\in \{0,\ldots, \Big\lfloor \frac{t}{K} \Big\rfloor -1 \} : r(u) \in \widetilde{R}_j \right\} \right| \\
		&
		\leq
		3^d
		\sum_{r(u)\in  \bigcup_{j=0}^{\lfloor \frac{t}{K} \rfloor -1 } \widetilde{R}_j } \deg(r(u)) 
		.
	\end{align}
	The set $\bigcup_{j=0}^{\lfloor \frac{t}{K} \rfloor -1 } \widetilde{R}_j$ is a connected set with size at least $\tfrac{t}{2}$ \eqref{eq:t halbe} that contains the origin $r(\mz)$.
	So assuming that the event $\mathcal{H}_t$ defined in \eqref{eq:Ht assumption} holds, we get that
	\begin{align}\label{eq:degree sum bound}
		\notag
		&\sum_{j=0}^{\lfloor \frac{t}{K} \rfloor -1 } \deg\left(\widetilde{R}_j\right) 
		\overset{\eqref{drei hoch d}}{\leq} 3^d \sum_{r(u)\in  \bigcup_{j=0}^{\lfloor \frac{t}{K} \rfloor -1 } \widetilde{R}_j } \deg(r(u)) 
		= 3^d \left| \bigcup_{j=0}^{\lfloor \frac{t}{K} \rfloor -1 } \widetilde{R}_j \right| \overline{\deg}\left( \bigcup_{j=0}^{\lfloor \frac{t}{K} \rfloor -1 } \widetilde{R}_j \right)\\
		&
		\overset{\eqref{eq:Ht assumption}}{\leq}
		3^d\ 20 \mu_\beta \left| \bigcup_{j=0}^{\lfloor \frac{t}{K} \rfloor -1 } \widetilde{R}_j \right|
		\overset{\eqref{eq:t doppelte}}{\leq}
		3^d\ 40 \mu_\beta t .
	\end{align} 
	Thus there can be at most $0.5 \lfloor \tfrac{t}{K} \rfloor$ many indices $j \in \{0,\ldots, \lfloor \tfrac{t}{K} \rfloor -1\}$ with 
	\begin{equation*}
		\deg\left(\widetilde{R}_j\right) \geq 40^{d+1} \mu_\beta .
	\end{equation*}
	Indeed, if there would be more than $0.5 \lfloor \tfrac{t}{K} \rfloor$ many indices $j \in \{0,\ldots, \lfloor \tfrac{t}{K} \rfloor -1\}$ with $\deg\left(\widetilde{R}_j\right) \geq 40^{d+1} \mu_\beta$, then
	\begin{equation*}
		\sum_{j=0}^{\lfloor \frac{t}{K} \rfloor -1 } \deg\left(\widetilde{R}_j\right) > 0.5 \Big\lfloor \frac{t}{K} \Big\rfloor \cdot  40^{d+1} \mu_\beta
		\geq
		\Big\lfloor \frac{t}{3^d+1} \Big\rfloor \cdot  20^{d} 40 \mu_\beta
		\geq
		\frac{t}{4^d} \cdot  20^{d} 40 \mu_\beta
		> 3^d 40 \mu_\beta t
	\end{equation*}
	which is a contradiction to \eqref{eq:degree sum bound}. If there are at most $0.5 \lfloor \tfrac{t}{K} \rfloor$ many indices $j \in \{0,\ldots, \lfloor \tfrac{t}{K} \rfloor -1\}$ with $\deg\left(\widetilde{R}_j\right) \geq 40^{d+1} \mu_\beta$, then there need to be at least $\lfloor \tfrac{t}{K} \rfloor - 0.5 \lfloor \tfrac{t}{K} \rfloor \geq \lceil \tfrac{t}{10^d} \rceil$ many indices $j \in \{0,\ldots, \lfloor \tfrac{t}{K} \rfloor -1\}$ with $\deg\left(\widetilde{R}_j\right) \leq 40^{d+1} \mu_\beta$, i.e.,
	\begin{equation*}
		\left|\left\{j \in \left\{0,\ldots, \Big\lfloor \frac{t}{K} \Big\rfloor -1 \right\} : \deg\left(\widetilde{R}_j\right) \leq 40^{d+1} \mu_\beta \right\}\right| \geq \Big\lceil \frac{t}{10^d} \Big\rceil.
	\end{equation*}
	Say that $\left\{j_1,\ldots, j_{\lceil \frac{t}{10^d} \rceil }\right\} \subset \left\{0,\ldots, \lfloor \frac{t}{K} \rfloor -1 \right\}$ are the first such indices, where $j_1 < j_2 < \ldots < j_{\lceil \frac{t}{10^d} \rceil }$. We define a further subset $\mathcal{IND} = \mathcal{IND}(P) \subset \left\{j_1,\ldots, j_{\lceil \frac{t}{10^d} \rceil }\right\}$ of these indices as follows. Start with $\mathcal{IND}_0 = \emptyset$ and then iteratively define $\left(\mathcal{IND}_{i}\right)_{i=1}^{\lceil \frac{t}{10^d} \rceil }$ by
	\begin{align*}
		\mathcal{IND}_{i} \coloneqq \begin{cases}
			\mathcal{IND}_{i-1} \cup \{j_i\} & \text{if } \widetilde{R}_{j_i} \nsim \bigcup_{\ell \in \mathcal{IND}_{i-1}} \widetilde{R}_{\ell}\\
			\mathcal{IND}_{i-1} & \text{else}
		\end{cases}.
	\end{align*}
	Then we define $\mathcal{IND}(P) = \mathcal{IND}  \coloneqq \mathcal{IND}_{\lceil \frac{t}{10^d} \rceil}$ as the last element in this sequence.
	By the construction, there is no edge between $\widetilde{R}_j$ and $ \widetilde{R}_{j^\prime}$ for different $j,j^\prime \in \mathcal{IND}  \coloneqq \mathcal{IND}_{\lceil \frac{t}{10^d} \rceil}$. Further, for every $j_i \in \left\{j_1,\ldots, j_{\lceil \frac{t}{10^d} \rceil }\right\} \setminus \mathcal{IND}$ one has
	\begin{equation*}
		\widetilde{R}_{j_i} \sim \bigcup_{\ell \in \mathcal{IND}_{i-1}} \widetilde{R}_{\ell} \subset \bigcup_{\ell \in \mathcal{IND}} \widetilde{R}_{\ell} ,
	\end{equation*}
	as otherwise $j_i$ would have been included into $\mathcal{IND}$.
	Therefore we have that
	\begin{align}\label{eq:IND union}
		\notag \left\{j_1,\ldots, j_{\big\lceil \frac{t}{10^d} \big\rceil} \right\} & = \mathcal{IND} \cup \bigcup_{\ell \in \mathcal{IND}} \left\{ j \in \left\{j_1,\ldots, j_{\big\lceil \frac{t}{10^d} \big\rceil} \right\} : \widetilde{R}_{j} \sim \widetilde{R}_{\ell} \right\}\\
		&
		=
		\bigcup_{\ell \in \mathcal{IND}} \left(\{\ell\} \cup \left\{ j \in \left\{j_1,\ldots, j_{\big\lceil \frac{t}{10^d} \big\rceil} \right\} : \widetilde{R}_{j} \sim \widetilde{R}_{\ell} \right\}\right)
	\end{align}
	For each $\ell \in \mathcal{IND}$, there are at most $\deg\left(\widetilde{R}_{\ell}\right) \leq 40^{d+1} \mu_\beta$ many vertices $r(u) \in G^\prime \setminus \widetilde{R}_{\ell}$ for which $r(u) \sim \widetilde{R}_{\ell}$. For each such vertex $r(u)$, there can be at most $|\cN(r(u))|\leq 3^d$ many indices $j \in \left\{j_1,\ldots, j_{\big\lceil \frac{t}{10^d} \big\rceil} \right\}$ for which $r(u) \in \widetilde{R}_{j}$. Thus we see that
	\begin{equation*}
		\left| \left\{ j \in \left\{j_1,\ldots, j_{\big\lceil \frac{t}{10^d} \big\rceil} \right\} : \widetilde{R}_{j} \sim \widetilde{R}_{\ell} \right\} \right|
		\leq 3^d \deg\left(\widetilde{R}_{\ell}\right) \leq 3^d 40^{d+1} \mu_\beta \leq 120^{d+1} \mu_\beta 
	\end{equation*}
	for each $\ell \in \mathcal{IND}$. Taking the sizes of the sets in the above equality \eqref{eq:IND union}, we see that
	\begin{align*}
		\Big\lceil \frac{t}{10^d} \Big\rceil & =\left| \left\{j_1,\ldots, j_{\big\lceil \frac{t}{10^d} \big\rceil} \right\}\right| 
		=
		\left|\bigcup_{\ell \in \mathcal{IND}} \left(\{\ell\} \cup \left\{ j \in \left\{j_1,\ldots, j_{\big\lceil \frac{t}{10^d} \big\rceil} \right\} : \widetilde{R}_{j} \sim \widetilde{R}_{\ell} \right\}\right)\right|\\
		&
		\leq
		\sum_{\ell \in \mathcal{IND}}
		\left| \{\ell\} \cup \left\{ j \in \left\{j_1,\ldots, j_{\big\lceil \frac{t}{10^d} \big\rceil} \right\} : \widetilde{R}_{j} \sim \widetilde{R}_{\ell} \right\} \right|
		\leq
		\sum_{\ell \in \mathcal{IND}}
		\left(1+120^{d+1}\mu_\beta\right).
	\end{align*}
	From this, we can directly deduce that
	\begin{equation}\label{eq:ind_bound}
		\left|\mathcal{IND}\right| \geq \frac{1}{120^{d+1}\mu_\beta + 1} \Big\lceil \frac{t}{10^d} \Big\rceil \geq \frac{t}{10^{6d} \mu_\beta } \text .
	\end{equation}
	
	\noindent
	\textbf{Part (B)}. Next, we define what it means for a subpath $R_j = \left(r(u_{jK}) , \ldots, r(u_{jK+3^d})\right)$ to be $\eps$-{separated\sl}, where $\eps\in \left(0,\frac{1}{3}\right)$. Here, we make a distinction whether there exists $i \in \{jK, \ldots, jK +3^d - 1\}$ with $\|u_{i+1}-u_i\|_\infty \geq 2$ or not. The goal is to have a guaranteed lower bound on the minimal length of a path in the original model that goes through the blocks $\left(V_{u_{jK}}^{2^k} , \ldots, V_{u_{jK+3^d}}^{2^k}\right)$.\\
	
	Let $R_j = \left(r(u_{jK}) , \ldots, r(u_{jK+3^d})\right)$ be a subpath of length $K=3^d+1$ in the graph $G^\prime$ with $\deg \left(\widetilde{R}_j\right) \leq 40^{d+1} \mu_\beta$.	
	If there exists $i \in \{jK, \ldots, jK +3^d - 1\}$ with $\|u_{i+1}-u_i\|_\infty \geq 2$, let $\iota$ be the smallest such index.  If $D_{V_{u_\iota}^{2^k}}(x,y) > \eps \Lambda(2^k, \beta)$ for all $x,y \in V_{u_\iota}^{2^k}$ such that $x \sim  V_{u_{\iota+1}}^{2^k} $, $y \sim  V_{w}^{2^k}$, where $w \notin \{u_\iota, u_{\iota+1}\}$, we say that the path $R_j$ and the set $\widetilde{R}_j$ are $\eps$-separated. Conditional on $G^\prime$, the probability that there exist $x,y \in V_{u_\iota}^{2^k}$ such that $x \sim  V_{u_{\iota+1}}^{2^k} $, $y \sim  V_{w}^{2^k}$, where $w \notin \{u_\iota, u_{\iota+1}\}$, and $D_{V_{u_\iota}^{2^k}}(x,y) \leq \eps \Lambda(2^k, \beta)$ is bounded by 
	\begin{align}\label{eq:separation1}
		&
		\notag
		\p_\beta \left( D_{V_{u_\iota}^{2^k}}(x,y) \leq \eps \Lambda(2^k, \beta) \text{ for some } x,y \in V_{u_\iota}^{2^k} \text{ with } x \sim  V_{u_{\iota+1}}^{2^k},y \sim  V_{w}^{2^k}, w \notin \{u_\iota, u_{\iota+1}\} \Big| G^\prime \right)
		\\
		& 
		\leq
		\deg \left(V_{u_\iota}^{2^k}\right)  \left(1-g_1(\eps)\right) \leq \deg \left(\widetilde{R}_j\right)  \left(1-g_1(\eps)\right) 
		\leq 40^{d+1}\mu_\beta \left(1-g_1(\eps)\right)
	\end{align}
	by Lemma \ref{lem:graphspacing}, for $k$ large enough.\\
	
	Now suppose that $\|u_i-u_{i+1}\|_\infty = 1$ for all $i \in \{jK,\ldots, jK+3^d-1\}$. There exists an index $\iota \in \{jK+1,\ldots, jK+3^d\}$ with $\|u_{jK}-u_\iota\|_\infty \geq 2$, as there are only $3^d-1$ many points $w \in \Z^d$ with $\|u_{jK}-w\|_\infty = 1$. When the path exits the cube $V_{u_{jK}}^{2^k}$ for the last time, it goes to $V_{u_{jK+1}}^{2^k}$, so in particular the path does not use a long edge from $V_{u_{jK}}$ to $\bigcup_{w : \|w-u_{jK}\|_\infty \geq 2} V_w^{2^k}$ for the last exit. If the indirect distance between the sets $V_{u_{jK}}^{2^k}$ and the set $\bigcup_{r(w) \in G^\prime: \|w-u_{jK}\|_\infty \geq 2} V_w^{2^k}$ is at least $\eps \Lambda \left(2^k, \beta\right)$, i.e, if
	\begin{align*}
		D^\star_{V_\mz^{2^n}} \left( V_{u_{jK}}^{2^k} , \bigcup_{r(w) \in G^\prime : \|w-u_{jK}\|_\infty \geq 2} V_w^{2^k}  \right) \geq \eps \Lambda \left(2^k, \beta\right) \text ,
	\end{align*}
	we also say that the subpath $R_j$ and the set $\widetilde{R}_j$ are $\eps$-separated. Note that this directly implies that
	\begin{align*}
		D^\star_{V_\mz^{2^n}} \left( V_{u_{jK}}^{2^k} , V_{u_{\iota}}^{2^k} \right) \geq D^\star_{V_\mz^{2^n}} \left( V_{u_{jK}}^{2^k} , \bigcup_{r(w) \in G^\prime : \|w-u_{jK}\|_\infty \geq 2} V_w^{2^k}  \right) \geq \eps \Lambda \left(2^k, \beta\right) \text ,
	\end{align*}
	Whenever $\deg\left(\widetilde{R}_j\right) \leq 40^{d+1} \mu_\beta$, the conditional probability that there is a path of length at most $\eps \Lambda(2^k,\beta)$ that starts at $V_{u_{jK}}^{2^k}$ and goes through $V_{u_{jK+1}}^{2^k}, \ldots , V_{u_{jK+3^d-1}}^{2^k}$ to $V_{u_{jK+3^d}}^{2^k}$ is bounded by
	\begin{align}\label{eq:separation2}
		&\notag \p_{\beta} \left(  D^\star_{V_\mz^{2^n}} \left( V_{u_{jK}}^{2^k} , \bigcup_{r(w) \in G^\prime : \|w-u_{jK}\|_\infty \geq 2} V_w^{2^k}  \right) \leq \eps \Lambda \left(2^k, \beta\right) \Big| G^\prime \right)\\
		& \leq \notag 3^d \deg^\cN(r(u_{jK})) (1-g_1(\eps)) + (1-g_2(\eps))
		\leq
		3^d \deg\left(\widetilde{R}_j\right) (1-g_1(\eps)) + (1-g_2(\eps))\\
		&
		\leq
		3^d 40^{d+1} \mu_\beta (1-g_1(\eps)) + (1-g_2(\eps))
	\end{align}
	by Lemma \ref{lem:goodeventW}. So we see that in both cases, if $\deg \left(\widetilde{R}_j\right) \leq 40^{d+1} \mu_\beta$, the probability that $R_j$, respectively $\widetilde{R}_j$, is $\eps$-separated is at least
	\begin{align*}
		1-\left(3^d 40^{d+1} \mu_\beta (1-g_1(\eps)) + (1-g_2(\eps))\right),
	\end{align*}
	see \eqref{eq:separation1} and \eqref{eq:separation2}. Note that the condition $\deg \left(\widetilde{R}_j\right) \leq 40^{d+1} \mu_\beta$ is satisfied for all $j\in \mathcal{IND}$.
	\\
	
	For $\eps\in \left(0,\frac{1}{3}\right)$, we say that a subpath $R_j$ is {\sl $\eps$-influential }, if $\widetilde{R}_j$ is $\eps$-separated and all boxes $V_{u_{jK}}^{2^k},\ldots$$,V_{u_{jK+3^d}}^{2^k}$ are $\eps^{1/\dxp}$-near (see Lemma \ref{lem:epsnearness} for the definition of $\eps$-near). The events of the form $\left\{V_{u_i}^{2^k} \text{ is } \eps^{1/\theta}\text{-near }\right\}$ are independent of $G^\prime$, as they only depend on edges with both endpoints in $V_{u_i}^{2^k}$, whereas $G^\prime$ only depends on edges with endpoints in different blocks $V_{u}^{2^k}, V_{w}^{2^k},u\neq w$. For a subpath $R_j$ with $j \in \mathcal{IND}$, we can bound the probability that the sequence $R_j$ is not  $\eps$-influential  by
	\begin{align*}
		&\p_\beta \left( R_j\text{ is not } \eps\text{-influential } \Big|  G^\prime \right)\\
		&
		\leq
		\p_\beta \left( R_j\text{ is not } \eps\text{-separated } \Big|  G^\prime \right) + \sum_{i=jK}^{jK+3^d} \p_\beta \left( V_{u_i}^{2^k} \text{ is not } \eps^{1/\theta}\text{-near } \Big|  G^\prime \right)\\
		&
		\leq 3^d 40^{d+1} \mu_\beta (1-g_1(\eps)) + (1-g_2(\eps)) + (3^d+1) \left(1-h_1\left( \eps^{1/\dxp} \right)\right) \text .
	\end{align*}
	Note that it only depends on edges with at least one endpoint inside $\widetilde{R}_i$, whether $R_i$ is $\eps$-influential. For different values of $j_1, \ldots, j_\ell \in \mathcal{IND}$, the sets $\left(\widetilde{R}_{j_i}\right)_{i\in \{1,\ldots,\ell\}}$ are not connected, and thus it is independent whether these sub-paths are $\eps$-influential.\\
	
	\noindent
	\textbf{Part (C).} 
	Define the function $f:\left(0,\frac{1}{3}\right) \to \R_{\geq 0}$ by
	\begin{equation*}
		f(\eps) = 3^d 40^{d+1} \mu_\beta (1-g_1(\eps)) + (1-g_2(\eps)) + (3^d+1) \left(1-h_1\left( \eps^{1/\dxp} \right)\right).
	\end{equation*}
	So in particular the function $f$ is such that $\p_\beta \left( R_j\text{ is not } \eps\text{-influential } | G^\prime\right) \leq f(\eps)$ for all $j\in \mathcal{IND}$.
	Note that $\lim_{\eps \searrow 0} f(\eps) = 0$ and therefore we can take $\eps\in \left(0,\frac{1}{3}\right)$ small enough such that
	\begin{align}\label{eq:eps strict}
		f(\eps)^{\frac{1}{2\cdot 10^{6d} \mu_\beta}} \leq  \frac{1}{100 \mu_\beta^2} \text .
	\end{align}
	Let $\mathcal{INF} = \mathcal{INF}(P)\subset \mathcal{IND}(P)$ be all indices $j \in \mathcal{IND}$, for which $R_j$ is $\eps$-influential. Assuming that the event $\mathcal{H}_t$ defined in \eqref{eq:Ht assumption} holds, we want to get bounds on the cardinality of the set $\mathcal{INF}$ for a fixed path $P\subset G^\prime$ of length $t$. Remember that we have
	\begin{align*}
		\left|\mathcal{IND}\right| \geq \frac{t}{10^{6d}\mu_\beta} \text ,
	\end{align*}
	as shown in \eqref{eq:ind_bound}. For a path $P=\left(r(u_0),r(u_1),\ldots, r(u_t)\right) \subset G^\prime$, conditioned on $G^\prime$, the events of the form $\{R_j \text{ is } \eps\text{-influential}\}$ are independent for different $j\in \mathcal{IND}$, as the event $\{R_j \text{ is } \eps\text{-influential}\}$ only depends on edges with at least one end in $\widetilde{R}_j$ and $\widetilde{R}_j \nsim \widetilde{R}_{j^\prime}$ for different $j,j^\prime \in \mathcal{IND}$. Further, for every $j\in \mathcal{IND}$, the probability that $R_j$ is not $\eps$-influential is at most $\p_\beta  \left( R_j \text{ not $\eps$-influential } |  G^\prime \right) \leq f(\eps)$ by the definition of $f(\eps)$. Thus we can deduce that
	\begin{align}\label{50}
		&\notag  \p_\beta  \left( \left|\mathcal{INF}\right| < \frac{t}{2\cdot 10^{6d}\mu_\beta} \ \Big| \ G^\prime \right)
		\leq
		\p_\beta  \left( \bigcup_{U \subset \mathcal{IND} : |U| \geq  |\mathcal{IND}|/2} \left\{R_j \text{ not $\eps$-influential } \forall j\in U \right\}  \Big| \ G^\prime \right)
		\\
		& \notag
		\leq
		\sum_{U \subset \mathcal{IND} : |U| \geq  |\mathcal{IND}|/2}
		\p_\beta  \left( \left\{R_j \text{ not $\eps$-influential } \forall j\in U \right\}  \Big| \ G^\prime \right)
		\\
		& \notag
		=
		\sum_{U \subset \mathcal{IND} : |U| \geq  |\mathcal{IND}|/2} \ \
		\prod_{j\in U}
		\p_\beta  \left( R_j \text{ not $\eps$-influential } \Big| G^\prime  \right)
		\leq  \sum_{U \subset \mathcal{IND} : |U| \geq  |\mathcal{IND}|/2} \ \
		f(\eps)^{|U|}
		\\
		& 
		\leq  2^{|\mathcal{IND}|} f(\eps)^{|\mathcal{IND}|/2} \leq 2^t f(\eps)^{\frac{t}{2\cdot 10^{6d}\mu_\beta}} \overset{\eqref{eq:eps strict}}{\leq} 2^{t} \left( \frac{1}{100 \mu_\beta^2} \right)^t
		= \frac{1}{50^t \mu_\beta^{2t}} 
	\end{align}
	where used the assumption \eqref{eq:eps strict} on $\eps$ for the last inequality. This bound holds for a fixed path $P\subset G^\prime$ starting at $r(\mz)$, and assuming that the event $\cH_t$ holds. 
	Next, we derive a similar bound that holds for all paths $P$ of length $t$ starting at $r(\mz)$ simultaneously. For this, let $\mathcal{P}_t$ be the set of all paths in $G^\prime$ of length $t$ starting at $r(\mz)$. We call the previously mentioned event $\mathcal{G}_t$, i.e.,
	\begin{align}\label{eq:Gt}
		\mathcal{G}_t = \left\{  \left|\mathcal{INF}(P)\right| \geq \frac{t}{2\cdot 10^{6d}\mu_\beta} \text{ for all } P \in \mathcal{P}_t \right\} \text .
	\end{align}
	Since $G^\prime$ and $V_\mz^{2^{n-k}}$ have the same distribution under the measure $\p_\beta$, a union bound over all possible paths of length $t$ shows that $\E_\beta \left[ \left| \mathcal{P}_t \right| \right]  \leq \mu_\beta^t$. Thus we have, again by a union bound
	\begin{align*}
		\p_\beta \left(\mathcal{G}_t^c \right) 
		& \leq \p_\beta \left(\mathcal{H}_t^c \right)
		+ \p_\beta \left( \left|\mathcal{P}_t\right| > 2^t \mu_\beta^t \right) 
		+ \p_\beta \left(\mathcal{G}_t^c \ \Big|  \mathcal{H}_t ,  \left|\mathcal{P}_t\right| \leq 2^t \mu_\beta^t \right)\\
		&
		\overset{\eqref{50}}{\leq}
		\p_\beta \left(\mathcal{H}_t^c \right)
		+ 
		\p_\beta \left( \left|\mathcal{P}_t\right| > 2^t \mu_\beta^t \right)
		+
		2^t\mu_\beta^t \frac{1}{50^t \mu_\beta^{2t}} \overset{\eqref{eq:Hbound}}{\leq} 2^{-t} + \frac{\E \left[\left|\mathcal{P}_t\right|\right]}{2^t \mu_\beta^t} + \frac{1}{25^t} \leq 3\cdot 2^{-t}.
	\end{align*}
	and with another union bound we get for the event $\mathcal{G}_{\geq t} \coloneqq \bigcap_{t^\prime =t}^\infty \mathcal{G}_{t^\prime}$ that
	\begin{equation}\label{eq:G prob bound}
		\p_\beta \left( \mathcal{G}_{\geq t}^c \right) \leq \sum_{t^\prime=t}^{\infty} \p_\beta \left( \mathcal{G}_{t^\prime}^c \right)
		\leq \sum_{t^\prime=t}^{\infty} 3 \cdot 2^{-t^\prime}
		\leq 6 \cdot 2^{-t}.
	\end{equation}
	So we see that
	\begin{equation*}
		\p_\beta \left(\bigcup_{t^\prime = t}^\infty \left\{\exists P \in \mathcal{P}_{t^\prime} \text{ with } |\mathcal{INF}(P)|< \frac{t^\prime}{2 \cdot 10^{6d} \mu_\beta} \right\} \right) \leq 6 \cdot 2^{-t}.
	\end{equation*}
	
	\noindent
	\textbf{Part (D).}
	Now, let $\hat{P}=(a_0,a_1,\ldots,a_N)$ be a geodesic between $\mz$ and $(2^n -1)\mo$ in the graph with vertex set $V_\mz^{2^n}$. Let $\widetilde{P}$ be the projection of this path onto $G^\prime$, i.e., the path $(r(b_0),\ldots,r(b_N))$ with $r(b_i)$ such that $a_i \in V_{b_i}^{2^k}$ for all $i$. Let $P$ be the loop-erasure of $\widetilde{P}$, i.e., the path $(r(u_0),\ldots,r(u_t))$ defined by $r(u_i)=r(b_{j_i})$, where $j_i$ is defined by $j_0=0$ and, whenever $j_i < N$, we define $j_{i+1}$ by
	\begin{equation*}
		j_{i+1} = \max\{k : r(u_k)=r(u_{j_i})\}+1 .
	\end{equation*}
	
	Whenever the path $P=(r(u_0),\ldots,r(u_t))$ contains an $\eps$-influential subpath $R_j = \left(r(u_{jK}),\ldots, r(u_{jK+3^d})\right) \subset P$, let $\ell = \ell(j) \in \{jK,\ldots, jK+3^d-1\}$ be the first index for which $\|r(u_\ell)-r(u_{\ell+1})\|_\infty \geq 2$ if such an index exists. Respectively, let $\ell = \ell(j) \in \{jK,\ldots, jK+3^d\}$ be the first index for which $\|r(u_\ell)-r(u_{jK})\|_\infty \geq 2$, if there does not exist $\ell \in \{jK,\ldots,jK+3^d-1\}$ with $\|r(u_\ell)-r(u_{\ell+1})\|_\infty \geq 2$. For both cases, we observe that $\left\| r(u_{jK})-r(u_\ell) \right\|_\infty \in \{0,\ldots,3^d\}$.
	
	Whenever the path $\hat{P}$ crosses the set $R_j$ (for $j\geq 1$), it enters $V_{u_{jK}}^{2^k}$ for the first time through some vertex $x_L$ and it leaves $V_{u_{\ell}}^{2^k}$ to $V_{u_{\ell+1}}^{2^k}$ through some vertex $x_R$. As the boxes $V_{u_{jK}}^{2^k}$ and $V_{u_{\ell}}^{2^k}$ are $\eps^{1/\dxp}$-near (see Lemma \ref{lem:epsnearness} for the definition of $\eps$-near), there exist cubes $B_L, B_R$ of side length at least $2\eps^{2/\dxp} 2^k$ such that
	\begin{align*}
		& x_L \in B_L \subset V_{u_{jK}}^{2^k} , x_R \in B_R \subset V_{u_{\ell}}^{2^k}, \text{ and } \\
		&\dia(B_L), \dia(B_R) < \frac{\left(\eps^{\frac{1.5}{\dxp}} 2^k\right)^\dxp}{3} = \frac{\eps^{1.5}}{3} 2^{k\dxp}.
	\end{align*}
	The graph distance between $x_L$ and $x_R$ is at least $\eps \Lambda\left(2^k,\beta\right)$, as we will argue now. If there exists an index $\ell \in \{jK,\ldots, jK+3^d-1\}$ for which $\|r(u_\ell)-r(u_{{\ell+1}})\|_\infty \geq 2$, then we know that the box $V_{u_\ell}^{2^k}$ is $\eps$-separated, as $R_j$ was assumed to be $\eps$-influential. At the first visit of the box $V_{u_{\ell}}^{2^k}$, the geodesic $\hat{P}$ enters the box $V_{u_{\ell}}^{2^k}$ through some point $z\in V_{u_{\ell}}^{2^k}$ with $z \sim V_w^{2^k}$, for some $w \in V_\mz^{2^{n-k}} \setminus\{u_{\ell},u_{\ell+1}\}$. We have $w \neq u_{\ell+1}$, as the loop-erased projection $P$ is self-avoiding.
	As $D_{V_{u_{\ell}}^{2^k}}(z,x_R) \geq \eps \Lambda\left(2^k,\beta\right)$ for all $z \in V_{u_{\ell}}^{2^k}$ with $z \sim V_{w}^{2^k}$ for $w\notin\{u_{\ell},u_{\ell+1}\}$, we automatically get that
	\begin{equation*}
		D_{V_\mz^{2^n}}(x_L,x_R) \geq \inf \left\{D_{V_{u_{\ell}}^{2^k}}(z,x_R) : z \in V_{u_{\ell}}^{2^k}, z \sim V_{w}^{2^k} \text{ for some } w\notin\{u_{\ell},u_{\ell+1}\} \right\}  \geq \eps \Lambda\left(2^k,\beta\right)
	\end{equation*}
	as either $x_L \in V_{u_{\ell}}^{2^k}$ with $x_L \sim V_{w}^{2^k}$ for $w\notin\{u_{\ell},u_{\ell+1}\}$, or $x_L \notin V_{u_{\ell}}^{2^k}$. If there does not exist an index $\ell \in \{jK,\ldots, jK+3^d-1\}$ for which $\|r(u_{\ell})-r(u_{\ell+1})\|_\infty \geq 2$, then we know that $\|u_{\ell}-u_{jK}\|_\infty =2$ and the geodesic between $x_L$ and $x_R$ traverses the set $\bigcup_{r(w) \in G^\prime : \|w-u_{jK}\|_\infty = 1} V_w^{2^k}$, and thus the distance between $x_L$ and $x_R$ is at least 
	\begin{align*}
		D_{V_\mz^{2^n}}\left(x_L, x_R\right) \geq D^\star_{V_\mz^{2^n}} \left(V_{u_{jK}}^{2^k} , \bigcup_{r(w) \in G^\prime : \|w-u_{jK}\|_\infty \geq 2} V_w^{2^k}\right)
		\geq \eps \Lambda \left(2^k, \beta\right)
	\end{align*}
	where the last inequality holds, as the subpath $R_j$ was assumed to be $\eps$-separated.

	When we insert an edge between the boxes $B_L$ and $B_R$, the distance between $x_L$ and $x_R$ is at most $\dia(B_L)+\dia(B_R)+1\leq 2\frac{\eps^{1.5}}{3} 2^{k\dxp} + 1$. Remember that $\Lambda(2^k,\beta) \geq 2^{k\dxp}$ (see \eqref{eq:subsuperimpli}). Thus we have for all edges $e \in B_L \times B_R \coloneqq \left\{\{u,v\}: u\in B_L, v \in B_R\right\}$ 
	\begin{align*}
		D_{V_\mz^{2^n}}\left(x_L,x_R;\omega\right) - D_{V_\mz^{2^n}}\left(x_L,x_R;\omega^{e+} \right) \geq \eps \Lambda\left(2^k, \beta \right) -\left( 2\frac{\eps^{1.5}}{3} 2^{k\dxp} + 1\right) \geq \frac{\eps \Lambda(2^k,\beta)}{4}
	\end{align*}
	where the last inequality holds for $k$ large enough. The boxes $B_L$ and $B_R$ are of side length at least $2\eps^{2/\dxp}2^k$ and are disjoint, as $D_{V_\mz^{2^n}}(x_L,x_R) > \dia \left(B_L\right) + \dia\left(B_R\right)$. Thus, for $k$ large enough, there are at least $0.5 \left(\eps^{2/\dxp}2^k\right)^d \cdot \left(\eps^{2/\dxp}2^k\right)^d$ pairs of vertices $(a,b)\in B_L \times B_R$ for which $|\{a,b\}| \geq \eps^{2/\dxp}2^k$. On the other hand, we also have $|\{a,b\}| \leq (3^d+1)2^k \leq 6^d 2^k$ for all pairs $(a,b)\in B_L \times B_R$, as $\|r(u_{iK})-r(u_{\ell})\|_\infty \leq 3^d$. So in particular we have
	\begin{align}\label{eq:longlist}
		&\notag \sum_{\substack{e \in B_L \times B_R : \\ \eps^{2/\dxp}2^k \leq |e| \leq 6^d 2^k }} p^\prime(\beta, e) \left( D_{V_\mz^{2^n}} \left(x_L,x_R;\omega\right) - D_{V_\mz^{2^n}} \left(x_L,x_R;\omega^{e+} \right) \right)\\
		&\notag
		\geq 
		\sum_{\substack{e \in B_L \times B_R : \\ \eps^{2/\dxp}2^k \leq |e| \leq 6^d 2^k }} p^\prime(\beta, e) \frac{\eps \Lambda(2^k,\beta)}{4} 
		\overset{\eqref{eq:derivativebound1}}{\geq}
		\sum_{\substack{e \in B_L \times B_R : \\ \eps^{2/\dxp}2^k \leq |e| \leq 6^d 2^k }} \frac{e^{-\beta}}{\left(d |e|+\sqrt{d}\right)^{2d}} \frac{\eps \Lambda(2^k,\beta)}{4}
		\\
		&\notag
		\geq
		\sum_{\substack{e \in B_L \times B_R : \\ \eps^{2/\dxp}2^k \leq |e| \leq 6^d 2^k }}
		\frac{e^{-\beta}}{\left(d 6^d 2^k +\sqrt{d}\right)^{2d}} \frac{\eps \Lambda(2^k,\beta)}{4}\\
		&
		\geq
		0.5
		\left(\eps^{2/\dxp} 2^k\right)^d \cdot \left(\eps^{2/\dxp} 2^k\right)^d \cdot
		\frac{e^{-\beta}}{\left(d 6^d 2^k +\sqrt{d}\right)^{2d}} \frac{\eps \Lambda(2^k,\beta)}{4}
		\geq c \Lambda(2^k,\beta)
	\end{align}
	with a constant $c>0$, that depends on $\beta,\eps,\dxp$, and $d$, and for $k$ large enough. For the two points $\mz$ and $(2^n-1)\mo$, and two points $x_L$ and $x_R$ which are in a geodesic between $\mz$ and $(2^n-1)\mo$ in this order, and any edge $e$ we have
	\begin{align*}
		&D_{V_\mz^{2^n}}\left(\mz,(2^n-1)\mo;\omega\right) = D_{V_\mz^{2^n}}\left(\mz,x_L;\omega\right) +
		D_{V_\mz^{2^n}}\left(x_L,x_R;\omega\right) +
		D_{V_\mz^{2^n}}\left(x_R,(2^n-1)\mo;\omega\right) \text{, and }\\
		&
		D_{V_\mz^{2^n}}\left(\mz,(2^n-1)\mo;\omega^{e+}\right) \leq  D_{V_\mz^{2^n}}\left(\mz,x_L;\omega\right) +
		D_{V_\mz^{2^n}}\left(x_L,x_R;\omega^{e+}\right) +
		D_{V_\mz^{2^n}}\left(x_R,(2^n-1)\mo;\omega\right) \text .
	\end{align*}
	Subtracting the second line from the first line already yields that
	\begin{align*}
		\Delta D_{V_\mz^{2^n}}(e) & \coloneqq
		D_{V_\mz^{2^n}} \left(\mz,(2^n-1)\mo;\omega\right) - D_{V_\mz^{2^n}} \left(\mz,(2^n-1)\mo;\omega^{e+} \right) \\
		&
		\geq
		D_{V_\mz^{2^n}} \left(x_L,x_R;\omega\right) - D_{V_\mz^{2^n}} \left(x_L,x_R;\omega^{e+} \right) \text .
	\end{align*}
	This inequality combined with \eqref{eq:longlist} already shows that
	\begin{align*}
		&\sum_{\substack{e \in B_L \times B_R : \\ \eps^{2/\dxp}2^k \leq |e| \leq 6^d 2^k }} p^\prime(\beta, e) \Delta D_{V_\mz^{2^n}}(e)\\
		& \geq
		\sum_{\substack{e \in B_L \times B_R : \\ \eps^{2/\dxp}2^k \leq |e| \leq 6^d 2^k }} p^\prime(\beta, e) \left( D_{V_\mz^{2^n}} \left(x_L,x_R;\omega\right) - D_{V_\mz^{2^n}} \left(x_L,x_R;\omega^{e+} \right) \right)
		\geq c \Lambda\left(2^k,\beta\right) \text .
	\end{align*}
	The above inequality holds for fixed $B_L \subset V_{u_{jK}}^{2^k} , B_R \subset V_{u_{\ell}}^{2^k}$. However, such boxes exist for all indices $j \in \mathcal{INF}(P)$ with $j\geq 1$. Thus, assuming that $D_{G^\prime}\left(\mz, (2^{n-k}-1)\mo \right)= t$ and $\mathcal{G}_{\geq t}$ holds for large enough $t \geq T$, we have for large enough $k$
	\begin{align}\label{eq:ongt}
		& \notag \sum_{e : \eps^{2/\dxp}2^k \leq |e| \leq 6^d 2^k } p^\prime(\beta, e) \Delta D_{V_\mz^{2^n}}(e) \geq \left(\left| \mathcal{INF}(P) \right|-1\right) c \Lambda \left(2^k, \beta\right)\\
		&
		\overset{\eqref{eq:Gt}}{\geq} \frac{t}{4 \cdot 10^{6d}\mu_\beta}  c \Lambda \left(2^k, \beta\right)
		\eqqcolon c^\prime t \Lambda \left(2^k, \beta\right) \text .
	\end{align}
	So far, we always worked on the event $\mathcal{G}_{\geq t}$. 
	Now, we want to get a similar bounds in probability, not conditioning on $\mathcal{G}_{\geq t}$. We write $E_k$ for the set of edges $e$ with $2\vee\eps^{2/\dxp}2^k \leq |e| \leq 6^d 2^k$. For $T\in \N$ large enough we have that
	\begin{align}\label{eq:gleich plug in}
		\notag & \sum_{e\in E_k} p^\prime(\beta, e)
		\E_\beta \left[\Delta D_{V_\mz^{2^n}}(e) \right]\\
		\notag & \geq
		\sum_{t=T}^{\infty} \sum_{e \in E_k } p^\prime(\beta, e)
		\E_\beta \Big[ \Delta D_{V_\mz^{2^n}}(e)  \cdot \mathbbm{1}_{\left\{D_{G^\prime}\left(r(\mz), r((2^{n-k}-1)\mo)\right)= t\right\}} \mathbbm{1}_{\left\{\mathcal{G}_{\geq t}\right\}} \Big]\\
		\notag & \overset{\eqref{eq:ongt}}{\geq} \sum_{t=T}^{\infty} c^\prime t \Lambda\left(2^k,\beta \right) 
		\E_\beta \left[ \mathbbm{1}_{\left\{D_{G^\prime}\left(r(\mz), r((2^{n-k}-1)\mo)\right)= t\right\}} \mathbbm{1}_{\left\{\mathcal{G}_{\geq t}\right\}} \right] \\
		&
		=
		c^\prime  \Lambda\left(2^k,\beta \right) 
		\sum_{t=T}^{\infty}  t
		\E_\beta \left[ \mathbbm{1}_{\left\{D_{G^\prime}\left(r(\mz), r((2^{n-k}-1)\mo)\right)= t\right\}} \mathbbm{1}_{\left\{\mathcal{G}_{\geq t}\right\}} \right] .
	\end{align}
	Next, we use that $\p_\beta \left(\mathcal{G}_{\geq t}^c \right) \leq 6 \cdot 2^{-t}$ \eqref{eq:G prob bound}, and the self-similarity of the model to further calculate that
	\begin{align*}
		&  \sum_{t=T}^{\infty} t 
		\E_\beta \left[ \mathbbm{1}_{\left\{D_{G^\prime}\left(r(\mz), r((2^{n-k}-1)\mo)\right)= t\right\}} \mathbbm{1}_{\left\{\mathcal{G}_{\geq t}\right\}} \right]\\
		& \geq  \sum_{t=T}^{\infty}  t  
		\left( \p_\beta \left( D_{G^\prime}\left(r(\mz), r((2^{n-k}-1)\mo)\right)= t \right) - \p_\beta \left( \mathcal{G}_{\geq t}^c \right) \right)\\
		&  \overset{\eqref{eq:G prob bound}}{\geq}  \sum_{t=T}^{\infty}  t  
		\left( \p_\beta \left( D_{V_\mz^{2^{n-k}}} \left(\mz, (2^{n-k}-1)\mo\right)= t \right) \right) - \sum_{t=T}^\infty 6t 2^{-t}\\
		&
		\geq
		\sum_{t=0}^{\infty}  t  
		\left( \p_\beta \left( D_{V_\mz^{2^{n-k}}} \left(\mz, (2^{n-k}-1)\mo\right)= t \right) \right)  - 
		\sum_{t=0}^{T-1}  t  
		\left( \p_\beta \left( D_{V_\mz^{2^{n-k}}} \left(\mz, (2^{n-k}-1)\mo\right)= t \right) \right) - 6
		\\
		& \geq  \left( \E_\beta \left[D_{V_\mz^{2^{n-k}}} \left( \mz, (2^{n-k}-1)\mo \right)\right] - T - 6 \right)
		\geq \tilde{c} \Lambda\left(2^{n-k},\beta \right) 
	\end{align*} 
	for some $\tilde{c}>0$ and all $k,n-k$ large enough. Inserting this into \eqref{eq:gleich plug in} finally shows that
	\begin{align*}
		\sum_{e \in E_k } p^\prime(\beta, e)
		\E_\beta \left[\Delta D_{V_\mz^{2^n}}(e) \right] \geq c^\prime \Lambda\left(2^{k},\beta \right) \tilde{c} \Lambda\left(2^{n-k},\beta \right)
		\geq \bar{c} \Lambda\left(2^{n},\beta \right) 
	\end{align*}
	for some $\bar{c}>0$ small enough and $k,n-k$ large enough.\\
	
	\noindent
	For each edge $e$, there is only a bounded number of levels $k \in \N$ for which $2\vee \eps^{2/\dxp}2^k \leq |e| \leq 6^d 2^k$. Writing $E=\bigcup_{k=0}^{\infty} E_k$ for the set of edges with length at least $2$, we get that
	\begin{align*}
		& \sum_{e \in E} p^\prime(\beta, e)
		\E_\beta \left[ \Delta D_{V_\mz^{2^n}}(e) \right]\geq c_1 \sum_{k=1}^{n} \sum_{e \in E_k } p^\prime(\beta, e)
		\E_\beta \left[ \Delta D_{V_\mz^{2^n}}(e) \right]\\
		& \geq c_1  \sum_{k=1}^{n} \bar{c} \Lambda\left(2^n,\beta \right)
		\geq c_2 \log\left(2^n\right) \Lambda\left(2^n,\beta \right)
	\end{align*}
	for constants $c_1,c_2>0$ and $n$ large enough. This already implies that for $n$ large enough
	\begin{align}\label{eq:finalineqmono}
		&\notag \sum_{e \in E} p^\prime(\beta, e)
		\E_\beta \left[ D_{V_\mz^{2^n}}\left(\mz,(2^n-1)\mo;\omega^{e-}\right) - D_{V_\mz^{2^n}}\left(\mz,(2^n-1)\mo;\omega^{e+} \right) \right]\\
		&
		\geq 
		\sum_{e \in E} p^\prime(\beta, e)
		\E_\beta \left[ D_{V_\mz^{2^n}} \left(\mz,(2^n-1)\mo;\omega \right) - D_{V_\mz^{2^n}} \left(\mz,(2^n-1)\mo;\omega^{e+} \right) \right]
		\geq c_2 \log\left(2^n\right) \Lambda\left(2^n,\beta \right) \text .
	\end{align}
	
	\noindent
	\textbf{Part (E).}
	Finally, let us see how this bound implies strict monotonicity of the distance exponent $\dxp(\beta)$. We know that
	\begin{align*}
		\dxp(\beta) = \lim_{n\to \infty} \frac{\log \left( \E_\beta\left[ D_{V_\mz^{2^n}}\left( \mz, (2^n-1)\mo \right) \right]\right)}{\log(2^n)}
	\end{align*}
	and that for fixed $n$ the function 
	\begin{align*}
		\beta \mapsto \frac{\log \left( \E_\beta\left[ D_{V_\mz^{2^n}}\left( \mz, (2^n-1)\mo \right) \right]\right)}{\log(2^n)}
	\end{align*}
	is, by Russo's formula for expectations \eqref{eq:russoexp}, differentiable. So we can calculate the derivative and bound it from above by
	\begin{align*}
		&\frac{\md }{\md \beta}  \frac{\log \left( \E_\beta\left[D_{V_\mz^{2^n}}\left( \mz, (2^n-1)\mo \right) \right]\right)}{\log(2^n)} \\
		& = \frac{1}{\E_\beta\left[D_{V_\mz^{2^n}}\left( \mz, (2^n-1)\mo \right) \right] \log(2^n)} \frac{\md }{\md \beta} \E_\beta\left[D_{V_\mz^{2^n}}\left( \mz, (2^n-1)\mo \right) \right] \\
		& = \frac{\sum_{ e \in E} p^\prime(\beta, e) \E_\beta \left[ D_{V_\mz^{2^n}}\left( \mz, (2^n-1)\mo ; \omega^{e+}\right) - D_{V_\mz^{2^n}}\left( \mz, (2^n-1)\mo ; \omega^{e-}\right) \right]}{\E_\beta\left[D_{V_\mz^{2^n}}\left( \mz, (2^n-1)\mo \right) \right] \log(2^n)}  \\
		& \overset{\eqref{eq:finalineqmono}}{\leq} \frac{-c_2 \Lambda\left(2^n,\beta\right) \log(2^n)}{\E_\beta\left[D_{V_\mz^{2^n}}\left( \mz, (2^n-1)\mo \right) \right] \log(2^n)} \leq -c_2 \eqqcolon c(\beta)
	\end{align*}
	for some $c(\beta)<0$ and this holds for all $n \in \N_{>0}$ large enough. Now fix $0<\beta_1<\beta_2<\infty$.  
	For each fixed $\beta \in \left[ \beta_1, \beta_2 \right]$ there exists $n(\beta) < \infty$ such that for all $n \geq n(\beta)$
	\begin{align}\label{eq:finalder}
		\frac{\md }{\md \beta} \frac{\log \left( \E_\beta\left[D_{V_\mz^{2^n}}\left( \mz, (2^n-1)\mo \right) \right]\right)}{\log\left(2^n\right)} \leq \frac{c(\beta)}{2}
	\end{align}
	holds. So in particular we can take $N$ large enough and $c<0$ with $|c|$ small enough so that the set of $\beta \in \left[\beta_1,\beta_2\right]$ which satisfy $\frac{c(\beta)}{2} < c$, and which satisfy \eqref{eq:finalder} for all $n\geq N$, has Lebesgue measure at least $\frac{\beta_2-\beta_1}{2}$.
	Thus we get
	\begin{align*}
		\dxp(\beta_2)-\dxp(\beta_1) & = \lim_{n\to \infty} \left( \frac{\log \left( \E_{\beta_2}\left[D_{V_\mz^{2^n}}\left( \mz, (2^n-1)\mo \right) \right]\right)}{\log(2^n)} - \frac{\log \left( \E_{\beta_1}\left[D_{V_\mz^{2^n}}\left( \mz, (2^n-1)\mo \right) \right]\right)}{\log(2^n)} \right) \\
		& = \lim_{n \to \infty} \int_{\beta_1}^{\beta_2} \frac{\md }{\md \beta}  \frac{\log \left( \E_\beta\left[D_{V_\mz^{2^n}}\left( \mz, (2^n-1)\mo \right) \right]\right)}{\log(2^n)} \md \beta \\
		& \leq \frac{\beta_2-\beta_1}{2} c < 0 \text, 
	\end{align*}
	which finishes the proof of the strict monotonicity
\end{proof}

\section{Continuity of the distance exponent}\label{sec:continuity}

In this section, we show that the distance exponent is continuous in $\beta$.  With the tools that we have developed so far, we can already prove continuity from the left:

\begin{lemma}\label{lem:leftcontinuity}
	The distance exponent $\dxp(\beta)$ is continuous from the left.
\end{lemma}

\begin{proof}
	Remember that 
	\begin{align*}
		\dxp(\beta) = \inf_{n\geq 2} \frac{\log\left( \Lambda(n, \beta) \right) }{\log(n)} \ 
	\end{align*}
	which is stated in Lemma \ref{lem:submultiplicativity}. For fixed $n \in \N$, the function $\beta \mapsto \Lambda(n, \beta )$ is continuous and decreasing in $\beta$. The continuity holds, as the inclusion probabilities $p(\beta,e)$ are continuous in $\beta$ for all edges $e$, and the quantity $\Lambda(n,\beta)$ only depends on the finitely many edges with both endpoints in $V_\mz^n$. As the functions $p(\beta,e)$ are increasing in $\beta$ for all edges $e$, one can see with the Harris coupling that for every $n\in \N$ the function $\beta \mapsto \Lambda\left(n,\beta\right)$ is decreasing.
	So we get that for all $\beta >0$
	\begin{align*}
		\lim_{\eps \searrow 0} \dxp(\beta - \eps) & 
		= \inf_{\eps > 0} \dxp(\beta - \eps) 
		= \inf_{\eps > 0} \inf_{n\geq 2} \frac{\log\left( \Lambda(n, \beta-\eps) \right) }{\log(n)}
		= \inf_{n\geq 2} \inf_{\eps > 0} \frac{\log\left( \Lambda(n, \beta-\eps) \right) }{\log(n)}\\
		& = \inf_{n\geq 2} \frac{\log\left( \Lambda(n, \beta ) \right) }{\log(n)} = \dxp(\beta) \text ,
	\end{align*}
	and this shows continuity from the left.
\end{proof}
The proof of continuity from the right is more difficult. We consider independent bond percolation on the complete graph with vertex set $V=V_\mz^{2^n}$ and edge set $E=\left\{\{x,y\}: x,y \in V_\mz^{2^n}, x\neq y\right\}$. For $k\in \{1,\ldots,n\}$ and $\beta_1,\beta_2 \geq 0$, we denote by $\p_{\beta_1 < k}^{\beta_2 \geq k}$ the product probability measure on the space $\{0,1\}^E$ where edges $e=\{u,v\}$ are  open with the following probabilities:
\begin{align*}
	\p_{\beta_1 < k}^{\beta_2 \geq k}\left(\omega(\{u,v\})=1\right) = 
	\begin{cases}
		1-e^{-\beta_1 \int_{u+\cC} \int_{v+\cC} \frac{1}{\|x-y\|^{2d}} \md x \md y}  & \text{if } 1 < |\{u,v\}| \leq 2^k-1 \\
		1-e^{-\beta_2 \int_{u+\cC} \int_{v+\cC} \frac{1}{\|x-y\|^{2d}} \md x \md y}  & \text{if } |\{u,v\}| \geq 2^k \\
		1  & \text{if } |\{u,v\}|=1
	\end{cases} \text ,
\end{align*}
so in particular the measure $\p_{\beta_1 < 1}^{\beta_2 \geq 1}$ is identical to the measure $\p_{\beta_2}$, and the measure $\p_{\beta_1 < n}^{\beta_2 \geq n}$ on the graph with vertex set $V_\mz^{2^n}$ is identical to the measure $\p_{\beta_1}$. For $k \in \{2,\ldots,n-1\}$, we think of the measure $\p_{\beta_1 < k}^{\beta_2 \geq k}$ as an interpolation between the probability measures $\p_{\beta_1}$ and $\p_{\beta_2}$ on the graph  with vertex set $V_\mz^{2^n}$. We denote by $\E_{\beta_1 < k}^{\beta_2 \geq k}$ the expectation under $\p_{\beta_1 < k}^{\beta_2 \geq k}$. Our main strategy of the proof of Theorem \ref{theo:continuity} is as follows: We know that
\begin{align*}
	\dxp(\beta) = \lim_{n\to \infty}
	\frac{\log\left(\E_\beta\left[D_{V_\mz^{2^n}} \left(\mz,(2^n-1)\mo\right)\right]\right)}{\log \left(2^n -1 \right)} = \lim_{n\to \infty}
	\frac{\log\left(\E_\beta\left[D_{V_\mz^{2^n}} \left(\mz,(2^n-1)\mo\right)\right]\right)}{\log(2) n}
\end{align*}
and thus, using a telescopic sum, we also have
\begin{align}\label{eq:cont:sum}
	& \dxp(\beta) \notag - \dxp(\beta + \eps ) \\
	& \notag =  \lim_{n\to \infty} \frac{1}{\log(2) n} \left(\log\left(\E_\beta\left[D_{V_\mz^{2^n}} \left(\mz,(2^n-1)\mo\right)\right]\right) - \log\left(\E_{\beta+\eps}\left[D_{V_\mz^{2^n}} \left(\mz,(2^n-1)\mo\right)\right]\right)\right) \\
	& \notag = \lim_{n\to \infty} \frac{1}{\log(2) n} \left(\log\left(\E_{\beta < n}^{\beta+\eps \geq n} \left[D_{V_\mz^{2^n}} \left(\mz,(2^n-1)\mo\right)\right]\right) - \log\left(\E_{\beta < 1}^{\beta+\eps \geq 1} \left[ D_{V_\mz^{2^n}} \left(\mz,(2^n-1)\mo\right)\right]\right)\right) \\
	& \notag = \frac{1}{\log(2)} \lim_{n\to \infty} \frac{1}{ n} \sum_{k=2}^n \left(\log\left(\E_{\beta < k}^{\beta+\eps \geq k} \left[D_{V_\mz^{2^n}} \left(\mz,(2^n-1)\mo\right)\right]\right) - \log\left(\E_{\beta < k-1}^{\beta+\eps \geq k-1} \left[D_{V_\mz^{2^n}} \left(\mz,(2^n-1)\mo\right)\right]\right)\right) \\
	& = \frac{1}{\log(2)} \lim_{n\to \infty} \frac{1}{ n} \sum_{k=2}^n \log\left(\frac{\E_{\beta < k}^{\beta+\eps \geq k} \left[D_{V_\mz^{2^n}} \left(\mz,(2^n-1)\mo\right)\right]}{\E_{\beta < k-1}^{\beta+\eps \geq k-1} \left[D_{V_\mz^{2^n}} \left(\mz,(2^n-1)\mo\right)\right]}\right)  \text .
\end{align}
So in order to show the (right-)continuity of the distance exponent it suffices to show that for all $\beta \geq 0$
\begin{align}\label{eq:cont:toshow}
	\lim_{\eps \searrow 0} \lim_{n\to \infty} \frac{1}{ n} \sum_{k=2}^n \log\left(\frac{\E_{\beta < k}^{\beta+\eps \geq k} \left[D_{V_\mz^{2^n}} \left(\mz,(2^n-1)\mo\right)\right]}{\E_{\beta < k-1}^{\beta+\eps \geq k-1} \left[D_{V_\mz^{2^n}} \left(\mz,(2^n-1)\mo\right)\right]}\right) = 0 . 
\end{align}
In order to show this, we will show that the terms
\begin{align*}
	\log\left(\frac{\E_{\beta < k}^{\beta+\eps \geq k} \left[D_{V_\mz^{2^n}} \left(\mz,(2^n-1)\mo\right)\right]}{\E_{\beta < k-1}^{\beta+\eps \geq k-1} \left[D_{V_\mz^{2^n}} \left(\mz,(2^n-1)\mo\right)\right]}\right)  , \ k \in \{2,\ldots, n\},
\end{align*}
are bounded uniformly for all $\eps \in \left[0,1\right], k,n\in \N$ (cf. Lemmas \ref{lem:uniformovereps} and \ref{lem:auxil}) and we will show that the terms converge to $0$, as $\eps \searrow 0 $ and $ k , n-k \to  \infty$, with explicit bounds on the speed of the decay (cf. Lemma \ref{lem:eps h functions}). We then show below, how exactly this implies \eqref{eq:cont:toshow}. For the proof of the convergence to $0$, we construct a coupling between the measures $\p_{\beta < k}^{\beta+\eps \geq k}$ and $\p_{\beta < k-1}^{\beta+\eps \geq k-1}$ and compare the differences in terms of the chemical distance. Intuitively, it makes sense that this difference is small, as there will be only very few edges on which the two environments differ as $\eps \to 0$. \\

The measure $\p_{\beta_1 < k}^{\beta_2 \geq k}$ also has the following useful properties. If $e=\{a,b\}$ is an edge with $a,b\in V_u^{2^k}$, then $|e|=\|a-b\|_\infty \leq 2^{k}-1$. So in particular, the probability that the edge $e$ is open is the same for the measures $\p_{\beta_1 < k}^{\beta_2 \geq k}$ and $\p_{\beta_1}$. This implies that for a fixed box $V_u^{2^k}$, the percolation configuration sampled within this box has the exact same distribution, no matter if it is sampled from $\p_{\beta_1 < k}^{\beta_2 \geq k}$ or from $\p_{\beta_1}$. On the other hand, if $e=\{a,b\}$ is an edge with $a\in V_u^{2^k}, b\in V_v^{2^k}$ with $\|u-v\|_\infty \geq 2$, then $|e|=\|a-b\|_\infty \geq 2^{k}$. So in particular, the probability that the edge $e$ is open is the same for the measures $\p_{\beta_1 < k}^{\beta_2 \geq k}$ and $\p_{\beta_2}$. From this, we also directly get that for all $u,v \in \Z^d$ with $\|u-v\|_\infty \geq 2$ one has
\begin{equation}\label{eq:block contraction}
	\p_{\beta_2} \left(V_u^{2^k} \sim V_v^{2^k}\right) = \p_{\beta_1 < k}^{\beta_2 \geq k}  \left(V_u^{2^k} \sim V_v^{2^k}\right) .
\end{equation}
Note that equality \eqref{eq:block contraction} also holds when $\|u-v\|_\infty = 1$, as both sides of the equality equal $+1$.
This also has an implication when contracting boxes of the form $V_u^{2^k}$. Start with the vertex set $V=\Z^d$ $\big($respectively $V=V_\mz^{2^{n}}$, where $n\geq k\big)$ and consider bond percolation with the measure $\p_{\beta_1 < k}^{\beta_2 \geq k}$ on $V$. Then contract all boxes of the form $V_u^{2^k}$ into a single vertex named $r(u)$ for all $u\in \Z^d$ $\big($respectively, for all $u\in V_\mz^{2^{n-k}}\big)$ and call the resulting graph $G^\prime =(V^\prime,E^\prime)$ $($cf. section \ref{sec:contraction} for the definition of the graph $G^\prime)$. Then $G^\prime$ has the exact same distribution as percolation on the set of vertices $\Z^d$ $\big($respectively $V_{\mz}^{2^{n-k}}\big)$ with the measure $\p_{\beta_2}$.\\

Before going to the proof of the right-continuity, we need to prove several technical results. In Lemma \ref{lem:moments for diameter}, we investigate the exponential moments of $\frac{\dia\left(V_\mz^m\right)}{m^{\dxp}}$, uniformly over $m$. In section \ref{subsec:cont:lemmas}, we present various properties of the mixed measure $\p_{\beta < k}^{\beta+\eps \geq k}$ that we need later in the proof. These properties are proven only later in section \ref{sec:auxil}. In section \ref{subsec:cont:proof} we prove $\eqref{eq:cont:toshow}$ and thus continuity of the distance exponent $\dxp$.

\begin{lemma}\label{lem:moments for diameter}
	For all $\beta \geq 0$, there exists a constant $C_1 < \infty$ such that for all $s\geq1$, and all $m \in \N_{>0}$
	\begin{align}\label{eq:expomoments for diameter}
		\E_\beta \left[e^{s \frac{\dia\left(V_\mz^m\right)}{m^{\dxp(\beta)}}}\right] < e^{C_1 s^{C_1}} .
	\end{align}
\end{lemma}

\begin{proof}
	Define $Y_m \coloneqq \frac{\dia\left(V_\mz^m\right)}{m^{\dxp(\beta)}}$.
	In \cite[Theorem 6.1]{baeumler2022distances}, it is proven that for each $\beta \geq 0$ there exists an $\eta>1$ and a $C< \infty$ such that
	\begin{equation}\label{eq:moments for diameter}
		\sup_{m}\E_\beta \left[e^{ Y_m^\eta}\right] \leq C .
	\end{equation}
	In fact, it is proven that \eqref{eq:moments for diameter} holds for all $\eta \in \left(0,\frac{1}{1-\theta(\beta)}\right)$ for sufficiently large $C$.
	So let $\eta>1$ and $C<\infty$ be such that \eqref{eq:moments for diameter} is satisfied.
	For all $y,s>0$ one has
	\begin{align*}
		sy 
		\leq 
		sy \mathbbm{1}_{\{s< y^{\eta-1}\}} + sy \mathbbm{1}_{\{s \geq y^{\eta-1}\}}
		\leq 
		y^{\eta-1}y \mathbbm{1}_{\{s< y^{\eta-1}\}} + sy \mathbbm{1}_{\left\{s^{\frac{1}{\eta-1}} \geq y\right\}} \leq y^{\eta} + s^{\frac{\eta}{\eta-1}}.
	\end{align*}
	Combining this with \eqref{eq:moments for diameter}, we get that
	\begin{equation*}
		\E_\beta \left[e^{ s Y_m}\right] 
		\leq 
		\E_\beta \left[e^{ Y_m^\eta + s^{\frac{\eta}{\eta-1}} }\right]
		=
		\E_\beta \left[e^{ Y_m^\eta  }\right]  e^{s^{\frac{\eta}{\eta-1}}}
		\leq 
		C e^{s^{\frac{\eta}{\eta-1}}} \leq e^{C_1 s^{C_1}}
	\end{equation*}
	for some $C_1 < \infty$ and all $s\geq 1$.
\end{proof}

\subsection{Uniform bounds for the mixed measure}\label{subsec:cont:lemmas}

In this chapter, we give several bounds for the chemical distances of points under the measure $\p_{\beta < k}^{\beta+\eps \geq k}$ that hold uniformly over $\eps \in \left[0,1\right]$ and $k \leq n$. These bounds were partially already proven in the previous chapters or in \cite{baeumler2022distances,ding2013distances} for fixed $\beta$ and $\eps=0$. The results stated in this chapter will then help us to prove Theorem \ref{theo:continuity} in section \ref{subsec:cont:proof}. The proofs of the results stated in this section are deferred to section \ref{sec:auxil}.

The next two lemmas deal with the graph distance of certain points in boxes that have direct edges to other far away blocks. They say that it is unlikely to have two points $x,y \in V_\mz^{2^k}$ that are both endpoints of long edges and are close in the metric $D_{V_\mz^{2^k}}$. They are very similar to Lemma \ref{lem:graphspacing}, but for the measure $\p_{\beta < k}^{\beta+\eps \geq k}$ instead of $\p_\beta$.

\begin{lemma}\label{lem:goodeventB}
	Let $V_u^{2^k}$ be a block with side length $2^k$ that is connected to $V_\mz^{2^k}$ with $\|u\|_\infty\geq 2$. Let $\mathcal{B}_u (\delta)$ be the following event:
	\begin{align*}
		\mathcal{B}_u (\delta) = \bigcap_{ \substack{ x,y \in V_\mz^{2^k} : \\ x,y \sim V_u^{2^k} , x \neq y} } \left\{D_{V_\mz^{2^k}}\left(x,y\right) \geq \delta 2^{k\dxp(\beta)} \right\} \text .
	\end{align*}
	For every $\beta >0$, there exists a function $f_1(\delta)$ with $f_1(\delta) \underset{\delta\to 0}{\longrightarrow}$ 1 such that for all large enough $k \geq k(\delta)$, all $u \in \Z^d$ with $\|u\|_\infty \geq 2$, and all $\eps \in \left[0,1\right]$
	\begin{equation*}
		\p_{\beta < k}^{\beta+\eps \geq k} \left(\mathcal{B}_u (\delta) \ | \ V_\mz^{2^k} \sim V_u^{2^k} \right) \geq f_1(\delta) \text .
	\end{equation*}
\end{lemma}

The intuition behind this result is the following. Conditioned on the event $\left\{V_\mz^{2^k} \sim V_u^{2^k}\right\}$, there typically are only a small number of vertices $x \in V_\mz^{2^k}$ with  $x \sim V_u^{2^k}$ (see Lemma \ref{lem:number of edges conditioned}). Typically, these vertices are relatively far apart $($of order $\approx 2^k)$ in terms of their Euclidean distance. As the structure of the percolation configuration inside the set $V_\mz^{2^k}$ is independent of the position of the vertices $x$ with $x \sim V_\mz^{2^k}$, these vertices are typically also far apart in terms of the graph distance inside $V_\mz^{2^k}$. Also note that this result is not true when $\|u\|_\infty=1$, as for $\|u\|_\infty=1$ there are typically a lot of vertices $x\in V_\mz^{2^k}$ with $x\sim V_\mz^{2^k}$. The next lemma is similar to that before, with the difference being that one considers vertices $x,y\in V_\mz^{2^k}$ with $x\sim V_u^{2^k}, y \sim V_v^{2^k}$ for distinct $u,v\in \Z^d$.

\begin{lemma}\label{lem:goodeventA}
	Let $V_u^{2^k}, V_v^{2^k}$ be two blocks of side length $2^k$ that are connected to $V_\mz^{2^k}$, with $u\neq v \neq \mz$ and $\|u\|_\infty \geq 2$. Let $\mathcal{A}_{u,v}(\delta)$ be the following event:
	\begin{align*}
		\mathcal{A}_{u,v} (\delta) 
		= 
		\bigcap_{ \substack{x \in V_\mz^{2^k}: \\ x \sim V_u^{2^k}} } \ \bigcap_{ \substack{ y \in V_\mz^{2^k}: \\ y \sim V_v^{2^k}} } \left\{D_{V_\mz^{2^k}}\left(x,y\right) \geq \delta 2^{k\dxp(\beta)} \right\} \text .
	\end{align*}
	For every $\beta >0$, there exists a function $f_2(\delta)$ with $f_2(\delta) \underset{\delta\to 0}{\longrightarrow}1$ such that for all large enough $k \geq k(\delta)$, all $u,v \in \Z^d\setminus \{\mz\}$ with $u\neq v, \|u\|_\infty \geq 2$, and all $\eps \in \left[0,1\right]$
	\begin{equation*}
		\p_{\beta < k}^{\beta+\eps \geq k} \left(\mathcal{A}_{u,v}(\delta) \ | \ V_u^{2^k} \sim V_\mz^{2^k} \sim V_v^{2^k} \right) \geq f_2(\delta) \text .
	\end{equation*}
\end{lemma}

For the next lemma, we define the graph $G^\prime$ as the graph, in which we contract boxes of the form $V_u^{2^k}$ for $u\in \Z^d$. The vertex that results from contracting the box $V_u^{2^k}$ is called $r(u)$. This lemma is the analogue of Lemma \ref{lem:goodeventW} for the measure $\p_{\beta < k}^{\beta+\eps \geq k}$ instead of $\p_\beta$.

\begin{lemma}\label{lem:goodeventD}
	Let $\mathcal{B}(\delta)$ be the event
	\begin{align*}
		\mathcal{B}(\delta) \coloneqq \left\{ D^\star \left( V_\mz^{2^k} , \bigcup_{u \in \Z^d : \| u \|_\infty \geq 2} V_u^{2^k} \right) \geq \delta 2^{k\dxp(\beta)} \right\} \text .
	\end{align*}
	For every $\beta >0$, there exists a function $f_3(\delta)$ with $f_3(\delta) \underset{\delta\to 0}{\longrightarrow}1$ such that for all large enough $k \geq k(\delta)$, all $\eps \in \left[0,1\right]$, and all realizations of $G^\prime$ with $\deg^{\cN} \left(r(\mz)\right) \leq 9^d 100 \mu_{\beta+1}$
	\begin{equation*}
		\p_{\beta < k}^{\beta+\eps \geq k} \left( \mathcal{B}(\delta) \ | \ G^\prime \right) \geq f_3(\delta) \text .
	\end{equation*}
\end{lemma}

The proof of this lemma is similar to the proof of \cite[Lemma 5.4]{baeumler2022distances}, and we prove it in section \ref{sec:auxil}.
From here on we also use the notation $f(\delta) = \min \left\{f_1(\delta) , f_2(\delta) , f_3(\delta)\right\}$.
The next lemma is a version of the sub- and supermultiplicativity for the mixed measure $\p_{\beta < k}^{\beta+\eps \geq k}$ and is conceptually similar to Lemma \ref{lem:submultiplicativity}, respectively Remark \ref{remark:super}. 

\begin{lemma}\label{lem:uniformovereps}
	For all $\beta > 0$, all $\eps \in \left[0,1\right]$, and all $k,n\in \N$ with $k\leq n$ one has
	\begin{align}\label{eq:trivia}
		& \frac{1}{72d}\Lambda\left(2^{n-k},\beta + \eps\right)
		\leq \E_{\beta < k}^{\beta+\eps \geq k} \left[D_{V_\mz^{2^n}}\left(\mz,(2^n-1) \mo \right)\right] .
	\end{align}
	Further, for all $\beta > 0$, there exist constants $c_\beta>0$ and $C(\beta)<\infty$ such that for all $\eps \in \left[0,1\right]$ and all $k,n \in\N$ with $k\leq n$ and all $x,y \in V_\mz^{2^n}$
	\begin{align}
		\label{trivia2}
		& \E_{\beta < k}^{\beta+\eps \geq k} \left[D_{V_\mz^{2^n}}\left(x,y \right)\right]   \leq C(\beta) \Lambda \left(2^k,\beta\right)  \Lambda\left(2^{n-k},\beta + \eps\right), \text{ and } \\
		& \label{eq:uniformovereps} c_\beta  \Lambda\left(2^k,\beta\right)  \Lambda\left(2^{n-k},\beta + \eps\right)
		\leq \E_{\beta < k}^{\beta+\eps \geq k} \left[D_{V_\mz^{2^n}}\left(\mz,(2^n-1) \mo \right)\right]  .
	\end{align}
\end{lemma}

The proof of this lemma is similar to the proofs of \cite[Lemma 2.3 and Lemma 5.5]{baeumler2022distances}, and we prove it in section \ref{sec:auxil}. We want to get similar bounds on the second moment of distances $D_{V_\mz^{2^n}}$ under the measure $\p_{\beta < k}^{\beta+\eps \geq k}$. For this, we introduce the following lemma, which was already proven in \cite[Lemma 4.5]{baeumler2022distances}. Roughly speaking, this lemma says that the second moment of distances is of the same order as the square of the first moment.

\begin{lemma}\label{lem:uniform 2nd moment bound}
	For all $\beta \geq 0$, there exists a constant $\tilde{C_\beta} < \infty$ such that for all $n \in \N$, all $\eps \in \left[0,1\right]$ and all $x,y \in V_{\mz}^n$
	\begin{align}\label{eq:second moment bound}
		\E_{\beta+\eps} \left[ D_{V_{\mz}^n} (x,y)^2 \right] \leq \tilde{C_\beta} \Lambda(n,\beta+\eps)^2.
	\end{align}
\end{lemma}

Having this lemma allows us to prove a uniform bound on the second moment of distances under the measure $\p_{\beta < k}^{\beta+\eps \geq k}$.

\begin{lemma}\label{lem:uniform2ndmomentbound:renormalized}
	For all $\beta \geq 0$, there exists a constant $C_\beta<\infty$ such that uniformly over all $\eps \in \left[0,1\right]$, all $k \leq n$, and all $x,y \in V_\mz^{2^n}$
	\begin{align}\label{eq:uniformoverbeta}
		\E_{\beta < k}^{\beta+\eps \geq k} \left[D_{V_\mz^{2^n}}\left(x,y\right)^2 \right] \leq
		C_\beta  \Lambda(2^k,\beta)^2 \Lambda(2^{n-k},\beta+\eps)^2.
	\end{align}
\end{lemma}

The proof of this lemma is given in section \ref{sec:auxil}.

\subsection{The proof of Theorem \ref{theo:continuity} }\label{subsec:cont:proof}

In order to prove Theorem \ref{theo:continuity}, we  use a coupling between the measures $\p_{\beta < k}^{\beta+\eps \geq k}$ and $\p_{\beta < k-1}^{\beta + \eps \geq k-1}$. Let $\omega \in \{0,1\}^E $ be distributed according to $\p_{\beta < k}^{\beta+\eps \geq k}$. Let $E=\left\{\{x,y\} \subset V_\mz^{2^n}: x\neq y\right\}$ and let $\chi \in \{0,1\}^E$ be a random vector that is independent of $\omega$ and has independent coordinates such that
\begin{align}\label{eq:chi}
	\p\left(\chi\left(\{u,v\}\right)= 1\right) = \begin{cases}
		1- e^{-\eps \int_{u+\cC} \int_{v+\cC} \frac{1}{\|x-y\|^{2d}} \md x \md y } & \text{if } 2^{k-1} \leq |\{u,v\}| \leq 2^k-1 \\
		0 & \text{else}
	\end{cases}.
\end{align}
Then set $\omega^\chi (e) \coloneqq \omega(e) \vee \chi(e) = \max\{ \omega(e) , \chi(e)\}$ for all edges $e \in E$. 
The coordinates of $\omega^\chi$ are independent and for $e =\{u,v\} \in E$ with $2^{k-1} \leq |e| \leq 2^k -1$ we have
\begin{align*}
	\p \left(\omega^\chi (e) = 0 \right) & = \p \left(\omega(e) = 0 \right) \p\left( \chi(e)= 0 \right) 
	= e^{- \int_{u+\cC} \int_{v+\cC} \frac{\beta}{\|x-y\|^{2d}} \md x \md y } e^{- \int_{u+\cC} \int_{v+\cC} \frac{\eps}{\|x-y\|^{2d}} \md x \md y } \\ 
	& = e^{- \int_{u+\cC} \int_{v+\cC} \frac{\beta + \eps}{\|x-y\|^{2d}} \md x \md y }  = 1- p\left(\beta+\eps, \{u,v\} \right)
\end{align*}
and thus $\omega^\chi$ is distributed according to the measure $\p_{\beta \leq k-1}^{\beta+\eps > k-1}$. The pair $(\omega, \omega^\chi)$ is a coupling of the measures  $\p_{\beta < k}^{\beta+\eps \geq k}$ and  $\p_{\beta < k-1}^{\beta + \eps \geq k-1}$. Next, we consider properties of the differences between the two measures.

For a block $ V_u^{2^k} =  2^k u + \left\{0,\ldots,2^k-1\right\}^d$ of side length $2^k$ and every vertex $v\in V_u^{2^k}$, there are at most $\left(2(2^k-1)+1\right)^d \leq 2^{(k+1)d}$ vertices $w \in \Z^d$ with $2^{k-1} \leq \left|\{v,w\}\right| \leq 2^k - 1$. As $\chi$ can only be $+1$ on edges $e$ with $2^{k-1} \leq \left|e \right| \leq 2^k - 1$, we have
\begin{align}\label{eq:coupling bound}
	& \notag \p_{\beta < k}^{\beta+\eps \geq k} \left(\exists v \in V_u^{2^k}, w \in \Z^d \text{ with } \chi\left(\{v,w\}\right)=1 \right)\\
	& \notag \leq 2^{kd} \p_{\beta < k}^{\beta+\eps \geq k} \left(\exists w \in \Z^d \text{ with } \chi\left(\{\mz,w \}\right)=1 \right) \\
	&\notag  \leq 2^{kd} \sum_{ w\in \Z^d : \|w\|_\infty \in \left[2^{k-1},2^k - 1 \right]} \left(1-e^{- \int_{\mz + \cC} \int_{w+\cC} \frac{ \eps}{\|x-y\|^{2d}} \md x \md y }\right)
	\\
	&\notag \overset{\eqref{eq:connectionupper bound} }{ \leq }
	2^{kd} \sum_{w \in \Z^d : \|w\|_\infty \in \left[2^{k-1},2^k - 1 \right]}	
	\frac{2^{2d} \eps}{\|w\|_\infty^{2d}}
	\leq
	2^{kd} \sum_{w \in \Z^d : \|w\|_\infty \in \left[2^{k-1},2^k - 1 \right]}	
	\frac{2^{2d} \eps}{2^{(k-1)2d}} \\
	&
	\leq 2^{kd} 2^{(k+1)d} \frac{2^{4d}\eps}{2^{2kd}} = 2^{5d} \eps \text , 
\end{align}
where we used that
\begin{align}\label{eq:connectionupper bound}
	1- e^{-\eps \int_{\mz+\cC} \int_{w+\cC} \frac{1}{\|x-y\|^{2d}} \md x \md y } = \p_\eps \left(\mz\sim w\right) \leq \frac{2^{2d}\eps}{\|w\|_\infty^{2d}} 
\end{align}
for all $\eps \geq 0$, all $n\in \N$, and all $w\in \Z^d$ with $\|w\|_\infty \geq 2$. This was proven in \cite[Lemma 2.1]{baeumler2022distances}.

Next, we define a notion of good sets inside the graph with vertex set $V_\mz^{2^n}$. For $w\in V_\mz^{2^{n-k}}$, we contract the box $V_w^{2^k}\subset V_\mz^{2^n}$ to a vertex, named $r(w)$, and call the resulting graph $G^\prime=(V^\prime, E^\prime)$. Remember the definition of the events $\mathcal{B}_u(\delta), \mathcal{A}_{u,v}(\delta)$, and $\mathcal{B}(\delta)$ from Lemmas \ref{lem:goodeventB}, \ref{lem:goodeventA}, and \ref{lem:goodeventD}, respectively.

\begin{definition}
	For a small $\delta>0$, we call a vertex $r(w)$ and the underlying block $V_w^{2^{k}}$ $\delta$-{\sl good} (for the environment $\omega$), if all the translated events of $\mathcal{A}_{u,v}(\delta), \mathcal{B}_u(\delta)$, and $\mathcal{B}(\delta)$ occur for the environment $\omega$, i.e., if
	\begin{align}\label{eq:good1}
		\bigcap_{\substack{ x \in V_w^{2^k} : \\ x \sim V_u^{2^{k}}} } \ \bigcap_{\substack{ y \in V_w^{2^k} : \\ y \sim V_v^{2^{k}}   }  } \left\{D_{V_w^{2^{k}}}\left(x,y;\omega\right) \geq \delta 2^{k \dxp(\beta)} \right\}
	\end{align}
	for all $u,v\in V_\mz^{2^{n-k}}$ for which $w \neq v\neq u$, $V_u^{2^{k}} \sim V_w^{2^{k}} \sim V_v^{2^{k}}$ and $\|u-w\|_\infty \geq 2$, and if
	\begin{align}\label{eq:good2}
		\bigcap_{\substack{ x,y \in V_w^{2^k} : \\ x ,y \sim V_u^{2^{k}} , x\neq y } } 
		\left\{D_{V_w^{2^{k}}} (x,y;\omega) \geq \delta 2^{k \dxp(\beta)} \right\}
	\end{align}
	for all $u \in V_\mz^{2^{n-k}}$ with $\|u-w\|_\infty \geq 2$ and $V_u^{2^k}\sim V_w^{2^k}$, and if
	\begin{align}\label{eq:good3}
		D^\star \left( V_w^{2^k} , \bigcup_{u \in V_\mz^{2^{n-k}} : \| u - w \|_\infty \geq 2} V_u^{2^k} ; \omega \right) \geq \delta 2^{k\dxp(\beta)} .
	\end{align} 
\end{definition}

Suppose that a self-avoiding path in $V_\mz^{2^n}$ crosses a $\delta$-good set $V_w^{2^k}$, in the sense that it starts somewhere outside of the set $\bigcup_{u\in V_\mz^{2^{n-k}} : \| u - w \|_\infty \leq 1} V_u^{2^k} $, then goes to the set $V_w^{2^k}$, and then leaves the set $\bigcup_{u \in  V_\mz^{2^{n-k}} : \| u - w \|_\infty \leq 1} V_u^{2^k} $ again. When the path enters the set $V_w^{2^k}$ at the vertex $x \in V_w^{2^k}$, coming from some a block $V_u^{2^k}$ with $\|u-w\|_\infty\geq 2$, the path needs to walk a distance of at least $\delta 2^{k\dxp(\beta)}$ inside the set $V_w^{2^k}$ to reach a vertex $y\in V_w^{2^k}$ that is connected to the complement of $V_w^{2^k}$, because of \eqref{eq:good1} and \eqref{eq:good2}. When the path enters the set $V_w^{2^k}$ from a block $V_v^{2^k}$ with $\|v-w\|_\infty=1$, then the path crosses the annulus between $V_w^{2^k}$ and $\bigcup_{u \in  V_\mz^{2^{n-k}} : \| u - w \|_\infty \geq 2} V_u^{2^k}$. So in particular it needs to walk a distance of at least $\delta 2^{k\dxp(\beta)}$ in order to cross this annulus, because of \eqref{eq:good3}. Overall, we see that the path needs to walk a distance of at least $\delta 2^{k\dxp(\beta)}$ within the set $\bigcup_{u \in  V_\mz^{2^{n-k}} : \| u - w \|_\infty \leq 1} V_u^{2^k}$ in order to cross the set $V_w^{2^k}$. Let $f_1,f_2$, and $f_3$ be the functions defined in Lemmas \ref{lem:goodeventB}, \ref{lem:goodeventA}, and \ref{lem:goodeventD} and let $f(\delta)= \min \left\{f_1(\delta) , f_2(\delta) , f_3(\delta)\right\}$.
Let $\delta$ be small enough such that
\begin{equation}\label{eq:deltasmall}
	9^{2d} 5000 \mu_{\beta+1}^2 \left(1-f(\delta)\right) \leq \left(32\mu_{\beta + 1}\right)^{-9^d400\mu_{\beta+1}}.
\end{equation}
Such a $\delta >0$ exists, as $f(\delta)= \min \left\{f_1(\delta) , f_2(\delta) , f_3(\delta)\right\}$ tends to $1$ for $\delta \to 0$.
\begin{definition}
	From here on we call a block $V_w^{2^k}$ {\sl good} (for the environment $\omega$) if it is $\delta$-good for the specific choice of $\delta$ in \eqref{eq:deltasmall}, and we call a vertex $r(w)\in G^\prime$ {\sl good} (for the environment $\omega$) if the underlying block $V_w^{2^k}$ is good.
\end{definition}
For a connected set $Z \subset G^\prime$, we are interested in the number of {\sl separated good} vertices inside this set. We say that a subset $U\subset Z$ is a set of separated good vertices if $r(v)$ is good for all $r(v) \in U$ and if $\cN(r(v)) \nsim \cN(r(w))$ for all $r(v)\neq r(w) \in U$. We say that a set $Z\subset G^\prime$ has less than $t$ separated good vertices if every subset $U\subset Z$ of separated good vertices satisfies $|U|<t$, and we say that $Z$ has at least $t$ separated good vertices if there exists a subset $U\subset Z$ of separated good vertices with $|U|\geq t$. Remember that $\mathcal{CS}_K (G^\prime)$ was defined as the set of connected subsets of $G^\prime$ that have size $K$ and contain the origin $r(\mz)$.
In the next lemma, we study how many separated good vertices the connected sets $Z \in \mathcal{CS}_K (G^\prime)$ contain.

\begin{lemma}
	Let $\eps \in \left[0,1\right]$. Consider long-range percolation with the measure $\p_{\beta < k}^{\beta+\eps \geq k}$ on the set $V_\mz^{2^n}$, and let $G^\prime$ be the graph that results from contracting boxes of the form $V_w^{2^k}$ for $w\in V_\mz^{2^{n-k}}$. Then for large enough $K$ one has
	\begin{align}\label{eq:connectedset goodbound}
		\p_{\beta < k}^{\beta+\eps \geq k} \left( \exists Z \in \mathcal{CS}_K\left(G^\prime\right) \text{ with less than $\frac{K}{9^d 400 \mu_{\beta + 1}}$ separated good vertices}  \right) \leq 3 \cdot 2^{-K}.
	\end{align}
\end{lemma}

\begin{proof}
	Let $\hat{Z}=\left\{r(v_1),\ldots,r(v_K) \right\}$ be a connected set in $G^{\prime}$ containing $r(\mz)$. 
	Let $\preceq$ be the lexicographic ordering on $\Z^d$, where we write $\prec$ for strict inequalities. So we can assume that $\hat{Z} = \left\{ r(v_1) , \ldots, r(v_K)\right\}$, where $v_1 \prec v_2 \prec \ldots \prec v_K$.
	For such a set, we add the nearest neighbors to it. Formally, we define the set
	\begin{align*}
		\hat{Z}^\cN = \bigcup_{ r(v) \in \hat{Z} }  \cN\left(r(v)\right)
	\end{align*}
	which is still a connected set containing the origin $r(\mz)$ and satisfies
	\begin{equation}\label{eq:ZN size}
		K \leq |\hat{Z}^\cN| \leq 3^d K .
	\end{equation}
	A vertex $r(u) \in G^\prime$ can be included into the set $\hat{Z}^{\cN}$ in more than one way, meaning that there can be different vertices $r(v),r(\tilde{v}) \in \hat{Z}$ such that $r(u) \in \cN\left(r(v)\right)$ and $r(u) \in \cN\left(r(\tilde{v})\right)$.
	However, if $r(u)\in \hat{Z}^\cN$, then $\cN(r(u)) \cap \hat{Z} \neq \emptyset$. This implies that for each vertex $r(v) \in G^\prime$ there are at most $3^d$ many indices $i\in \{1,\ldots,K\}$ with $ r(v) \in \cN(r(v_i))$. So in particular, using the definition of the neighborhood-degree of a vertex $\deg^{\cN} \left(r(u)\right)$ \eqref{eq:nbh degree}, we see that
	\begin{align}\label{overset}
		\notag &\sum_{i=1}^K \deg^{\cN} \left(r(v_i)\right)
		=
		\sum_{i=1}^K \ \sum_{r(v) \in \cN(r(v_i))} \deg \left(r(v)\right)\\
		&
		=
		\sum_{r(v) \in \hat{Z}^{\cN} } \deg\left(r(v)\right) \cdot \Big| \left\{i\in \{1,\ldots,K\} : r(v) \in \cN(r(v_i))\right\} \Big|
		\leq 
		3^d \sum_{r(v) \in \hat{Z}^{\cN} } \deg\left(r(v)\right) .
	\end{align}
	Next, we iteratively define a set $\mathbb{LI} = \mathbb{LI}(\hat{Z}) = \mathbb{LI}_K \subset \hat{Z}$ as follows:
	\begin{enumerate}\addtocounter{enumi}{-1}
		\item Start with $\mathbb{LI}_0 = \emptyset$.
		\item For $i=1,\ldots, K$: If $\deg^{\cN} \left(r(v_i)\right) \leq 9^d 50 \mu_{\beta+1}$ and $ \cN \left( r(v_i) \right)  \nsim \bigcup_{r(u) \in \mathbb{LI}_{i-1}} \cN(r(u))$, then set $\mathbb{LI}_{i} = \mathbb{LI}_{i-1} \cup \{r(v_i)\}$; else set $\mathbb{LI}_{i}= \mathbb{LI}_{i-1}$.
	\end{enumerate}
	On the event where $\overline{\deg}(\widetilde{Z}) \leq 20 \mu_{\beta + 1}$ for all sets $\widetilde{Z}\in \mathcal{CS}_{\geq K} \left(G^\prime\right) $, we have
	\begin{align*}
		\sum_{i=1}^K \deg^{\cN} \left(r(v_i)\right)  \overset{\eqref{overset}}{\leq} 3^d \sum_{r(v) \in \hat{Z}^{\cN} } \deg\left(r(v)\right) \leq 3^d 20 \mu_{\beta + 1} \left|\hat{Z}^{\cN}\right| \overset{\eqref{eq:ZN size}}{\leq} 9^d 20 \mu_{\beta + 1} K
	\end{align*}
	and thus there can be at most $\frac{K}{2}$ many indices $i\in \{1,\ldots,K\}$ in the above sum with $\deg^{\cN} \left(r(v_i)\right) > 9^d 50 \mu_{\beta+1}$. Vice versa, this implies that there are at least  $\frac{K}{2}$ many indices $i\in \{1,\ldots,K\}$ with $\deg^{\cN} \left(r(v_i)\right) \leq 9^d 50 \mu_\beta$. 
	Let $U$ be the set
	\begin{equation*}
		U=\left\{i\in \{1,\ldots,K\} : \deg^{\cN}(r(v_i)) \leq 9^d 50 \mu_{\beta+1}\right\} .
	\end{equation*}
	Then the above argument says that the event where $\overline{\deg}(\widetilde{Z}) \leq 20 \mu_{\beta + 1}$ for all sets $\widetilde{Z}\in \mathcal{CS}_{\geq K} \left(G^\prime\right)$ already implies $|U| \geq \frac{K}{2}$. Note that the event where $\overline{\deg}(\widetilde{Z}) \leq 20 \mu_{\beta+1}$ for all $\widetilde{Z}\in \mathcal{CS}_{\geq K} \left(G^\prime\right) $ is very likely for large $K$, by Corollary \ref{coro:connnectedsetsinLRP}. 
	For each $i\in U$ with $r(v_i) \notin \mathbb{LI}$, we have that
	\begin{equation*}
		\cN(r(v_i)) \sim \bigcup_{r(u) \in \mathbb{LI}} \cN(r(u)),
	\end{equation*}
	as otherwise, the vertex $r(v_i)$ would have been included into the set $\mathbb{LI}_i$. Using this, we get that
	\begin{equation}\label{eq:set ineq}
		U \subseteq \bigcup_{j:r(v_j) \in \mathbb{LI}} \Big( \{j\} \cup \left\{i \in \{1,\ldots,K\} : \cN(r(v_i)) \sim \cN(r(v_j)) \right\} \Big).
	\end{equation}
	For each $j\in \{1,\ldots,K\}$ with $r(v_j) \in \mathbb{LI}$, there can be at most $\deg^{\cN}(r(v_j)) \leq 9^d 50 \mu_{\beta+1}$ many vertices $r(w) \in G^\prime$ with $r(w) \sim \cN(r(v_j))$. Thus, there can be also at most $3^d \deg^{\cN}(r(v_j)) \leq 27^d 50 \mu_{\beta+1}$ many vertices $r(w) \in G^\prime$ with $\cN(r(w)) \sim \cN(r(v_j))$. This directly implies that 
	\begin{equation*}
		\left| \left\{i \in \{1,\ldots,K\} : \cN(r(v_i)) \sim \cN(r(v_j)) \right\} \right| \leq 27^d 50 \mu_{\beta+1}.
	\end{equation*}
	Taking the cardinality of the respective sets in \eqref{eq:set ineq}, we see that
	\begin{align*}
		\frac{K}{2} & \leq 
		\left|U\right| \leq \sum_{j:r(v_j) \in \mathbb{LI}} \Big( 1 + \big| \left\{i \in \{1,\ldots,K\} : \cN(r(v_i)) \sim \cN(r(v_j)) \right\} \big| \Big)\\
		&
		\leq
		\sum_{j:r(v_j) \in \mathbb{LI}}  \left(1 + 27^d 50 \mu_{\beta+1} \right) 
		=
		\left|\mathbb{LI}\right| \left(1 + 27^d 50 \mu_{\beta+1} \right)
	\end{align*}
	and solving this inequality for $|\mathbb{LI}|$ shows that
	\begin{equation}\label{eq:LI}
		|\mathbb{LI}| \geq \frac{K}{2 (27^d 50 \mu_{\beta+1} + 1)} \geq \frac{K}{27^d 200 \mu_{\beta+1}}  .
	\end{equation}
	If the block $V_w^{2^k}$ is not $\delta$-good, then at least one of the events of the form
	\begin{align*}
		&
		\mathcal{B}^\prime(\delta) \coloneqq D^\star \left( V_w^{2^k} , \bigcup_{u \in V_\mz^{2^{n-k}} : \| u - w \|_\infty \geq 2} V_u^{2^k} \right) < \delta 2^{k\dxp(\beta)} ,\text{ or } \\
		&
		\mathcal{B}_u^\prime(\delta) \coloneqq
		\bigcup_{\substack{ x,y \in V_w^{2^k} : \\ x ,y \sim V_u^{2^{k}} , x\neq y } } 
		\left\{D_{V_w^{2^{k}}} (x,y) < \delta 2^{k \dxp(\beta)} \right\} \text{ for } r(u) \text{ with } r(u) \sim r(w),  \|r(u)-r(w)\|_\infty \geq 2, \text{ or }\\
		& \mathcal{A}_{u,v}^\prime (\delta) \coloneqq \bigcup_{\substack{ x \in V_w^{2^k} : \\ x \sim V_u^{2^{k}}} } \ \bigcup_{\substack{ y \in V_w^{2^k} : \\ y \sim V_v^{2^{k}}   }  } \left\{D_{V_w^{2^{k}}}\left(x,y\right) < \delta 2^{k \dxp(\beta)} \right\} 
	\end{align*}
	for $r(u),r(v)$ with $r(u) \sim r(w) \sim r(v), \|r(u)-r(w)\|_\infty \geq 2, r(u)\neq r(v)\neq r(w)$
	needs to hold. 
	Conditioned on the graph $G^\prime$, and assuming that $\deg^{\cN} \left( r(w)\right) \leq 9^d 50 \mu_{\beta+1}$, the probability that the block $V_w^{2^k}$ is not $\delta$-good is thus bounded by
	\begin{align*}
		&\p_{\beta < k}^{\beta + \eps \geq k} \left( V_w^{2^k} \text{ not $\delta$-good} \Big| G^\prime \right)
		\leq \sum_{\substack{r(u),r(v): r(u) \sim r(w) \sim r(v), \\ \|r(u)-r(w)\|_\infty \geq 2, r(u)\neq r(v)\neq r(w)}}  \p_{\beta < k}^{\beta + \eps \geq k} \left(  \mathcal{A}_{u,v}^\prime (\delta) \big| G^\prime  \right)
		\\
		&
		\ \ \ \ \ \ \ \ \ \ \ \ \ \ \ \ \ \ \ \ \ \ \ \ \ \ \ \ \ \ \ \ \ \ \ \ \ \ \ \ 
		+
		\sum_{\substack{r(u): r(u) \sim r(w), \\ \|r(u)-r(w)\|_\infty \geq 2}}  \p_{\beta < k}^{\beta + \eps \geq k} \left(  \mathcal{B}_{u}^\prime (\delta)  \big| G^\prime \right)
		+
		\p_{\beta < k}^{\beta + \eps \geq k} \left(  \mathcal{B}^\prime (\delta) \big| G^\prime \right)
		\\
		&
		\leq
		\deg\left(r(w)\right)^2 \left(1-f_2(\delta)\right) +
		\deg^{\cN}\left( r(w) \right) \left(1-f_1(\delta)\right)
		+ (1-f_3(\delta))
		\leq 9^{2d} 5000 \mu_{\beta+1}^2 \left(1-f(\delta)\right) ,
	\end{align*}
	where $f$ was defined by $f(\delta)= \min\{f_1(\delta) , f_2(\delta) , f_3(\delta)\}$ and the functions $f_1,f_2$, and $f_3$ were defined in Lemmas \ref{lem:goodeventB}, \ref{lem:goodeventA}, and \ref{lem:goodeventD}.
	Remember that we chose $\delta>0$ small enough so that
	\begin{align}\label{ftilde}
		\tilde{f}(\delta) \coloneqq  9^{2d} 5000 \mu_{\beta+1}^2 \left(1-f(\delta)\right) \leq \left(32\mu_{\beta + 1}\right)^{-27^d 400\mu_{\beta+1}} .
	\end{align}
	Next, we argue that if $\hat{Z}$ and $G^\prime$ are such that $|\mathbb{LI}(\hat{Z})| \geq \frac{K}{27^d 200 \mu_{\beta+1}}$, then
	\begin{align}\label{eq:16}
		\p_{\beta < k}^{\beta + \eps \geq k} \left( \left|\left\{ r(u) \in \mathbb{LI}(\hat{Z}) : r(u) \text{ good} \right\}\right| \leq \frac{K}{27^d 400 \mu_{\beta + 1}} \Big| G^\prime \right) \leq \left(16 \mu_{\beta+1}\right)^{-K}.
	\end{align}
	Given the graph $G^\prime$, it is independent whether different vertices in $\mathbb{LI}(\hat{Z})$ are good or not, as we will argue now. For a vertex $r(u)$, it depends only on edges with at least one end point in the set $\bigcup_{r(v) \in \cN\left(r(u)\right)} V_v^{2^k}$ whether the vertex $r(u)$ is good or not. But for different vertices $r(u), r(u^\prime) \in \mathbb{LI}(\hat{Z})$ there are no edges with one end in $\bigcup_{r(v) \in \cN\left(r(u)\right)} V_v^{2^k}$ and the other end in $\bigcup_{r(v) \in \cN\left(r(u^\prime)\right)} V_v^{2^k}$, as $\cN\left(r(u)\right) \nsim \cN\left(r(u^\prime)\right)$. Thus, it is independent whether different vertices in $\mathbb{LI}(\hat{Z})$ are good. So in particular, for all connected sets $\hat{Z}\in \mathcal{CS}_K(G^\prime)$ with $\left|\mathbb{LI}(\hat{Z})\right| \geq \frac{K}{27^d 200 \mu_{\beta + 1}}$, the probability that at most $\frac{K}{27^d 400 \mu_{\beta + 1}}$ of the vertices in $\mathbb{LI}(\hat{Z})$ are good is bounded by
	\begin{align*}
		&
		\p_{\beta < k}^{\beta + \eps \geq k} \left( \left|\left\{ r(u) \in \mathbb{LI}(\hat{Z}) : r(u) \text{ good} \right\}\right| \leq \frac{K}{27^d 400 \mu_{\beta + 1}} \Big| G^\prime \right)
		\\
		& 
		\leq 
		\p_{\beta < k}^{\beta+\eps \geq k} \left( \left|\left\{r(u) \in \mathbb{LI}(\hat{Z}) : r(u) \text{ not good}\right\}\right| \geq \frac{\left|\mathbb{LI}(\hat{Z})\right|}{2}  \ \Big| G^\prime  \right)\\
		&
		\leq
		\sum_{H \subset \mathbb{LI}(\hat{Z}): |H| \geq \frac{|\mathbb{LI}(\hat{Z})|}{2}}
		\p_{\beta < k}^{\beta+\eps \geq k} \left( r(u) \text{ not good for all } r(u)\in H  \ \Big| G^\prime  \right)
		\\
		&
		\leq
		2^{|\mathbb{LI}(\hat{Z})|} \tilde{f}(\delta)^{\frac{|\mathbb{LI}(\hat{Z})|}{2}} \leq 2^K \tilde{f}(\delta)^{\frac{K}{27^d 400\mu_{\beta + 1}}} \overset{\eqref{ftilde}}{\leq} 2^K\left(32 \mu_{\beta+1}\right)^{-K} = \left(16 \mu_{\beta+1}\right)^{-K},
	\end{align*}
	which proves \eqref{eq:16}. Note that if $\left| \left\{r(v) \in \mathbb{LI}(\hat{Z}) : r(v) \text{ good}\right\} \right| \geq \frac{K}{27^d 400 \mu_{\beta+1}}$, then the set $Z$ also automatically contains $\frac{K}{27^d 400 \mu_{\beta+1}}$ many separated good vertices. The separation property directly follows from the construction, as we assumed that $\cN \left(r(u)\right) \nsim \cN \left(r(v)\right)$ for all distinct $r(u), r(v) \in \mathbb{LI}(\hat{Z})$. Further, note that if $\overline{\deg}(\widetilde{Z}) \leq 20 \mu_{\beta + 1}$ for all sets $\widetilde{Z}\in \mathcal{CS}_{\geq K} \left(G^\prime\right)$, then $|\mathbb{LI}(\hat{Z})| \geq \frac{K}{27^d 200 \mu_{\beta+1}}$ for all sets $\hat{Z}=\left\{r(v_1),\ldots,r(v_K)\right\}\in \mathcal{CS}_K(G^\prime)$, by \eqref{eq:LI}. This implies that inequality \eqref{eq:16} holds for all fixed sets $\hat{Z} \in \mathcal{CS}_K(G^\prime)$, under the assumption that $\overline{\deg}(\widetilde{Z})\leq 20 \mu_{\beta +1}$ for all sets $\widetilde{Z} \in \mathcal{CS}_{\geq K}(G^\prime)$.
	\\
	
	Next, we want to translate such a bound from one connected set to all connected sets simultaneously. Define the event
	\begin{equation*}
		\mathcal{T}_K\coloneqq \left\{\exists Z \in \mathcal{CS}_K\left(G^\prime\right) : \left| \left\{r(v) \in \mathbb{LI}(Z) : r(v) \text{ good}\right\} \right| \leq \frac{K}{27^d 400 \mu_{\beta+1}}\right\} .
	\end{equation*}
	If the event $\mathcal{T}_K$ does not hold, then every set $Z \in \mathcal{CS}_K\left(G^\prime\right) $ needs to contain at least  $\frac{K}{27^d 400 \mu_{\beta + 1}}$ separated good vertices. Conditioned on the event $\overline{\deg}(\widetilde{Z}) \leq 20\mu_{\beta + 1} \ \forall  \widetilde{Z} \in \mathcal{CS}_K\left(G^\prime\right), \left|\mathcal{CS}_K(G^\prime) \right| \leq 8^K \mu_{\beta + 1}^K$, we get via a union bound over the at most $8^K \mu_{\beta + 1}^K$ many sets $Z \in \mathcal{CS}_K(G^\prime)$ that
	\begin{align}\label{conditionTK}
		& \notag \p_{\beta < k}^{\beta+\eps \geq k} \left( \mathcal{T}_K \Big| \overline{\deg}(\widetilde{Z}) \leq 20\mu_{\beta + 1} \ \forall  \widetilde{Z} \in \mathcal{CS}_K\left(G^\prime\right), \left|\mathcal{CS}_K(G^\prime) \right| \leq 8^K \mu_{\beta + 1}^K \right)
		\\
		&
		\overset{\eqref{eq:16}}{\leq} 8^K \mu_{\beta + 1}^K  \left(16 \mu_{\beta + 1}\right)^{-K} = 2^{-K}.
	\end{align}
	Another union bound implies that
	\begin{align*}
		&\notag \p_{\beta < k}^{\beta+\eps \geq k} \left( \exists Z \in \mathcal{CS}_K\left(G^\prime\right) \text{ with less than $\frac{K}{27^d 400 \mu_{\beta + 1}}$ separated good vertices}  \right)\\
		&\notag \leq \p_{\beta < k}^{\beta+\eps\geq k} \left(\exists Z \in \mathcal{CS}_K\left(G^\prime\right) : \left| \left\{r(v) \in \mathbb{LI}(Z) : r(v) \text{ good}\right\} \right| \leq \frac{K}{27^d 400 \mu_{\beta+1}}\right)
		=
		\p_{\beta < k}^{\beta+\eps \geq k} \left(\mathcal{T}_K\right) \\
		&\notag \leq \p_{\beta < k}^{\beta+\eps \geq k} \left( \exists \widetilde{Z} \in \mathcal{CS}_K\left(G^\prime\right) :  \overline{\deg}(\widetilde{Z}) > 20\mu_{\beta + 1} \right) + \p_{\beta < k}^{\beta+\eps \geq k} \left(\left|\mathcal{CS}_K(G^\prime) \right|> 8^K \mu_{\beta + 1}^K \right) \\
		& \ \
		+  \p_{\beta < k}^{\beta+\eps \geq k} \left( \mathcal{T}_K \Big| \overline{\deg}(\widetilde{Z}) \leq 20\mu_{\beta + 1} \ \forall  \widetilde{Z} \in \mathcal{CS}_K\left(G^\prime\right), \left|\mathcal{CS}_K(G^\prime) \right| \leq 8^K \mu_{\beta + 1}^K \right)\\
		&
		\overset{\eqref{coroeq48},\eqref{conditionTK}}{\leq} 2^{-K} + \frac{\E_{\beta \leq k}^{\beta+\eps>k} \left[\left|\mathcal{CS}_K(G^\prime) \right| \right]}{8^K \mu_{\beta + 1}^K} +  2^{-K} 
		\overset{\eqref{coroeq482}}{\leq} 3 \cdot 2^{-K} ,
	\end{align*}
	which finishes the proof.
\end{proof}

Again, we write $G^\prime$ for the graph where start with long-range percolation on $\Z^d$ or $V_\mz^{2^n}$ and then contract all boxes of the form $V_v^{2^k}$ $($for $v\in \Z^d$ or $v\in V_\mz^{2^{n-k}})$ into vertices named $r(v)$. So each vertex $r(v)$ in $G^\prime$ corresponds to the set $V_v^{2^k}$.  Next, we investigate the sum of diameters in connected sets. 

\begin{lemma}
	There exists a constant $1 < C^\prime <\infty$ such that for all $\eps \in \left[0,1\right]$
	\begin{align}\label{eq:connectedset diabound}
		\p_{\beta < k}^{\beta+\eps \geq k} \left(\exists Z \in \mathcal{CS}_K\left(G^\prime\right) : \sum_{r(v) \in Z } \dia(V_v^{2^k};\omega) > C^\prime |Z| 2^{k\dxp(\beta)} \right) \leq 2^{-K}.
	\end{align}
\end{lemma}

\begin{proof}
	Let $Z$ be a fixed connected set in $G^\prime$. Under the measure $\p_{\beta < k}^{\beta+\eps \geq k}$, the diameter of the box $V_v^{2^k}$ corresponding to some vertex $r(v)\in G^\prime$ always has the same distribution, not depending on $\eps$. This holds because the diameter only depends on edges with both endpoints inside $V_v^{2^k}$, and such edges have length at most $2^k-1$. For a fixed set of vertices $Z\subset G^\prime$, the random variables $\left(\dia \left( V_v^{2^k} ; \omega \right)\right)_{r(v)\in Z}$ are also independent. Further, the random variables $ \frac{\dia\left(V_\mz^{2^k};\omega \right)}{ 2^{k\dxp(\beta)} }$ have uniform exponential moments, meaning that
	\begin{equation*}
		\sup_{k\in \N} \E_\beta \left[\exp\left(  \frac{\dia\left(V_\mz^{2^k};\omega\right)}{ 2^{k\dxp(\beta)} }\right)\right] < \infty ,
	\end{equation*}
	which directly follows from Lemma \ref{lem:moments for diameter}.
	Markov's inequality implies that
	\begin{align}\label{eq:8mubeta}
		\notag \p_{\beta} & \left(\sum_{r(v) \in Z } \dia\left(V_v^{2^k};\omega\right) > C^\prime |Z| 2^{k\dxp(\beta)} \right)
		= \p_{\beta } \left( e^{  \sum_{r(v) \in Z } \frac{\dia \left( V_v^{2^k} ; \omega \right)}{ 2^{k\dxp(\beta)} }  } > e^{C^\prime  |Z|}\right) \\
		& \leq \E_\beta \left[e^{  \frac{\dia \left( V_\mz^{2^k} ; \omega \right)}{ 2^{k\dxp(\beta)} }}\right]^{|Z|}e^{-C^\prime  |Z|} \leq \left(8 \mu_{\beta + 1}\right)^{-|Z|}
	\end{align}
	for $C^\prime$ large enough. The diameter of the vertices $r(v) \in G^\prime$ only depends on edges within one block $V_v^{2^k}$. Contrary, the graph $G^\prime$ and events of the form $\{Z\in \mathcal{CS}_K(G^\prime)\}$ only depend on edges between two different blocks $V_v^{2^k}, V_w^{2^k}$ with $v\neq w$. So in particular the events $\{Z\in \mathcal{CS}_K(G^\prime)\}$ and $\left\{ \sum_{r(v) \in Z } \dia \left( V_v^{2^k} ; \omega \right) > C^\prime |Z| 2^{k\dxp(\beta)} \right\}$ are independent for all sets $Z\subset G^\prime$. Thus we get that
	\begin{align*}
		& \p_{\beta < k}^{\beta+\eps \geq k} \left(\exists Z \in \mathcal{CS}_K\left(G^\prime\right) : \sum_{r(v) \in Z } \dia \left( V_v^{2^k} ; \omega \right) > C^\prime |Z| 2^{k\dxp(\beta)} \right)\\
		&
		\leq
		\sum_{Z\subset G^\prime} \p_{\beta < k}^{\beta+\eps \geq k} \left(  Z \in \mathcal{CS}_K(G^\prime) \ , \ \sum_{r(v) \in Z } \dia \left( V_v^{2^k} ; \omega \right) > C^\prime |Z| 2^{k\dxp(\beta)} \right)\\
		&
		=
		\sum_{Z\subset G^\prime} \p_{\beta < k}^{\beta+\eps \geq k} \left(  Z \in \mathcal{CS}_K(G^\prime) \right) \p_{\beta < k}^{\beta+\eps \geq k} \left( \sum_{r(v) \in Z } \dia \left( V_v^{2^k} ; \omega \right) > C^\prime |Z| 2^{k\dxp(\beta)} \right)\\
		&
		\overset{\eqref{eq:8mubeta}}{\leq}
		\sum_{Z\subset G^\prime} \p_{\beta < k}^{\beta+\eps \geq k} \left(  Z \in \mathcal{CS}_K(G^\prime) \right)  \left(8 \mu_{\beta + 1}\right)^{-|Z|}
		= \E_{\beta + \eps } \left[ \left|\mathcal{CS}_K \left(G^\prime \right) \right| \right] \left(8 \mu_{\beta + 1}\right)^{-K} \overset{\eqref{coroeq482}}{\leq} 2^{-K}.
	\end{align*}
\end{proof}

In the next lemma, we consider the sum of diameters of boxes that are adjacent to edges defined by $\chi$.

\begin{lemma}
	For a connected set $Z \subset G^\prime$, we write $Z_\chi$ for the set of vertices $r(w) \in Z$ for which there exists an edge $e=\{a,b\}$ with $a \in V_w^{2^k}$ and $\chi(e)=1$. Then for all $\eps> 0$ small enough
	\begin{align}
		\label{set z chi small} &
		\p_{\beta < k}^{\beta+\eps \geq k}  \left(\exists Z \in \mathcal{CS}_K \left( G^\prime \right) : |Z_\chi| >\frac{ 16 \mu_{\beta + 1}  }{\log(1/\eps)} K \right) \leq 2^{-K} , \quad \text{ and } \\
		\label{eq:D bound}
		& \p_{\beta < k}^{\beta+\eps \geq k} \left( \exists Z \in \mathcal{CS}_K(G^\prime) : \sum_{r(v) \in Z_\chi} \dia\left(V_v^{2^k} ;\omega \right) > r(\eps) 2^{k\dxp(\beta)} K \right)
		\leq 3\cdot 2^{-K}.
	\end{align}
	where $r(\eps) = \log(1/\eps)^{-\frac{1}{2C_1}}$, and $C_1$ is the constant from Lemma \ref{lem:moments for diameter}.
\end{lemma}

\begin{proof}
	We start the proof by bounding the size $|Z_\chi|$.
	When introducing the edges defined by $\chi$ into the set $V_\mz^{2^n}$, the structure of the graph $G^\prime$, in which we contracted blocks of side length $2^k$, does not change, as the edges $\{a,b\}$ with $\chi(\{a,b\})=1$ are either inside one block, i.e., $\{a,b\}\subset V_w^{2^k}$ or between neighboring blocks, i.e., $a \in V_w^{2^k}, b \in V_u^{2^k}$ with $\|u-v\|_\infty = 1$. The probability that a block $V_w^{2^k}$ is adjacent to an edge in $\omega^\chi$ that did not exist in the percolation environment $\omega$ is bounded by $2^{5d} \eps$, see \eqref{eq:coupling bound}. 
	Let $\preceq$ be the lexicographic ordering on $\Z^d$, where we write $\prec$ for strict inequalities.
	For a fixed connected set $Z\in \mathcal{CS}_K (G^\prime)$, define the set $Z_\chi^{\ell}$ as the set of vertices $r(w) \in Z_\chi$ for which an edge $e$ with $\chi(e)=1$ exists, so that $e$ has both endpoints in $V_w^{2^k}$, or one endpoint is in $V_w^{2^k}$ and one endpoint is in $V_{u}^{2^k}$ with $r(u) \notin Z$, or one endpoint in $V_w^{2^k}$ and one in $V_{u}^{2^k}$ with $w \prec u$ and $r(u) \in Z$. The probability that $r(u) \in Z_\chi^{\ell}$ is upper bounded by
	\begin{equation*}
		\p_{\beta < k}^{\beta+\eps \geq k} \left(r(u) \in Z_\chi^{\ell}\right) \leq \p_{\beta < k}^{\beta+\eps \geq k} \left(\exists v \in V_u^{2^k}, w \in \Z^d \text{ with } \chi\left(\{v,w\}\right)=1 \right)
		\overset{\eqref{eq:coupling bound}}{\leq} 2^{5d} \eps \text .
	\end{equation*}
	Further, for different vertices $r(u) \in Z$, it is independent whether they are in the set $Z_\chi^{\ell}$ or not.
	Hence the size of the set $Z_\chi^{\ell}$ is stochastically dominated by $\sum_{i=1}^{K} X_i$, where $X_i$ are independent Bernoulli-distributed random variables with parameter $2^{5d} \eps$. Furthermore, one has $|Z_\chi| \leq 2 |Z_\chi^{\ell}|$, as each edge $e$ with  $\chi(e) = 1$, that creates a vertex in $Z_\chi^{\ell}$, can add at most two vertices to $Z_\chi$. So in particular, for each fixed set $Z\subset G^\prime$, we get that
	\begin{align}\label{eq:bernoullibound}
		\notag & \p_{\beta < k}^{\beta+\eps \geq k}  \left( |Z_\chi| >\frac{ 16 \mu_{\beta + 1}  }{\log(1/\eps)} K \right)
		\leq
		\p_{\beta < k}^{\beta+\eps \geq k}  \left( |Z_\chi^{\ell}| >\frac{\mu_{\beta + 1} 8 }{\log(1/\eps)} K \right) \leq \p\left(  \sum_{i=1}^K X_i > \frac{\mu_{\beta + 1} 8 }{\log(1/\eps)}K \right)\\
		&
		\leq
		\sum_{\substack{U\subset \{1,\ldots,K\}: \\  |U|\geq \frac{\mu_{\beta + 1} 8 }{\log(1/\eps)}K }} \p \left(X_i = 1 \text{ for all } i\in U\right)
		\leq
		\sum_{\substack{U\subset \{1,\ldots,K\}: \\  |U|\geq \frac{\mu_{\beta + 1} 8 }{\log(1/\eps)}K }} \p(X_1=1)^{|U|}
		\leq
		2^K \left(2^{5d} \eps\right)^{\frac{\mu_{\beta + 1} 8 }{\log(1/\eps)}K } .
	\end{align}
	As $\chi$ can only be positive on pairs of vertices $\{a,b\}$ in the same block $V_w^{2^k}$ or in neighboring blocks $V_u^{2^k}, V_w^{2^k}$ with $\|u-w\|_\infty = 1$, we see that the graph $G^\prime$ and the random set $Z_\chi$ are independent. A union bound over all feasible sets $Z\subset G^\prime$ shows that
	\begin{align*}
		\notag & \p_{\beta < k}^{\beta+\eps \geq k}  \left(\exists Z \in \mathcal{CS}_K \left( G^\prime \right) : |Z_\chi| >\frac{ 16 \mu_{\beta + 1}  }{\log(1/\eps)} K \right)\\
		\notag &
		\leq
		\sum_{Z\subset G^\prime: |Z|=K} \p_{\beta < k}^{\beta+\eps \geq k}  \left( Z \in \mathcal{CS}_K \left( G^\prime \right) , |Z_\chi| >\frac{ 16 \mu_{\beta + 1}  }{\log(1/\eps)} K \right)\\
		\notag &
		=
		\sum_{Z\subset G^\prime: |Z|=K} \p_{\beta < k}^{\beta+\eps \geq k}  \left( Z \in \mathcal{CS}_K \left( G^\prime \right) \right)   \p_{\beta < k}^{\beta+\eps \geq k} \left( |Z_\chi| >\frac{ 16 \mu_{\beta + 1}  }{\log(1/\eps)} K \right)\\
		\notag & 
		\overset{\eqref{eq:bernoullibound}}{\leq} \E_{\beta \leq k}^{\beta + \eps > k} \left[ \left| \mathcal{CS}_K \left( G^\prime \right) \right| \right] 
		2^K \left(2^{5d} \eps\right)^{\frac{\mu_{\beta + 1} 8 }{\log(1/\eps)}K } 
		\leq
		\E_{\beta+\eps} \left[ \left| \mathcal{CS}_K \left( \Z^d \right) \right| \right] 
		2^K \left(2^{5d} \eps\right)^{\frac{\mu_{\beta + 1} 8 }{\log(1/\eps)}K } \\
		&
		\overset{\eqref{coroeq482}}{\leq} 
		4^K \mu_{\beta + 1}^K 2^K \left( 2^{5d} \eps \right)^{ \frac{\mu_{\beta + 1} 8 }{\log(1/\eps)}K}
		= \left( 2^{5d} \right)^{ \frac{\mu_{\beta + 1} 8 }{\log(1/\eps)}K} 8^K \mu_{\beta + 1}^K e^{ - \mu_{\beta + 1} 8  K} 
		\leq 2^{-K}
	\end{align*}
	where the last inequality holds for small enough $\eps$. This proves \eqref{set z chi small}.
	Next, let us see how the sums of the diameters of vertices in the sets $Z_\chi$ grow. The constant $C_1$ was chosen in Lemma \ref{lem:moments for diameter} such that $\E_\beta \left[\exp \left( s \tfrac{\dia\left(V_\mz^{2^k}\right)}{2^{k\dxp(\beta)}} \right) \right] < e^{C_1 s^{C_1}}$ for all $s\geq 1$ and $k \in \N$. We define the function 
	\begin{equation*}
		s(\eps) \coloneqq \left(\frac{\log(1/\eps)}{  16 \mu_{\beta + 1}  C_1 }\right)^{\frac{1}{C_1}} .
	\end{equation*}
	Remember that the function $r(\eps)$ was defined by $r(\eps) = \log(1/\eps)^{-\frac{1}{2C_1}}$. Let $Z^\prime \subset G^\prime$ be a fixed set of size at most $\frac{ 16 \mu_{\beta + 1}  }{\log(1/\eps)}K$. Then we have for all small enough $\eps$ that
	\begin{align}\label{16mu}
		& \notag \p_{\beta<k}^{\beta+\eps\geq k} \left( \sum_{r(v) \in Z^\prime } \dia\left( V_v^{2^k} ; \omega \right) > r(\eps) 2^{k\dxp(\beta)} K \right) 
		\\
		& \notag
		= \p_{\beta<k}^{\beta+\eps\geq k} \left( \exp \left(s(\eps) \sum_{r(v) \in Z^\prime} \frac{\dia\left(V_v^{2^k} ;\omega \right)}{2^{k\dxp( \beta)}} \right) > \exp \left( s(\eps)r(\eps) K \right) \right)\\
		& \notag \leq \E_\beta \left[ \exp \left(s(\eps) \frac{\dia\left( V_\mz^{2^k} \right)}{2^{k \dxp(\beta)}}\right) \right]^{|Z^\prime|} e^{-s(\eps) r(\eps)K} 
		\leq
		e^{\frac{\mu_{\beta + 1}  16 K}{\log(1/\eps)} C_1 s(\eps)^{C_1} }
		e^{-s(\eps) r(\eps) K }
		\\
		& = e^{ K}  \exp \left( - \frac{\log(1/\eps)^{\frac{1}{2C_1}}}{\left( 16 \mu_{\beta + 1}  C_1\right)^{\frac{1}{C_1}}} K \right)
		\leq \left(16\mu_{\beta + 1}\right)^{-K}
	\end{align}
	where the last inequality holds for $\eps$ small enough. Denote by $\mathcal{S}_K$ the event
	\begin{equation*}
		\mathcal{S}_K= \left\{ \exists Z \in \mathcal{CS}_K(G^\prime) : \sum_{r(v) \in Z_\chi} \dia\left(V_v^{2^k}; \omega \right) > r(\eps) 2^{k\dxp(\beta)} K \right\}
	\end{equation*}
	As the random variables $\left(\dia \left(V_v^{2^k}; \omega\right) \right)_{v\in \Z^d}$, the sets $Z_\chi$, and the connections inside the graph $G^\prime$ are independent, we get that for $\eps >0$ small enough
	\begin{align*}
		&\notag  \p_{\beta < k}^{\beta+\eps \geq k} \left( \exists Z \in \mathcal{CS}_K(G^\prime) : \sum_{r(v) \in Z_\chi} \dia\left(V_v^{2^k}; \omega \right) > r(\eps) 2^{k\dxp(\beta)} K \right) 
		=
		\p_{\beta < k}^{\beta+\eps \geq k} \left( \mathcal{S}_K \right)
		\\
		&\notag 
		\leq
		\p_{\beta < k}^{\beta+\eps \geq k} \left(\mathcal{S}_K \Big|  |Z_\chi| \leq \frac{ 16 \mu_{\beta + 1} }{\log(1/\eps)} K \text{ for all } Z \in \mathcal{CS}_K(G^\prime) ,  \left|\mathcal{CS}_K(G^\prime)\right| \leq 8^K \mu_{\beta + 1}^K  \right) 	\\
		&
		\ \
		+ \p_{\beta < k}^{\beta+\eps \geq k} \left(  \left|\mathcal{CS}_K(G^\prime)\right| > 8^K \mu_{\beta + 1}^K \right) 
		+ 
		\p_{\beta < k}^{\beta+\eps \geq k} \left( \exists Z \in \mathcal{CS}_K(G^\prime) : |Z_\chi| > \frac{ 16 \mu_{\beta + 1} }{\log(1/\eps)} K \right) 
		\\
		&
		\overset{\eqref{set z chi small}}{\leq}
		\p_{\beta < k}^{\beta+\eps \geq k} \left(\mathcal{S}_K \Big|  |Z_\chi| \leq \frac{ 16 \mu_{\beta + 1} }{\log(1/\eps)} K \text{ for all } Z \in \mathcal{CS}_K(G^\prime) ,  \left|\mathcal{CS}_K(G^\prime)\right| \leq 8^K \mu_{\beta + 1}^K  \right) 
		\\
		&
		\ \ \ \ +
		\frac{\E_{\beta \leq k}^{\beta + \eps > k} \left[  \left|\mathcal{CS}_K(G^\prime)\right|  \right]}{ 8^K \mu_{\beta + 1}^K} + 2^{-K}	
		\overset{\eqref{16mu}}{\leq}
		8^K \mu_{\beta+1}^K \left(16 \mu_{\beta+1}\right)^{-K}
		+
		\frac{\E_{\beta +\eps} \left[  \left|\mathcal{CS}_K(Z^d)\right|  \right]}{ 8^K \mu_{\beta + 1}^K} + 2^{-K}	
		\\
		&
		\overset{\eqref{coroeq482}}{\leq} 3\cdot 2^{-K}.
	\end{align*}
\end{proof}

The main ingredient for the proof of Theorem \ref{theo:continuity} is the following lemma, which allows us to make comparisons between the measures $\p_{\beta < k}^{\beta+\eps \geq k}$ and $\p_{\beta < k-1}^{\beta+\eps \geq k-1}$

\begin{lemma}\label{lem:eps h functions}
	There exist a constant $C<\infty$ and functions $r:(0,1) \to \R_{\geq 0}$ and $h : \N \to \R_{\geq 0}$ with $\lim_{\eps \searrow 0} r(\eps) = 0$ and $\lim_{n \to \infty} h(n)=0$ such that for all large enough $k \in \N$
	\begin{align}\label{eq:eps h functions}
		&\notag \E_{\beta < k}^{\beta+\eps \geq k} \left[D_{V_\mz^{2^n}}\left(\mz,(2^n-1)\mo \right)\right]-\E_{\beta < k-1}^{\beta+\eps \geq k-1} \left[D_{V_\mz^{2^n}}\left(\mz,(2^n-1)\mo \right)\right] \\
		& \leq C \left(r(\eps)+h(n-k)\right) \Lambda \left(2^k, \beta\right) \Lambda \left(2^{n-k}, \beta+\eps \right) .
	\end{align}
\end{lemma}

Let us first see how this lemma implies Theorem \ref{theo:continuity}.

\begin{proof}[Proof of Theorem \ref{theo:continuity} given Lemma \ref{lem:eps h functions}]
	We want to show that for all $\beta \geq 0$, the difference $\dxp(\beta) - \dxp(\beta + \eps )$ converges to $0$ as $\eps \to 0$. In Lemma \ref{lem:leftcontinuity}, we already showed that the function $\dxp(\cdot)$ is continuous from the left, so it suffices to consider $\eps > 0$ now. We have also seen in \eqref{eq:cont:sum} that
	\begin{align}\label{eq:finalproof1}
		\dxp(\beta) - \dxp(\beta + \eps )  = \frac{1}{\log(2)} \lim_{n\to \infty} \frac{1}{ n} \sum_{k=2}^n \log\left(\frac{\E_{\beta < k}^{\beta+\eps \geq k} \left[D_{V_\mz^{2^n}} \left(\mz , (2^n-1)\mo\right)\right]}{\E_{\beta < k-1}^{\beta+\eps \geq k-1} \left[D_{V_\mz^{2^n}} \left(\mz, (2^n-1)\mo\right)\right]}\right)  \text .
	\end{align}
	Each of the summands in \eqref{eq:finalproof1} is bounded, since
	\begin{align}\label{bounded}
		\notag 1 & \leq \frac{\E_{\beta < k}^{\beta+\eps \geq k} \left[D_{V_\mz^{2^n}} \left(\mz , (2^n-1)\mo\right)\right]}{\E_{\beta < k-1}^{\beta+\eps \geq k-1} \left[D_{V_\mz^{2^n}} \left(\mz, (2^n-1)\mo\right)\right]}
		\overset{\eqref{trivia2},\eqref{eq:uniformovereps}}{\leq}
		\frac{C(\beta) \Lambda\left(2^k,\beta\right) \Lambda\left(2^{n-k},\beta+\eps\right)}{c_\beta \Lambda\left(2^{k-1},\beta\right) \Lambda\left(2^{n-k+1},\beta+\eps\right) }
		\\
		&
		\overset{\eqref{2-comparison}}{\leq} 2
		\frac{C(\beta) \Lambda\left(2^{k-1},\beta\right) \Lambda\left(2^{n-k},\beta+\eps\right)}{c_\beta \Lambda\left(2^{k-1},\beta\right) \Lambda\left(2^{n-k},\beta+\eps\right) } = \frac{2 C(\beta)}{c_\beta},
	\end{align}
	where $c_\beta, C(\beta)$ are the constants from Lemma \ref{lem:uniformovereps}, and where we used the inequalities $ \Lambda\left(2^{k},\beta\right) \leq   2\Lambda\left(2^{k-1},\beta\right)$ and $ \Lambda\left(2^{n-k+1},\beta+\eps\right)\geq  \Lambda\left(2^{n-k},\beta+\eps\right)$, which are proven in Lemma \ref{lem:auxil} below.\\
	
	Lemma \ref{lem:uniformovereps} and Lemma \ref{lem:auxil} below imply that for all $\eps \in (0,1)$ and all $k,n \in \N$ with $k\leq n$ one has
	\begin{align}
		& \notag \E_{\beta < k-1}^{\beta+\eps \geq k-1} \left[D_{V_\mz^{2^n}}  \left(\mz,(2^n-1)\mo  \right)\right]  \overset{\eqref{eq:uniformovereps}}{\geq} c_\beta  \Lambda\left(2^{k-1},\beta\right)  \Lambda\left(2^{n-k+1},\beta + \eps\right)
		\\
		& \label{eq:denominator}
		\overset{\eqref{2-comparison}}{\geq} \frac{c_\beta}{2}  \Lambda\left(2^{k},\beta\right)  \Lambda\left(2^{n-k+1},\beta + \eps\right)
		\overset{\eqref{2-comparison}}{\geq} \frac{c_\beta}{2}  \Lambda\left(2^{k},\beta\right)  \Lambda\left(2^{n-k},\beta + \eps\right) .
	\end{align}
	
	Together with Lemma \ref{lem:eps h functions}, this implies that for all large enough $k \in \N$ and $\eps \in (0,1)$
	\begin{align}\label{eq:fraction}
		& \notag \frac{\E_{\beta < k}^{\beta+\eps \geq k} \left[ D_{V_\mz^{2^n}}  \left(\mz,(2^n-1)\mo\right)\right]
			-
			\E_{\beta < k-1}^{\beta+\eps \geq k-1} \left[ D_{V_\mz^{2^n}} \left(\mz,(2^n-1)\mo  \right)\right]}{\E_{\beta < k-1}^{\beta+\eps \geq k-1} \left[D_{V_\mz^{2^n}}  \left(\mz,(2^n-1)\mo  \right)\right]}
		\\
		&\notag 
		\overset{\eqref{eq:eps h functions}}{\leq}
		\frac{ C \left(r(\eps)+h(n-k)\right) \Lambda \left(2^k, \beta\right) \Lambda \left(2^{n-k}, \beta+\eps \right)}{\E_{\beta < k-1}^{\beta+\eps \geq k-1} \left[D_{V_\mz^{2^n}}  \left(\mz,(2^n-1)\mo  \right)\right]}
		\\
		& 
		\overset{\eqref{eq:denominator}}{\leq}
		\frac{ 2 C \left(r(\eps)+h(n-k)\right) \Lambda \left(2^k, \beta\right) \Lambda \left(2^{n-k}, \beta+\eps \right)}{c_\beta  \Lambda\left(2^{k},\beta\right)  \Lambda\left(2^{n-k},\beta + \eps\right)} =  2 C \left(r(\eps)+h(n-k)\right) .
	\end{align}
	Applying the logarithm, we get that for all large enough $k \in \N$, say for $k\geq K$, one has
	\begin{align}
		&\notag \log\left(\frac{\E_{\beta < k}^{\beta+\eps \geq k} \left[ D_{V_\mz^{2^{n}}} \left(\mz,(2^n-1)\mo \right)\right]}{\E_{\beta < k-1}^{\beta+\eps \geq k-1} \left[D_{V_\mz^{2^{n}}} \left(\mz,(2^n-1)\mo \right)\right]}\right)
		\\
		& \notag
		= \log\left(1 + \frac{\E_{\beta < k}^{\beta+\eps \geq k} \left[ D_{V_\mz^{2^{n}}} \left(\mz,(2^n-1)\mo \right)\right] - \E_{\beta < k-1}^{\beta+\eps \geq k-1} \left[D_{V_\mz^{2^{n}}} \left(\mz,(2^n-1)\mo \right)\right]}{\E_{\beta < k-1}^{\beta+\eps \geq k-1} \left[D_{V_\mz^{2^{n}}} \left(\mz,(2^n-1)\mo \right)\right]}\right)
		\\
		& \label{eq:log}
		\overset{\eqref{eq:fraction}}{\leq}
		\log\left(1 +  2 C \left(r(\eps)+h(n-k)\right) \right) \leq  2 C \left(r(\eps)+h(n-k)\right) .
	\end{align}
	For fixed $\delta > 0$, take $\eps^\prime \in (0,1)$ and $K^\prime \in \N$ such that $ 2 C \left(r(\eps)+h(n-k)\right) < \delta \log(2)$ for all $\eps \in \left(0,\eps^\prime\right]$ and $n-k \geq K^\prime$. Then for all $\eps \in \left(0,\eps^\prime\right]$ one has
	\begin{align}
		&\notag \dxp(\beta) - \dxp(\beta + \eps )  =  \frac{1}{\log(2)} \lim_{n\to \infty} \frac{1}{ n} \sum_{k=2}^n \log\left(\frac{\E_{\beta < k}^{\beta+\eps \geq k} \left[D_{V_\mz^{2^{n}}}\left(\mz,(2^n-1)\mo \right)\right]}{\E_{\beta < k-1}^{\beta+\eps \geq k-1} \left[D_{V_\mz^{2^{n}}}\left(\mz,(2^n-1)\mo \right)\right]}\right) 
		\\
		& \label{1st term}
		\leq  \frac{1}{\log(2)} \limsup_{n\to \infty} \frac{1}{ n} \sum_{k=K}^{n-K^\prime} \log\left(\frac{\E_{\beta < k}^{\beta+\eps \geq k} \left[D_{V_\mz^{2^{n}}}\left(\mz,(2^n-1)\mo \right)\right]}{\E_{\beta < k-1}^{\beta+\eps \geq k-1} \left[D_{V_\mz^{2^{n}}}\left(\mz,(2^n-1)\mo \right)\right]}\right) 
		\\
		& \label{2nd term}
		+
		\frac{1}{\log(2)} \limsup_{n\to \infty} \frac{1}{ n} \sum_{k \in \{2,\ldots,n\}\setminus\{K,\ldots,n-K^\prime\}} \log\left(\frac{\E_{\beta < k}^{\beta+\eps \geq k} \left[D_{V_\mz^{2^{n}}}\left(\mz,(2^n-1)\mo \right)\right]}{\E_{\beta < k-1}^{\beta+\eps \geq k-1} \left[D_{V_\mz^{2^{n}}}\left(\mz,(2^n-1)\mo \right)\right]}\right) .
	\end{align}
	For the term in line \eqref{1st term}, each of the summands is bounded by $ 2 C \left(r(\eps)+h(n-k)\right) < \delta \log(2)$ by \eqref{eq:log}, so that
	\begin{align}
		& \notag \frac{1}{\log(2)} \limsup_{n\to \infty} \frac{1}{ n} \sum_{k=K}^{n-K^\prime} \log\left(\frac{\E_{\beta < k}^{\beta+\eps \geq k} \left[D_{V_\mz^{2^{n}}}\left(\mz,(2^n-1)\mo \right)\right]}{\E_{\beta < k-1}^{\beta+\eps \geq k-1} \left[D_{V_\mz^{2^{n}}}\left(\mz,(2^n-1)\mo \right)\right]}\right)
		\\
		& \label{1st term bound}
		\leq \frac{1}{\log(2)} \limsup_{n\to \infty} \frac{1}{ n} \sum_{k=K}^{n-K^\prime} 2 C \left(r(\eps)+h(n-k)\right)
		\leq  \frac{1}{\log(2)} \limsup_{n\to \infty} \frac{1}{ n} \sum_{k=K}^{n-K^\prime} \delta \log(2) = \delta .
	\end{align}
	Fort the term in line \eqref{2nd term}, we use that all the summands are bounded by \eqref{bounded}, and thus
	\begin{align}
		& \notag \frac{1}{\log(2)} \limsup_{n\to \infty} \frac{1}{ n} \sum_{k \in \{2,\ldots,n\}\setminus\{K,\ldots,n-K^\prime\}} \log\left(\frac{\E_{\beta < k}^{\beta+\eps \geq k} \left[D_{V_\mz^{2^{n}}}\left(\mz,(2^n-1)\mo \right)\right]}{\E_{\beta < k-1}^{\beta+\eps \geq k-1} \left[D_{V_\mz^{2^{n}}}\left(\mz,(2^n-1)\mo \right)\right]}\right) 
		\\
		& \label{2nd term bound}
		\overset{\eqref{bounded}}{\leq} \frac{1}{\log(2)} \limsup_{n\to \infty} \frac{1}{ n} \left(K+K^\prime\right) \frac{2 C(\beta)}{c_\beta} = 0 .
	\end{align}
	Inserting the upper bounds \eqref{1st term bound} and \eqref{2nd term bound} into \eqref{1st term}, respectively \eqref{2nd term}, we see that $\theta(\beta)-\theta(\beta+\eps)\leq \delta$ for all $\eps \in \left(0,\eps^\prime\right]$. As the function $\beta \mapsto \theta(\beta)$ is non-increasing, we also directly get that $|\theta(\beta)-\theta(\beta+\eps)| \leq \delta$ for all $\eps \in \left(0,\eps^\prime\right]$. As $\delta > 0$ was arbitrary, this shows continuity from the right of the distance exponent $\dxp(\cdot)$ and thus finishes the proof of Theorem \ref{theo:continuity}.
	
\end{proof}

Let us now go the the proof of Lemma \ref{lem:eps h functions}

\begin{proof}[Proof of Lemma \ref{lem:eps h functions}]
	Let $E=\left\{\{x,y\}\subset V_\mz^{2^n} : x \neq y\right\}$ be the set of edges with both end points in $V_\mz^{2^n}$. Let $\omega \in \{0,1\}^E $ be distributed according to $\p_{\beta < k}^{\beta+\eps \geq k}$ and let $\chi \in \{0,1\}^E$ be a random vector that is independent of $\omega$ and has independent coordinates with distribution defined in \eqref{eq:chi}.
	Then set $\omega^\chi (e) \coloneqq \omega(e) \vee \chi(e) = \max\{ \omega(e) , \chi(e)\}$ for all edges $e \in E$. As discussed earlier, the percolation environment $\omega^\chi$ is distributed like percolation on $V_\mz^{2^n}$ under the measure $\p_{\beta \leq k-1}^{\beta+\eps > k-1}$. Again, we define $G^\prime$ as the graph that results from contracting the blocks of the form $V_u^{2^k}$ to a vertex, called $r(u)$, for all $u\in V_\mz^{2^{n-k}}$. Note that the graph $G^\prime$ only depends on $\omega$, as the open edges in the environment $\chi$ are either within one block $V_v^{2^k}$ or between neighboring blocks $V_v^{2^k}, V_u^{2^k}$ with $\|u-v\|_\infty = 1$. \\
	
	Let $\mathcal{Z}_K$ be the event that every set $Z \in \mathcal{CS}_K\left(G^\prime\right)$ satisfies $ \sum_{r(v) \in Z } \dia \left( V_v^{2^k} ; \omega \right) \leq C^\prime K 2^{k\dxp(\beta)} $ and that every connected set $Z \in \mathcal{CS}_K\left(G^\prime\right)$ contains at least $\frac{K}{9^d 400 \mu_{\beta + 1}}$ separated good vertices with respect to the environment $\omega$.
	We also define $\mathcal{Z}_{\geq K} \coloneqq \bigcap_{t=K}^\infty \mathcal{Z}_{t}$. By \eqref{eq:connectedset goodbound}, \eqref{eq:connectedset diabound} and a union bound over all $t \geq K$ we get that 
	\begin{align}\label{eq:z}
		\p_{\beta < k}^{\beta+\eps \geq k}\left( \mathcal{Z}_{\geq K}^c \right) \leq \sum_{t=K}^\infty
		\p_{\beta < k}^{\beta+\eps \geq k}\left( \mathcal{Z}_{t}^c \right)
		\leq 10 \cdot 2^{-K}
	\end{align}
	for all large enough $K$. 
	Let $\mathcal{D}_K$ be the event
	\begin{align*}
		\mathcal{D}_K & = \left\{  \sum_{r(v) \in Z_\chi} \dia\left(V_v^{2^k}; \omega\right) \leq  r(\eps) 2^{k\dxp(\beta)} K \text{ for all } Z \in \mathcal{CS}_K(G^\prime) \right\} \\
		& \cap \left\{|Z_\chi|  \leq \frac{16 \mu_{\beta+1}}{\log(1/\eps)}K \text{ for all } Z \in \mathcal{CS}_K(G^\prime) \right\},
	\end{align*}
	and let $\mathcal{D}_{\geq K} \coloneqq \bigcap_{t=K}^\infty \mathcal{D}_{t}$. 
	From \eqref{eq:z}, \eqref{set z chi small}, and \eqref{eq:D bound} we get for all $K$ large enough and $\eps > 0$ small enough that
	\begin{align}\label{eq:20}
		& \notag \p_{\beta < k}^{\beta+\eps \geq k} \left(\left(\mathcal{Z}_{\geq K} \cap \mathcal{D}_{\geq K}\right)^c \right) \leq \sum_{t=K}^\infty \left(\p_{\beta < k}^{\beta+\eps \geq k} \left(\mathcal{Z}_t^c \right) + \p_{\beta < k}^{\beta+\eps \geq k} \left(\mathcal{D}_t^c \right)\right) \\
		& \overset{\eqref{set z chi small},\eqref{eq:D bound}}{\leq} \sum_{t=K}^\infty \p_{\beta < k}^{\beta+\eps \geq k} \left(\mathcal{Z}_t^c \right) + 4 \sum_{t=K}^{\infty}  2^{-t} \overset{\eqref{eq:z}}{\leq} 10 \cdot 2^{-K} + 4 \sum_{t=K}^{\infty}  2^{-t} \leq 
		20 \cdot 2^{-K}.
	\end{align}
	Now assume that the event $\mathcal{Z}_{\geq K}$ holds and that $D_{G^\prime}\left(r(\mz), r((2^{n-k}-1)\mo )\right)=K$. Using this, we now construct a path between $\mz$ and $(2^n-1)\mo $. Let $\big(r(\mz)=r(u_0), \ldots, r(u_K) $ $= r((2^{n-k}-1)\mo )\big)$ be a geodesic in $G^\prime$ between $r(\mz)$ and $r((2^{n-k}-1)\mo )$. In case there are multiple geodesics between $r(\mz)$ and $r((2^{n-k}-1)\mo )$, we pick one by some arbitrary deterministic rule. Using this geodesic, we can build a path between $\mz$ and $(2^{n}-1)\mo$ in the environment $\omega$ that only goes from $V_{u_i}^{2^k}$ to $V_{u_{i+1}}^{2^k}$ for $i=0,\ldots,K-1$. Inside each box $V_{u_i}^{2^k}$, we can use the geodesic inside this box to connect the entry and the exit point. This construction gives a path between $\mz$ and $(2^n-1)$ whose length is at most $K+ 
	\sum_{i=0}^{K} \dia \left(V_{u_i}^{2^k};\omega\right)$.
	The path $\left(r(\mz)=r(u_0),\ldots,r(u_K)=r((2^{n-k}-1)\mo )\right)$ is also a connected set in $G^\prime$ containing the origin $r(\mz)$ and of size $K+1$. As we assumed that the event  $\mathcal{Z}_{\geq K}$ holds, we get that
	\begin{equation}\label{eq:distupperbound}
		D_{V_\mz^{2^n}}  \left(\mz,(2^n-1)\mo ; \omega \right) \leq 
		K+ 
		\sum_{i=0}^{K} \dia \left(V_{u_i}^{2^k};\omega\right)
		\leq C^\prime (K+1) 2^{k\dxp(\beta)} + K \leq 2C^\prime K 2^{k\dxp(\beta)} .
	\end{equation}
	On the other hand, let $P=(\mz=x_0,\ldots,x_\ell=(2^n-1)\mo)$ be a geodesic from $\mz$ to $(2^n-1)\mo$ in the environment $\omega^\chi$. Define the set 
	\begin{equation*}
		\hat{P} = \left\{r(w) \in G^\prime: x_i \in V_w^{2^k} \text{ for some } i\in \{0,\ldots,\ell\}\right\},
	\end{equation*}
	which can be thought of the projection of the path $P$ onto the graph $G^\prime$. Further, note that $\hat{P}$ is a connected set in $G^\prime$.
	As the graph distance between $r(\mz)$ and $r((2^{n-k}-1)\mo)$ is $K$ in graph $G^\prime$ $($i.e., $D_{G^\prime}\left(r(\mz), r((2^{n-k}-1) \mo)\right)=K)$, the projection $\hat{P}$ needs to have a size of at least $K$. In the following, we will show that the length of the path $P$ is at least of the same order as $\left|\hat{P}\right|2^{k\dxp(\beta)}$. \\
	
	\noindent
	Define the sets $\hat{P}^\cN$ and $\hat{P}^{\cN}_\chi$ by
	\begin{align*}
		&\hat{P}^\cN  \coloneqq \bigcup_{r(v) \in \hat{P}} \cN(r(v)) = \left\{r(u) \in G^\prime : \cN(r(u)) \cap \hat{P} \neq \emptyset\right\}, \text{ and }\\
		& \hat{P}^{\cN}_\chi \coloneqq \left\{r(v) \in \hat{P}^\cN : \exists a \in V_v^{2^k}, b \in V_\mz^{2^n} \text{ with } \chi(\{a,b\})=1\right\} .
	\end{align*}
	The definition of the set $\hat{P}^\cN$ directly implies that
	\begin{align*}
		K \leq |\hat{P}| \leq |\hat{P}^\cN| \leq 3^d |\hat{P}|.
	\end{align*}
	As we assumed that the event $\mathcal{D}_{\geq K}$ holds, we directly get that
	\begin{align}\label{eq:hat P N chi size}
		|\hat{P}^\cN_\chi| \leq \frac{ 16 \mu_{\beta + 1} }{\log(1/\eps)} |\hat{P}^\cN| \leq \frac{3^d 16 \mu_{\beta + 1} }{\log(1/\eps)} |\hat{P}|
	\end{align}
	
	The set $\hat{P}$ has size at least $K$ and - as we are assuming that the event  $\mathcal{Z}_{\geq K}$ holds - contains at least $\frac{\left|\hat{P}\right|}{9^d400\mu_{\beta + 1}}$ separated good vertices with respect to the environment $\omega$. Say that $Z \subset \hat{P}$ is a subset of separated good vertices (for the environment $\omega$) with $|Z| \geq \frac{\left|\hat{P}\right|}{9^d400\mu_{\beta + 1}}$. Then we also get the set $Z \setminus \hat{P}^{\cN}_\chi$ satisfies
	\begin{align*}
		\left|Z \setminus \hat{P}^{\cN}_\chi\right| \geq \left|Z \right| - \left| \hat{P}^{\cN}_\chi\right|
		\overset{\eqref{eq:hat P N chi size}}{\geq } \frac{\left|\hat{P}\right|}{9^d400\mu_{\beta + 1}} - \frac{3^d 16 \mu_{\beta + 1} }{\log(1/\eps)} |\hat{P}|
		\geq
		\frac{\left|\hat{P}\right|}{9^d500\mu_{\beta + 1}},
	\end{align*}
	where the last inequality holds for $\eps$ small enough.
	Now consider the situation where the path $P$ crosses a block $V_w^{2^k}$ with $r(w) \in Z \setminus \hat{P}^{\cN}_\chi$, in the sense that it starts somewhere outside of $\bigcup_{r(u) \in G^\prime : \|u-w\|_\infty \leq 1}  V_u^{2^k}$, then is inside the set $V_w^{2^k}$, and then leaves the set $\bigcup_{r(u) \in G^\prime : \|u-w\|_\infty \leq 1}  V_u^{2^k}$ again. As the set $\bigcup_{r(u) \in G^\prime : \|u-w\|_\infty \leq 1}  V_u^{2^k}$ is not adjacent to any edges induced by $\chi$, this part of the path can only use the edges induced by $\omega$. In this case, the path $P$ already needs to make at least $\delta 2^{k\dxp(\beta)}$ many steps inside the set $\bigcup_{r(u) \in G^\prime : \| u - w \|_\infty \leq 1} V_u^{2^k}$, as the block $V_w^{2^k}$ was assumed to be good and the path $P$ was assumed to be self-avoiding. The path $P$ crosses at least $\frac{|\hat{P}|}{9^d500\mu_{\beta + 1}}-2$ separated good boxes, where the subtraction of two is necessary because the path $P$ might touch good boxes at the beginning/end without crossing them. The sets $\bigcup_{r(u) \in G^\prime : \| u - w \|_\infty \leq 1} V_u^{2^k}$ are not directly connected for different vertices $r(w) \in Z \setminus \hat{P}^{\cN}_\chi$, as we assumed that $Z$ is a set of separated good vertices. So in particular, for each box of the form $V_w^{2^k}$ with $r(w) \in Z \setminus \hat{P}^{\cN}_\chi$, the path $P$ needs to have at least $\delta 2^{k\theta(\beta)}$ many steps inside this box. This already implies that
	\begin{align}\label{eq:distlowerbound}
		D_{V_\mz^{2^n}}\left(\mz, (2^n-1)\mo ; \omega^\chi \right) = \ell = \text{length}(P) \geq \left(\frac{|\hat{P}|}{9^d 500\mu_{\beta + 1}}-2\right) \delta 2^{k\dxp(\beta)} \geq c_1 2^{k \dxp(\beta)} .
	\end{align}
	for $c_1 = \frac{|\hat{P}|}{9^d 600\mu_{\beta + 1}}$ and $k$ large enough.
	This shows a lower bound on the length of paths in the environment $\omega^\chi$.
	Combining the inequalities \eqref{eq:distupperbound} and \eqref{eq:distlowerbound} we get that for $K=D_{G^\prime} \left( r(\mz), r((2^{n-k}-1)\mo ) \right)$
	\begin{align*}
		2C^\prime K 2^{k\dxp(\beta)} \geq D_{V_\mz^{2^n}} \left(\mz,(2^n-1)\mo ; \omega \right) \geq D_{V_\mz^{2^n}} \left(\mz,(2^n-1)\mo ; \omega^\chi \right)  \geq  c_1 \left|\hat{P}\right| 2^{k \dxp(\beta)}
	\end{align*}
	and thus
	\begin{equation}\label{eq:wandering}
		\left| \hat{P} \right| \leq 
		\frac{2 C^\prime K}{c_1} 
		= \frac{2 C^\prime}{c_1} D_{G^\prime}\left(r(\mz) , r((2^{n-k}-1)\mo )\right) \eqqcolon C_w D_{G^\prime}\left(r(\mz), r((2^{n-k}-1)\mo ) \right) \text .
	\end{equation}
	So the shortest path $P$ between $\mz$ and $(2^n-1)\mo $ in the environment $\omega^\chi = \omega\vee \chi$ does not touch more than $C_w D_{G^\prime}\left(r(\mz),r((2^{n-k}-1)\mo )\right)$ blocks of the form $V_w^{2^k}$. This is a useful observation, as the path also needs to touch at least $D_{G^\prime}\left(r(\mz),r((2^{n-k}-1)\mo )\right)$ many blocks of the form $V_w^{2^k}$.
	Using this, we can bound the difference 
	\begin{equation*}
		\Delta_\eps D_{V_\mz^{2^n}} \coloneqq 
		D_{V_\mz^{2^n}} \left(\mz,(2^n-1)\mo ;\omega\right) - 
		D_{V_\mz^{2^n}} \left(\mz,(2^n-1)\mo ;\omega^\chi\right)
	\end{equation*}
	Let $P = \left(x_0,\ldots,x_\ell\right)$ be a geodesic between $x_0=\mz$ and $x_\ell=(2^n-1)\mo $ in the environment $\omega^\chi$. Define the set $\hat{P} = \left\{r(w) \in G^\prime: x_i \in V_w^{2^k} \text{ for some } i \in \{0,\ldots,\ell\}\right\}$. It is proven in Claim \ref{claim:geodesic comparsions} below that the distance between $\mz$ and $(2^n-1)\mo$ in the environment $\omega$ is at most $\ell+ 2 \sum_{r(v) \in \hat{P}_\chi} \dia \left(V_v^{2^k};\omega\right)$.
	This already implies
	\begin{align}\label{eq:boundongoodevents}
		&\notag \Delta_\eps D_{V_\mz^{2^n}}  = D_{V_\mz^{2^n}} \left(\mz,(2^n - 1)\mo  ; \omega\right) - \ell 
		\leq 2 \sum_{r(v) \in \hat{P}_\chi} \dia\left(V_v^{2^k};\omega\right) \\
		& \leq 2 \left|\hat{P}_\chi\right| r(\eps) \overset{\eqref{eq:wandering}}{\leq} 2 C_w r(\eps) D_{G^\prime} \left(r(\mz),r((2^{n-k}-1)\mo )\right) 2^{k\dxp(\beta)}
	\end{align}
	for small enough $\eps$ and when $ \mathcal{D}_{\geq K} \cap \mathcal{Z}_{\geq K}$ and $D_{G^\prime}\left(r(\mz),r((2^{n-k}-1)\mo )\right) = K$ hold for large enough $K$. So this gives us a bound $\Delta_\eps D_{V_\mz^{2^n}}$ that goes to 0, as $\eps \searrow 0$. This bound only holds on the previously mentioned event, but we can also choose $K$, depending on $n-k$, in such a way such that the probability of this event goes to $1$ as $n-k \to \infty$. We have that
	\begin{align}
		&\notag \E_{\beta < k}^{\beta+\eps \geq k} \left[D_{V_\mz^{2^n}}\left(\mz,(2^n-1)\mo \right)\right]-\E_{\beta < k-1}^{\beta+\eps \geq k-1} \left[D_{V_\mz^{2^n}}\left(\mz,(2^n-1)\mo \right)\right] 
		= \E_{\beta < k}^{\beta+\eps \geq k} \left[ \Delta_\eps D_{V_\mz^{2^n}} \right]\\
		& \label{eq:zeile1} = \E_{\beta < k}^{\beta+\eps \geq k} \Big[ \Delta_\eps D_{V_\mz^{2^n}}	\cdot \mathbbm{1}_{\left\{ \mathcal{D}_{\geq n-k} \cap \mathcal{Z}_{\geq n-k} \right\}} \mathbbm{1}_{\left\{D_{G^\prime}\left( r(\mz), r((2^{n-k} - 1)\mo ) \right) \geq n-k\right\}} \Big] \\
		& \ \ \label{eq:zeile2} + \E_{\beta < k}^{\beta+\eps \geq k} \Big[ \Delta_\eps D_{V_\mz^{2^n}}	\cdot
		\mathbbm{1}_{\left\{ \mathcal{D}_{\geq n-k} \cap \mathcal{Z}_{\geq n-k} \right\}} \mathbbm{1}_{\left\{D_{G^\prime}\left( r(\mz), r((2^{n-k} - 1)\mo )  \right) < n-k\right\}} \Big]\\
		& \label{eq:zeile3} \ \ +  \E_{\beta < k}^{\beta+\eps \geq k} \Big[ \Delta_\eps D_{V_\mz^{2^n}} 	\mathbbm{1}_{\left\{ \left(\mathcal{D}_{\geq n-k} \cap \mathcal{Z}_{\geq n-k}\right)^c \right\}}  \Big] .
	\end{align}
	We upper bound the expression in \eqref{eq:zeile1} by
	\begin{align*}
		& \E_{\beta < k}^{\beta+\eps \geq k} \Big[ \Delta_\eps D_{V_\mz^{2^n}}	\cdot \mathbbm{1}_{\left\{ \mathcal{D}_{\geq n-k} \cap \mathcal{Z}_{\geq n-k} \right\}} \mathbbm{1}_{\left\{D_{G^\prime}\left( r(\mz), r((2^{n-k} - 1)\mo ) \right) \geq n-k\right\}} \Big] \\
		& \overset{\eqref{eq:boundongoodevents}}{\leq} \E_{\beta < k}^{\beta+\eps \geq k} \left[ D_{G^\prime}\left( r(\mz), r((2^{n-k} - 1)\mo ) \right) 2 C_w r(\eps) 2^{k \dxp(\beta)}  \right]
		\\
		&
		=
		2 C_w r(\eps) 2^{k \dxp(\beta)}
		\E_{\beta + \eps} \left[ D_{V_{\mz}^{2^{n-k}}}\left( \mz, (2^{n-k} - 1)\mo  \right)   \right]
		\leq
		2 C_w r(\eps) 2^{k\dxp(\beta)} \Lambda(2^{n-k},\beta+\eps) 
	\end{align*}
	for small enough $\eps > 0$ and large enough $k, n-k$. 
	
	Using that $\Delta_\eps D_{V_\mz^{2^n}} \leq D_{V_\mz^{2^n}}\left(\mz,(2^n-1)\mo;\omega\right)$, the term \eqref{eq:zeile3} can be upper bounded with the Cauchy-Schwarz inequality by
	\begin{align*}
		& \E_{\beta < k}^{\beta+\eps \geq k} \Big[ \Delta_\eps D_{V_\mz^{2^n}} 	\mathbbm{1}_{\left\{ \left(\mathcal{D}_{\geq n-k} \cap \mathcal{Z}_{\geq n-k}\right)^c \right\}}  \Big] \\
		& 
		\leq \sqrt{\E_{\beta < k}^{\beta+\eps \geq k} \left[D_{V_\mz^{2^n}}(\mz,(2^n-1)\mo )^2\right]} \sqrt{\E_{\beta < k}^{\beta+\eps \geq k} \left[ \mathbbm{1}^2_{ \left\{ \left(\mathcal{D}_{\geq n-k} \cap \mathcal{Z}_{\geq n-k} \right)^c \right\} } \right]}\\
		& \overset{\eqref{eq:20}}{\leq} \sqrt{C_\beta}  \Lambda(2^k,\beta) \Lambda(2^{n-k},\beta+\eps) \sqrt{20} \cdot 2^{-\frac{n-k}{2}}
		\eqqcolon  \Lambda(2^k,\beta) \Lambda(2^{n-k},\beta+\eps) h_1(n-k) ,
	\end{align*}
	where $C_\beta$ is the constant of Lemma \ref{lem:uniform2ndmomentbound:renormalized}. Further, \eqref{eq:zeile2} can be upper bounded with the same method  by
	\begin{align*}
		&\E_{\beta < k}^{\beta+\eps \geq k} \Big[ \Delta_\eps D_{V_\mz^{2^n}}	\cdot
		\mathbbm{1}_{\left\{ \mathcal{D}_{\geq n-k} \cap \mathcal{Z}_{\geq n-k} \right\}} \mathbbm{1}_{\left\{D_{G^\prime}\left( r(\mz), r((2^{n-k} - 1)\mo )  \right) < n-k\right\}} \Big] \\
		& 
		\leq \sqrt{\E_{\beta < k}^{\beta+\eps \geq k} \left[D_{V_\mz^{2^n}}(\mz,(2^n-1)\mo )^2\right]} \sqrt{\E_{\beta < k}^{\beta+\eps \geq k} \left[ \mathbbm{1}^2_{\left\{ D_{G^\prime}\left( r(\mz), r((2^{n-k} - 1)\mo )  \right) < n-k \right\}} \right]}\\
		& \leq \sqrt{C_\beta}  \Lambda(2^k,\beta) \Lambda(2^{n-k},\beta+\eps) \left(\p_{\beta + 1}\left( D_{V_\mz^{2^{n-k}}}   (\mz,(2^{n-k}-1)\mo ) < n-k \right)^{1/2}  \right)\\
		&
		\eqqcolon  \Lambda(2^k,\beta) \Lambda(2^{n-k},\beta+\eps) h_2(n-k) .
	\end{align*}
	Note that $h_2(n-k) \to 0$ as $n-k \to \infty$, as we expect that $D_{V_\mz^{2^{n-k}}}   (\mz,(2^{n-k}-1)\mo )$ is of order $2^{(n-k)\theta(\beta+1)} \gg n-k$ under the measure $\p_{\beta+1}$, see \eqref{eq:baeumler1}. This allows us to bound all the expressions occurring in lines \eqref{eq:zeile1}, \eqref{eq:zeile2} and \eqref{eq:zeile3}.
	Using that $2^{k\theta(\beta)} \leq \Lambda \left(2^k,\beta\right)$ by \eqref{eq:subsuperimpli}, we see that
	\begin{align*}
		& \E_{\beta < k}^{\beta+\eps \geq k} \left[D_{V_\mz^{2^n}}\left(\mz,(2^n-1)\mo \right)\right]-\E_{\beta < k-1}^{\beta+\eps \geq k-1} \left[D_{V_\mz^{2^n}}\left(\mz,(2^n-1)\mo \right)\right] \\
		&
		\leq 2 C_w r(\eps) 2^{k\dxp(\beta)} \Lambda(2^{n-k},\beta+\eps) + \Lambda(2^k,\beta) \Lambda(2^{n-k},\beta+\eps) \left(h_1(n-k)  + h_2(n-k) \right)\\
		&
		\leq \big(2 C_w r(\eps)  + h_1(n-k)  + h_2(n-k) \big) \Lambda(2^k,\beta) \Lambda(2^{n-k},\beta+\eps)
		\\
		&
		\leq
		C \big( r(\eps)  + h(n-k)  \big) \Lambda(2^k,\beta) \Lambda(2^{n-k},\beta+\eps)
	\end{align*}
	with $C=2C_w \vee 1$ and $h(n-k)=h_1(n-k)+h_2(n-k)$.
\end{proof}

\begin{claim}\label{claim:geodesic comparsions}
	Let $P = \left(x_0,\ldots,x_\ell\right)$ be a path between $x_0=\mz$ and $x_\ell=(2^n-1)\mo $ in the environment $\omega^\chi$. Define the set 
	\begin{equation*}
		\hat{P} = \left\{r(w) \in G^\prime: x_i \in V_w^{2^k} \text{ for some } i \in \{0,\ldots,\ell\}\right\}
	\end{equation*}
	Then
	\begin{equation*}
		D_{V_\mz^{2^n}}\left(\mz, (2^n-1)\mo ; \omega \right) \leq \ell + 2 \sum_{r(v) \in \hat{P}_\chi} \dia \left(V_v^{2^k};\omega\right).
	\end{equation*}
\end{claim}

\begin{proof}
	We build a path $\left(y_0, \ldots, y_{\tilde{\ell}}\right)$ between $\mathbf{0}$ and $\left(2^n-1\right) \mathbf{1}$ in the environment $\omega$ as follows:
	
	As long as $\omega\left(\left\{x_i, x_{i+1}\right\}\right)=1$, we follow the path $P$. If $\omega\left(\left\{x_i, x_{i+1}\right\}\right)=0$, say with $x_i \in V_u^{2^k}$ and $x_{i+1} \in V_w^{2^k}$, we take the shortest path from $x_i$ to $x_{i^{\prime}}$ where $i^{\prime}=\max \left\{\ell^{\prime}: x_{\ell^{\prime}} \in V_w^{2^k} \cup V_u^{2^k}\right\}$. That means, we go to the point $x_{i^{\prime}}$ where the path $\left(x_0, \ldots, x_\ell\right)$ leaves the boxes $V_w^{2^k} \cup V_u^{2^k}$ for the last time. As $\omega$ and $\omega^\chi$ can only differ at edges with length in $\left[2^{k-1}, 2^k-1\right]$, we already have $\|u-w\|_{\infty} \in\{0,1\}$, and thus the boxes $V_u^{2^k}$ and $V_w^{2^k}$ are either identical or connected by a nearest-neighbor-edge. So the length of the path segment between $x_i$ and $x_{i^{\prime}}$ constructed by this procedure is bounded by 
	\begin{align*}
		D_{V_\mz^{2^n}} (x_i, x_{i^\prime}) \leq \begin{cases}
			\dia\left(V_u^{2^k} ; \omega\right)+\dia\left(V_w^{2^k} ; \omega\right) + 1 & \text{ if } u \neq w \\
			\dia\left(V_u^{2^k} ; \omega\right) & \text{ if } u = w 
		\end{cases} .
	\end{align*}
	Thus, the length of path that we construct increases by at most $\dia\left(V_u^{2^k} ; \omega\right)+\dia\left(V_w^{2^k} ; \omega\right)$ if $u\neq w$, respectively by $\dia\left(V_u^{2^k} ; \omega\right)$ if $u=w$.
	When at $x_{i^{\prime}}$, we follow the path $P$ again until there appears again an edge $e=\left\{x_j, x_{j+1}\right\}$ with $\omega^\chi(e)=1=1-\omega(e)$ and do the same procedure as before. 
	
	Say that $x_i \in V_u^{2^k}, x_{i+1} \in V_w^{2^k}$, $\chi(\{x_i, x_{i+1}\}) = 1 = 1- \omega(\{x_i, x_{i+1}\})$, and we insert a path between $x_i$ and $x_{i^\prime}$ as described above.
	After leaving $x_{i^\prime}$, the path $\left(x_{i^\prime+1}, \ldots, x_\ell \right)$ can never return to $V_{u}^{2^k}$ by the definition of $i^\prime = \max \left\{\ell^{\prime}: x_{\ell^{\prime}} \in V_w^{2^k} \cup V_u^{2^k}\right\}$. So it can happen at most twice that the insertion of a path between pairs of vertices with one endpoint in $V_u^{2^k}$ is necessary, namely once in order to get to $x_{i^{\prime}}$ and, if $x_{i^\prime} \in V_{u}^{2^k}$ and $\chi\left(\{x_{i^\prime},x_{i^\prime + 1}\}\right)=1= 1- \omega\left(\{x_{i^\prime},x_{i^\prime + 1}\}\right)$, also once in order to leave $x_{i^{\prime}}$. The same also holds for $r(w)$ instead of $r(u)$. It can happen at most twice that the insertion of a path between pairs of vertices with one endpoint in $V_w^{2^k}$ is necessary, namely once in order to get to $x_{i^{\prime}}$ and, if $x_{i^\prime} \in V_{w}^{2^k}$  and $\chi\left(\{x_{i^\prime},x_{i^\prime + 1}\}\right)=1= 1- \omega\left(\{x_{i^\prime},x_{i^\prime + 1}\}\right)$, also once in order to leave $x_{i^{\prime}}$. Thus, for each $r(v) \in \hat{P}_\chi$ it can happen at most twice that the insertion of a path between a pair of vertices with at least one of them in $V_v^{2^k}$ is necessary.
	
	So the construction above gives a path from $\mathbf{0}$ to $\left(2^n-1\right) \mathbf{1}$ in the environment $\omega$, and this path has length at most $\ell+2 \sum_{r(v) \in \hat{P}_\chi} \operatorname{Diam}\left(V_v^{2^k} ; \omega\right)$.
\end{proof}

\subsection{Proof of the auxiliary lemmas}\label{sec:auxil}

\noindent
In this section, we state and prove several auxiliary lemmas.

\begin{lemma}\label{lem:auxil}
	For each $n\in \N, d\in \N$ and $\beta \geq 0$, the function $\Lambda(n,\beta)$ defined in Lemma \ref{lem:submultiplicativity} satisfies
	\begin{equation*}
		\Lambda(n,\beta) \geq n^{\dxp(d,\beta)}
	\end{equation*}
	and
	\begin{equation}\label{2-comparison}
		\Lambda(n,\beta) \leq \Lambda (2n,\beta) \leq 2 \Lambda(n,\beta).
	\end{equation}
\end{lemma}

\begin{proof}
	The fact that $\Lambda(n,\beta) \geq n^{\dxp(d,\beta)}$ directly follows from the submultiplicativity of $\Lambda$, since
	\begin{equation*}
		\theta(d,\beta) = \inf_{n\in \N} \frac{\log\left(\Lambda(n,\beta)\right)}{\log(n)}.
	\end{equation*}
	The first inequality $\Lambda (2n,\beta) \leq 2 \Lambda(n,\beta)$ also follows directly from the submultiplicativity, since
	\begin{equation*}
		\Lambda (2n,\beta) \leq \Lambda(2,\beta) \Lambda(n,\beta) = 2 \Lambda(n,\beta).
	\end{equation*}
	For the second inequality $\Lambda(n,\beta) \leq \Lambda (2n,\beta)$, it suffices to show that for all $x,y \in V_\mz^n$ one has
	\begin{equation*}
		\E_\beta \left[D_{V_\mz^n} (x,y)\right] \leq \E_\beta \left[D_{V_\mz^{2n}} (2x,2y)\right] .
	\end{equation*}
	This follows directly from the self-similarity of the model. We define the graph $G^\prime=(V^\prime, E^\prime)$ by starting with the random graph with vertex set $V_\mz^{2n}$ and contracting the boxes $\left(V_u^2\right)_{u\in V_\mz^n}$. We write $r(u)$ for the vertex in $V^\prime$ that corresponds to the set $V_u^2$. This construction directly implies that $D_{V_\mz^{2n}} (2x,2y) \geq D_{G^\prime} (r(x),r(y))$. By the self-similarity of the model, the graph $G^\prime$ has the exact same distribution as long-range percolation with measure $\p_\beta$ on $V_\mz^n$ and thus
	\begin{equation*}
		\E_\beta \left[D_{V_\mz^{2n}} (2x,2y)\right] \geq 
		\E_\beta \left[D_{G^\prime} (r(x),r(y))\right]
		=
		\E_\beta \left[D_{V_\mz^{n}} (x,y)\right].
	\end{equation*}
\end{proof}

The next lemma is an analogous version of Lemma \ref{lem:linearspacingcompanion} for the mixed measure $\p_{\beta < k}^{\beta+\eps \geq k}$.

\begin{lemma}\label{lem:linearspacingmixed}
	For all $k\in \N, \delta, \eps \in \left[0,1\right]$ with  $\frac{1}{2^k} < \delta \leq \frac{1}{4}$ and $u,w \in \Z^d \setminus \{\mz\}$ with $\|u\|_\infty \geq 2$ and $u\neq w$ one has
	\begin{align*}
		\p_{\beta < k}^{\beta+\eps \geq k} \big( \exists x,y \in V_{\mz}^{2^k} : \|x-y\|_\infty \leq \delta 2^k, x \sim V_u^{2^k}, y \sim V_w^{2^k} \ \big|  \ V_\mz^{2^k} \sim V_u^{2^k}, V_\mz^{2^k} \sim V_w^{2^k} \big) \leq C^\prime \delta^{1/2}
	\end{align*}
	for some $C^\prime = C^\prime(d,\beta)<\infty$.
\end{lemma}

\begin{proof}
	Using the result of Lemma \ref{lem:linearspacingcompanion}, we get that
	\begin{align}
		& \label{eq:nominator}
		\p_{\beta < k}^{\beta+\eps \geq k} \big( \exists x,y \in V_{\mz}^{2^k} : \|x-y\|_\infty \leq \delta 2^k, x \sim V_u^{2^k}, y \sim V_w^{2^k} ,  V_\mz^{2^k} \sim V_u^{2^k}, V_\mz^{2^k} \sim V_w^{2^k} \big)\\
		& \notag
		=
		\p_{\beta < k}^{\beta+\eps \geq k} \big( \exists x,y \in V_{\mz}^{2^k} : \|x-y\|_\infty \leq \delta 2^k, x \sim V_u^{2^k}, y \sim V_w^{2^k} \big)\\
		& \notag
		\leq
		\p_{\beta +\eps} \big( \exists x,y \in V_{\mz}^{2^k} : \|x-y\|_\infty \leq \delta 2^k, x \sim V_u^{2^k}, y \sim V_w^{2^k} \big) \\
		&
		\notag
		= \p_{\beta +\eps} \big( \exists x,y \in V_{\mz}^{2^k} : \|x-y\|_\infty \leq \delta 2^k, x \sim V_u^{2^k}, y \sim V_w^{2^k} \ | \  V_u^{2^k} \sim V_\mz^{2^k} \sim V_w^{2^k} \big) \p_{\beta + \eps} \left(V_u^{2^k} \sim V_\mz^{2^k} \sim V_w^{2^k}\right)
		\\
		& \label{eq:nominator2}
		\overset{\eqref{eq:graphspacingcompanion}}{\leq}
		C_d^\prime \delta^{1/2} \lceil\beta + \eps \rceil^2  \p_{\beta + \eps} \left(V_u^{2^k} \sim V_\mz^{2^k} \sim V_w^{2^k}\right)
		\leq
		C_d^\prime \delta^{1/2} \lceil\beta + 1 \rceil^2  \p_{\beta + \eps} \left(V_u^{2^k} \sim V_\mz^{2^k} \sim V_w^{2^k}\right) .
	\end{align}
	Now note that for all $u,w \in \Z^d \setminus \{\mz\}$ with $u \neq w$ one has
	\begin{align}\label{fraction}
		\frac{\p_{\beta + \eps} \left(V_u^{2^k} \sim V_\mz^{2^k} \sim V_w^{2^k}\right)}{\p_{\beta < k}^{\beta+\eps>k} \left(V_u^{2^k} \sim V_\mz^{2^k} \sim V_w^{2^k}\right)}
		\leq
		\frac{\p_{\beta + 1} \left(V_u^{2^k} \sim V_\mz^{2^k} \sim V_w^{2^k}\right)}{\p_{\beta} \left(V_u^{2^k} \sim V_\mz^{2^k} \sim V_w^{2^k}\right)}
		=
		\frac{\p_{\beta + 1} \left(\mz \sim u\right) \p_{\beta + 1} \left(\mz \sim w\right)}{\p_{\beta} \left(\mz \sim u\right) \p_{\beta} \left(\mz \sim w\right)}.
	\end{align}
	The last fraction is uniformly bounded over all choices of $u,w \in \Z^d \setminus \{\mz\}$, since the probability of a connection under the measure $\p_{\tilde{\beta}}$ is bounded from above and below by
	\begin{equation}\label{upperlowerbeta}
		\frac{(4d)^{-2d} \tilde{\beta}}{\|v\|_\infty^{2d}} \wedge \frac{1}{2} \leq \p_{\tilde{\beta}} \left(\mz \sim v\right) \leq \frac{2^{2d} \tilde{\beta}}{\|v\|_\infty^{2d}}
	\end{equation}
	for all $\tilde{\beta} \geq 0$, $v\in \Z^d$ with $\|v\|_\infty \geq 2$; This was proven in \cite[(3) and (4)]{baeumler2022distances}. So in particular both the numerator and the denominator in the last term of \eqref{fraction} are always of the same order. Dividing by $\p_{\beta < k}^{\beta+\eps>k} \left(V_u^{2^k} \sim V_\mz^{2^k} \sim V_w^{2^k}\right)$ in both \eqref{eq:nominator} and \eqref{eq:nominator2} shows that
	\begin{align*}
		& \p_{\beta < k}^{\beta+\eps \geq k} \big( \exists x,y \in V_{\mz}^{2^k} : \|x-y\|_\infty \leq \delta 2^k, x \sim V_u^{2^k}, y \sim V_w^{2^k} \ | \  V_\mz^{2^k} \sim V_u^{2^k}, V_\mz^{2^k} \sim V_w^{2^k} \big) \\
		&
		\leq
		C_d^\prime \delta^{1/2} \lceil\beta + 1 \rceil^2  \frac{\p_{\beta + \eps} \left(V_u^{2^k} \sim V_\mz^{2^k} \sim V_w^{2^k}\right)}{\p_{\beta < k}^{\beta+\eps>k} \left(V_u^{2^k} \sim V_\mz^{2^k} \sim V_w^{2^k}\right)}
		\overset{\eqref{fraction}}{\leq}
		C_d^\prime \delta^{1/2} \lceil\beta + 1 \rceil^2  \frac{\p_{\beta + 1} \left(\mz \sim w\right) \p_{\beta + 1} \left(\mz \sim u\right)}{\p_{\beta} \left(\mz \sim w\right) \p_{\beta} \left(\mz \sim u\right)}\\
		&
		\leq
		C^\prime \delta^{1/2}
	\end{align*}
	for some constant $C^\prime<\infty$ depending only on the dimension $d$ and $\beta$.
\end{proof}

We also introduce the following elementary result.

\begin{lemma}\label{lem:number of edges conditioned}
	If $X$ is the sum of finitely many independent Bernoulli random variables, then for all $K\in \N_{>0}$
	\begin{align*}
		\p\left(X \geq K | X \geq 1\right) \leq \frac{2(\E\left[X\right]\vee 1)}{K} .
	\end{align*}
\end{lemma}

\begin{proof}
	If $X=\sum_{i=1}^{n} X_i$ is the sum of independent Bernoulli random variables, then
	\begin{align}\label{eq:pbound}
		\notag \p(X\geq 1) & = 1- \p(X=0) =  1- \p(X_i = 0 \text{ for } i=1,\ldots,n) = 1 - \prod_{i=1}^{n} \p(X_i=0) \\
		& = 1-\prod_{i=1}^{n} (1-\E\left[X_i\right]) 
		\geq
		1-\prod_{i=1}^{n} e^{-\E\left[X_i\right]} = 1-e^{-\E\left[X \right]} \geq \frac{\E\left[X\right]\wedge 1}{2} ,
	\end{align}
	where we used the elementary inequality $1-e^{-s} \geq \frac{s\wedge 1}{2}$ in the last step.
	For $\E\left[X\right] > 1$, the above inequality says that $\p(X\geq 1) \geq 0.5$. Markov's inequality directly shows that $\p(X\geq K) \leq \frac{1}{K} \E\left[X\right]$. Combining these two inequalities we get that
	\begin{equation}\label{eq:groesser1}
		\p\left(X \geq K | X \geq 1\right) =
		\frac{\p\left(X \geq K\right)}{\p\left( X \geq 1\right)}
		\leq
		\frac{\tfrac{1}{K} \E \left[X\right]}{0.5} = \frac{2}{K} \E \left[X\right] 
	\end{equation}
	if $\E\left[X\right]>1$.
	Contrary to that, for $\E\left[X\right] \leq 1$, inequality \eqref{eq:pbound} says that $\p(X\geq 1) \geq 0.5 \E\left[X\right]$. Combining this again with the inequality $\p(X\geq K) \leq \frac{1}{K} \E\left[X\right]$, we get that
	\begin{equation}\label{eq:kleiner1}
		\p\left(X \geq K | X \geq 1\right)
		=
		\frac{\p\left(X \geq K\right)}{\p\left( X \geq 1\right)}
		\leq
		\frac{\tfrac{1}{K} \E \left[X\right]}{0.5 \E\left[X\right]} = \frac{2}{K}.
	\end{equation}
	The statement of the lemma directly follows by combining inequalities \eqref{eq:groesser1} and \eqref{eq:kleiner1}.
\end{proof}

With this, we are ready to go to the proofs of Lemmas \ref{lem:goodeventB} and \ref{lem:goodeventA}. We only do the proof of Lemma \ref{lem:goodeventA}, the proof of Lemma \ref{lem:goodeventB} works similar.

\begin{proof}[Proof of Lemma \ref{lem:goodeventA}]
	We write $\p^{u,v}\left(\cdot\right)$ for the conditioned probability measure $\p_{\beta<k}^{\beta+\eps>k} \big( \cdot \ \big| \ V_u^{2^k} \sim V_\mz^{2^k} \sim V_v^{2^k} \big)$. We define the event 
	\begin{align*}
		& \mathcal{A}(K,\delta_1,\delta)  = 
		\left\{\|x-y\|_\infty > \delta_1 2^k \text{ for all } x,y \in V_{\mz}^{2^k} \text{ with } x\sim V_u^{2^k}, y\sim V_v^{2^k} \right\}\\
		&
		\cap
		\left\{D_{V_{\mz}^{2^k}}\left(x,B_{\delta_1 2^k}(x)^c\right) > \delta 2^{k\dxp(\beta)} \text{ for all } x \in V_{\mz}^{2^k} \text{ with } x\sim V_u^{2^k} \right\}\\
		&
		\cap
		\left\{ \left| \left\{x\in V_{\mz}^{2^k} : x\sim V_u^{2^k} \right\} \right| \leq K \right\}
	\end{align*}
	and observe that 
	\begin{equation}\label{supset}
		\mathcal{A}_{u,v} (\delta) 
		= 
		\bigcap_{ \substack{x \in V_\mz^{2^k}: \\ x \sim V_u^{2^k}} } \ \bigcap_{ \substack{ y \in V_\mz^{2^k}: \\ y \sim V_v^{2^k}} } \left\{D_{V_\mz^{2^k}}\left(x,y\right) \geq \delta 2^{k\dxp(\beta)} \right\} \supset  \mathcal{A}(K,\delta_1,\delta) \text .
	\end{equation}
	Indeed, if the event $\mathcal{A}_{u,v} (\delta) $ does not hold, then there exists $x,y\in V_\mz^{2^k}$ with $x\sim V_u^{2^k}, y \sim V_v^{2^k}$, and $D_{V_\mz^{2^k}}\left(x,y\right) \leq \delta 2^{k\dxp(\beta)}$. For $x,y$ one either has $\|x-y\|_\infty \leq \delta_1 2^k$ (which says that the first event in the definition of $\mathcal{A}(K,\delta_1,\delta)$ does not hold) or $\|x-y\|_\infty > \delta_1 2^k$ and thus $D_{V_\mz^{2^k}}\left(x,B_{\delta_1 2^k}(x)^c\right) \leq D_{V_\mz^{2^k}}\left(x,y\right) \leq \delta 2^{k\dxp(\beta)}$ (which says that the second event in the definition of $\mathcal{A}(K,\delta_1,\delta)$ does not hold). In total, we see that if the event $\mathcal{A}_{u,v} (\delta) $ does not hold, then also the event $\mathcal{A}(K,\delta_1,\delta)$ does not hold, which proves \eqref{supset}.\\
	Thus it suffices to show that for $\delta>0$ small enough the probability $\p^{u,v}\left( \mathcal{A}(K,\delta_1,\delta) \right)$ can be arbitrarily close to 1, for an appropriate choice of $K = K(\delta), \delta_1 = \delta_1(\delta)$ and $k\in \N$ large enough. Respectively, that $\p^{u,v}\left( \mathcal{A}(K,\delta_1,\delta)^c \right)$ converges to $0$ for $\delta \to 0$. We have that
	\begin{align}\label{union}
		& \mathcal{A}(K,\delta_1,\delta)^c  = \left\{ \left| \left\{ x\in V_{\mz}^{2^k} : x\sim V_u^{2^k} \right\} \right| > K \right\}\\
		&\notag
		\cup
		\left\{\|x-y\|_\infty \leq \delta_1 2^k \text{ for some } x,y \in V_{\mz}^{2^k} \text{ with } x\sim V_u^{2^k}, y\sim V_v^{2^k} \right\}\\
		&\notag
		\cup
		\Big(\left\{D_{V_{\mz}^{2^k}}\left(x,B_{\delta_1  2^k} (x)^c \right) \leq \delta 2^{k\dxp(\beta)} \text{ for some } x \in V_{\mz}^{2^k} \text{ with } x\sim V_u^{2^k} \right\} \\
		& \notag \hspace{10mm} \cap \left\{ \left| \left\{ x\in V_{\mz}^{2^k} : x\sim V_u^{2^k} \right\} \right| \leq K  \right\} \Big).
	\end{align}
	We now upper bound the probability of each of the three events listed in the above union \eqref{union}. For the first event, note that under the measure $\p_{\beta < k}^{\beta+\eps \geq k}$, the random variable $X\coloneqq \left| \left\{ x\in V_{\mz}^{2^k} : x\sim V_u^{2^k} \right\} \right|$ is the sum of independent Bernoulli random variables with
	\begin{align}\label{eq:expectation mixed}
		\notag
		\E_{\beta\leq k}^{\beta+\eps>k} & \left[  \left| \left\{ x\in V_{\mz}^{2^k} : x \sim V_u^{2^k} \right\} \right| \right] \leq \sum_{x\in V_{\mz}^{2^k}} \sum_{ z \in V_u^{2^k}} \p_{\beta+1} \left( x \sim z \right)
		\overset{\eqref{eq:connectionupper bound}}{\leq}
		\sum_{x\in V_{\mz}^{2^k}} \sum_{ z \in V_u^{2^k}} \frac{(\beta+1)2^{2d}}{\|x-z\|_\infty ^{2d}}\\
		&
		\leq \sum_{x\in V_{\mz}^{2^k}} \sum_{ z \in V_u^{2^k}} \frac{(\beta+1)2^{2d}}{\left((\|u\|_\infty -1) 2^k \right)^{2d}}
		\leq 2^{2d} \frac{(\beta+1) 2^{2d}}{\|u\|_\infty^{2d}} \leq (\beta+1) 2^{4d}.
	\end{align}
	Under the measure $\p^{u,v}$, the random variable $X = \left| \left\{ x\in V_{\mz}^{2^k} : x \sim V_u^{2^k} \right\} \right|$ is not the sum of independent Bernoulli random variables, as the measure $\p^{u,v}$ conditions on the event $V_\mz^{2^k} \sim V_u^{2^k}$, i.e., $X = \left| \left\{ x\in V_{\mz}^{2^k} : x \sim V_u^{2^k} \right\} \right| \geq 1$. The conditioning on the event $V_\mz^{2^k} \sim V_v^{2^k}$ does not change the distribution of $\left| \left\{ x\in V_{\mz}^{2^k} : x \sim V_u^{2^k} \right\} \right|$, as the status of the edges between $V_\mz^{2^k}$ and $V_v^{2^k}$, respectively between $V_\mz^{2^k}$ and $V_v^{2^k}$ are independent. 
	Thus, using Lemma \ref{lem:number of edges conditioned}, we get that
	\begin{align*}
		& \p^{u,v} \left(  \left| \left\{ x\in V_{\mz}^{2^k} : x\sim V_u^{2^k} \right\} \right| > K \right) = \p^{u,v} \left( X > K \right)
		=
		\p_{\beta < k}^{\beta+\eps \geq k} \left( X > K  \Big|  X \geq 1 \right)
		\\
		&
		\leq \frac{2 \left(  \E_{\beta\leq k}^{\beta+\eps>k}  \left[ X \right] \vee 1\right)}{K}
		\overset{\eqref{eq:expectation mixed}}{\leq}
		\frac{2 \left(  ((\beta+1) 2^{4d}  ) \vee 1\right)}{K}
		\leq
		\frac{  (\beta+1) 2^{5d}  }{K}.
	\end{align*}
	For the second term in the union \eqref{union} we have that
	\begin{align*}
		\p^{u,v} \left( \|x-y\|_\infty < \delta_1 2^k \text{ for some } x,y \in V_{\mz}^{2^k} \text{ with } x\sim V_u^{2^k}, y\sim V_v^{2^k} \right) \leq C^\prime \sqrt{\delta_1}
	\end{align*}
	by Lemma \ref{lem:linearspacingmixed}. For the last term in \eqref{union}, first note that the random variable $X =  \left| \left\{ x\in V_{\mz}^{2^k} : x\sim V_u^{2^k} \right\} \right|$ does not reveal any information about the structure inside the set $V_\mz^{2^k}$. Thus
	\begin{align*}
		&
		\p^{u,v} \Big(\left\{ D_{V_{\mz}^{2^k}}\left(x,B_{\delta_1 2^k}(x)^c\right) \leq \delta 2^{k\dxp(\beta)} \text{ for some } x \in V_{\mz}^{2^k} \text{ with } x\sim V_u^{2^k} \right\} \cap \left\{ X \leq K  \right\} \Big)  \\
		& \leq
		\p^{u,v} \Big( D_{V_{\mz}^{2^k}}\left(x,B_{\delta_1 2^k}(x)^c\right) \leq \delta 2^{k\dxp(\beta)} \text{ for some } x \in V_{\mz}^{2^k} \text{ with } x\sim V_u^{2^k}  \Big|  X \leq K \Big)\\
		& \leq K \p_{\beta} \left( D \left(\mz,B_{\delta_1 2^k}(\mz)^c\right) \leq \delta 2^{k\dxp(\beta)} \right)  \text .
	\end{align*}
	In total, we see that we can bound the probability of the event considered in \eqref{union} by
	\begin{equation*}
		\p^{u,v} \left(\mathcal{A}(K,\delta_1,\delta)^c\right)
		\leq
		\frac{  (\beta+1) 2^{5d}  }{K}
		+
		C^\prime \sqrt{\delta_1}
		+
		K \p_{\beta} \left( D \left(\mz,B_{\delta_1 2^k}(\mz)^c\right) \leq \delta 2^{k\dxp(\beta)} \right) .
	\end{equation*}
	
	Using that $D\left(\mz,B_\ell(\mz)^c\right) \approx_P \ell^{\dxp(\beta)}$ (cf. \eqref{point to box}), we see that for fixed $\delta_1>0$ the term $\p_\beta \left( D \left(\mz,B_{\delta_1 2^k}(\mz)^c\right) \leq \delta 2^{k\dxp(\beta)} \right)$ converges to $0$ as $\delta \to 0$ and thus we can take $K=K(\delta)$ and $\delta_1=\delta_1(\delta)$ that converge to $+\infty$, respectively $0$, slow enough such that the product $K \p_\beta \left( D  \left(\mz,B_{\delta_1 2^k}(\mz)^c\right) \leq \delta 2^{k\dxp(\beta)} \right)$ also converges to $0$ for $\delta \to 0$.
\end{proof}

Next, we prove Lemma \ref{lem:goodeventD}.

\begin{proof}[Proof of Lemma \ref{lem:goodeventD}]
	Let $E=\left\{\{u,v\}\subset \Z^d: u\neq v\right\}$. Let $\omega \in \{0,1\}^E$ be distributed according to the measure $\p_\beta$, i.e., $\omega$ has independent increments with 
	\begin{equation*}
		\p(\omega(\{u,v\})=1)=p(\beta,\{u,v\})=1-\exp\left(-\beta \int_{u+\cC} \int_{v+\cC} \frac{1}{\|t-s\|} \md t \md s\right).
	\end{equation*}
	Further, let $\psi \in \{0,1\}^E$ be independent of $\omega$ with independent increments such that
	\begin{align*}
		\p(\psi(\{u,v\})=1)=
		\begin{cases}
			1-\exp\left(- \eps \int_{u+\cC} \int_{v+\cC} \frac{1}{\|t-s\|} \md t \md s\right) & \text{ if $\|u-v\|_\infty \geq 2^k$}\\
			0 & \text{ if $\|u-v\|_\infty < 2^k$}
		\end{cases}.
	\end{align*}
	We define $\omega^+=\omega \vee \psi$ as the union of these two percolation environments. Note that $\omega^+$ has exactly the same distribution as the percolation environment sampled by $\p_{\beta < k}^{\beta+\eps \geq k}$ and thus we only need to show that for each fixed $c>0$ there exists $\delta >0$ and $k(\delta) \in \N$ such that for all large enough $k\geq k(\delta)$ and all realizations of $G^\prime$ with $\deg^{\cN}(r(\mz))\leq 9^d 100 \mu_{\beta+1}$ one has
	\begin{equation*}
		\p \left( D^\star \left( V_\mz^{2^k} , \bigcup_{u \in \Z^d : \| u \|_\infty \geq 2} V_u^{2^k} ; \omega^+ \right) \leq \delta 2^{k\dxp(\beta)} \Big| G^\prime \right) < c .
	\end{equation*}
	We spit the environment $\omega^+$ into a {\sl lower environment} $\omega^\downarrow$ and into an {\sl upper environment} $\omega^\uparrow$, which are defined by
	\begin{align*}
		\omega^\downarrow (e) = \omega^+(e) \mathbbm{1}_{\{|e| < 2^k\}} = \begin{cases}
			\omega^+(e) & \text{ if } |e| < 2^k \\
			0 & \text{ if } |e| \geq 2^k
		\end{cases},
	\end{align*}
	respectively by
	\begin{align*}
		\omega^\uparrow (e) = \omega^+(e) \mathbbm{1}_{\{|e| \geq 2^k\}} = \begin{cases}
			0 & \text{ if } |e| < 2^k \\
			\omega^+(e) & \text{ if } |e| \geq 2^k
		\end{cases}.
	\end{align*}
	Note that the environment $\omega^\downarrow$ contains the short edges (of length at most $2^k-1$) whereas the environment $\omega^\uparrow$ contains the long edges (of length at least $2^k$). When taking their union, we have that $\omega^\downarrow \vee \omega^\uparrow = \omega^+ = \omega \vee \psi$.
	We also define the set
	\begin{equation*}
		\mathcal{V}(\omega^\uparrow) = \left\{u\in \Z^d: \exists v \in \Z^d \text{ with } \omega^\uparrow(\{u,v\})=1\right\}
	\end{equation*}
	as the set of endpoints of edges in the edge set defined by $\omega^\uparrow$. Further, for $s\in \left[0,1\right]$ and $k\in \N$ we define the set $A_s^k$  
	\begin{equation*}
		A_s^k=\left\{x\in \Z^d: s2^k < \inf_{y\in V_\mz^{2^k}} \|x-y\|_\infty \leq 2^k\right\}.
	\end{equation*}
	Note that the set $A_s^k$ always satisfies $A_s^k\subseteq \bigcup_{v\in \Z^d : \|v\|_\infty =1} V_v^{2^k}$, with an equality for $s=0$. 
	Define the event
	\begin{align}
		\label{line1} \widetilde{\mathcal{B}}(s,\delta) & = \left\{ D^\star \left( V_\mz^{2^k} , \bigcup_{u \in \Z^d : \| u \|_\infty \geq 2} V_u^{2^k} ; \omega^\downarrow \right) \leq \delta 2^{k\dxp(\beta)} \right\}\\
		&
		\cup \label{line2}
		\bigcup_{x \in A_s^k} \left\{x \in \mathcal{V}(\omega^\uparrow), D\left(x, B_{s 2^k}(x)^c ;\omega^\downarrow \right) \leq \delta 2^{k\dxp(\beta)} \right\}
		\\
		& \label{line3}
		\cup
		\bigcup_{x \in A_0^k \setminus A_s^k} \left\{x \in \mathcal{V}(\omega^\uparrow) \right\} .
	\end{align}
	Now note that for every $s\in \left[0,1\right]$ one has $\left\{D^\star \left( V_\mz^{2^k} , \bigcup_{u \in \Z^d : \| u \|_\infty \geq 2} V_u^{2^k} ; \omega^+ \right) \leq \delta 2^{k\dxp(\beta)}\right\} \subseteq \widetilde{\mathcal{B}}(s,\delta)$. Indeed if, $D^\star \left( V_\mz^{2^k} , \bigcup_{u \in \Z^d : \| u \|_\infty \geq 2} V_u^{2^k} ; \omega^+ \right) \leq \delta 2^{k\dxp(\beta)}$ then there exists either a path from $ V_\mz^{2^k} $ to $\bigcup_{u \in \Z^d : \| u \|_\infty \geq 2} V_u^{2^k}$ which is entirely contained in the environment $\omega^\downarrow$ - which corresponds to the event described in \eqref{line1} - or this path uses at least one of the edges in the environment $\omega^\uparrow$. Assume that $P$ is a path from $ V_\mz^{2^k} $ to $\bigcup_{u \in \Z^d : \| u \|_\infty \geq 2} V_u^{2^k}$ that has length at most $\delta 2^{k\theta(\beta)}$. Let $e=\{x,y\}$ be the first edge in the environment $\omega^\uparrow$ that is used in this path, with the path entering the edge at $x$ and leaving at $y$. Then either $x\in A_0^k \setminus A_s^k$ (which corresponds to the event in \eqref{line3}) or $x\in A_s^k$. For $x\in A_s^k$ one has $B_{s2^k}(x)^c \supseteq V_\mz^{2^k}$ and thus 
	\begin{equation*}
		D\left(x, B_{s 2^k}(x)^c ;\omega^\downarrow \right) \leq D\left(x, V_\mz^{2^k}; \omega^\downarrow\right) \leq \text{length}(P) \leq \delta 2^{k\dxp(\beta)}
	\end{equation*}
	(which corresponds to the event in \eqref{line2}). Thus we only need to show that  $\p\left(\widetilde{\mathcal{B}}(s,\delta) \big| G^\prime\right) < c$ for a suitable choice of $s$ and $\delta$. We do this via a union bound over the three events listed in \eqref{line1} through \eqref{line3}.\\
	
	For $x\in A_0^k$, say with $x \in V_v^{2^k}$ and $u\in \Z^d$ with $\|u-v\|_\infty \geq 2$ the condition that $\|u-v\|_\infty \geq 2$ directly implies that for all $y\in V_u^{2^k}$ one has  
	\begin{equation*}
		\|x-y\|_\infty > (\|u-v\|_\infty-1) 2^k \geq \|u-v\|_\infty 2^{k-1} \geq 2^k.
	\end{equation*}
	So in particular, the edge $\{x,y\}$ can only be open in the environment $\omega^{\uparrow}$, but not in the environment $\omega^\downarrow$. Further, the probability that the edge $\{x,y\}$ is open is bounded by
	\begin{equation*}
		\p \left( \omega^{\uparrow}(\{x,y\}) = 1 \right) = \p_{\beta+\eps} (x\sim y) \overset{\eqref{eq:connectionupper bound}}{\leq} \frac{2^{2d} (\beta+\eps)}{\|x-y\|^{2d}_\infty}
		\leq
		\frac{2^{2d} (\beta+1)}{\left(2^{k-1} \|u-v\|_\infty\right)^{2d}} 
		= 
		\frac{2^{4d} (\beta +1)}{\left(2^{k} \|u-v\|_\infty\right)^{2d}} .
	\end{equation*}
	A union bound over all possible $y\in V_u^{2^k}$ gives that
	\begin{align}\label{eq:psi ineq}
		\p \left( \exists y \in V_u^{2^k}: \omega^\uparrow(\{x,y\}) = 1 \right) & \notag
		\leq
		\sum_{y \in V_u^{2^k}}
		\p \left( \omega^\uparrow(\{x,y\}) = 1 \right)
		\leq
		\sum_{y \in V_u^{2^k}} \frac{2^{4d} (\beta+1) }{\left(2^{k} \|u-v\|_\infty\right)^{2d}}
		\\
		&
		=
		\frac{2^{4d} (\beta+1) }{\|u-v\|_\infty^{2d} \ 2^{kd}}.
	\end{align}
	Further, we have 
	\begin{equation*}
		\p \left(   V_u^{2^k} \sim V_v^{2^k} \text{ in } \omega^+ \right)
		\geq
		\p_{\beta} \left( V_u^{2^k} \sim V_v^{2^k} \right) = \p_\beta (u \sim v) \geq \frac{(4d)^{-2d} \beta}{\|u-v\|_\infty^{2d}} \wedge \frac{1}{2} ,
	\end{equation*}
	where the last inequality follows by \eqref{upperlowerbeta}.
	So in particular we see that for all $x \in V_v^{2^k}$ and all $u\in \Z^d$ with $\|u-v\|_\infty \geq 2$ one has
	\begin{align}\label{dist2}
		&\notag \p \left( \exists y \in V_u^{2^k}: \omega^\uparrow(\{x,y\}) = 1 \Big|  V_u^{2^k} \sim V_v^{2^k} \text{ in } \omega^+ \right)
		=
		\frac{\p \left( \exists y \in V_u^{2^k}: \omega^\uparrow(\{x,y\}) = 1 \right) }{\p \left(   V_u^{2^k} \sim V_v^{2^k} \text{ in } \omega^+ \right)}\\
		&
		\leq
		\frac{\frac{2^{4d} (\beta+1) }{\|u-v\|_\infty^{2d} \ 2^{kd}}}{\frac{(4d)^{-2d} \beta}{\|u-v\|_\infty^{2d}} \wedge \frac{1}{2}}
		=
		\frac{2^{4d+1} (\beta+1) }{\|u-v\|_\infty^{2d} \ 2^{kd}}
		\vee
		\frac{2^{8d} d^{2d} (\beta+1) }{ 2^{kd} \beta} \leq \frac{(2d)^{8d} (\beta+1)}{(\beta \wedge 1) 2^{kd}}.
	\end{align}
	When $\|u-v\|_\infty = 1$, then for all $x\in V_v^{2^k}$
	\begin{align*}
		& \p \left( \exists y \in V_u^{2^k}: \omega^\uparrow(\{x,y\}) = 1 \right) 
		\leq
		\sum_{\substack{y \in V_u^{2^k} : \\ \|x-y\|_\infty \geq  2^k}}
		\p \left( \omega^\uparrow(\{x,y\}) = 1 \right)
		\leq
		\sum_{\substack{y \in V_u^{2^k} : \\ \|x-y\|_\infty \geq  2^k}}
		\p_{\beta+\eps} \left( x \sim y \right)
		\\
		&
		\overset{\eqref{eq:connectionupper bound}}{\leq}
		\sum_{\substack{y \in V_u^{2^k} : \\ \|x-y\|_\infty \geq  2^k}} \frac{2^{2d} (\beta+\eps) }{(2^k)^{2d}} 
		\leq
		\left|V_u^{2^k}\right| \frac{2^{2d} (\beta+1) }{(2^k)^{2d}}
		=
		\frac{2^{2d} (\beta+1) }{2^{kd}} .
	\end{align*}
	As the event $\left\{V_u^{2^k} \sim V_v^{2^k} \text{ in }\omega^+\right\}$ has probability $1$ for $\|u-v\|_\infty=1$, we get that
	\begin{equation}\label{dist1}
		\p \left( \exists y \in V_u^{2^k}: \omega^\uparrow(\{x,y\}) = 1 \big|  V_u^{2^k} \sim V_v^{2^k} \text{ in } \omega^+ \right) \leq
		\frac{2^{2d} (\beta+1) }{2^{kd}} 
	\end{equation}
	for $\|u-v\|_\infty =1$. Combining \eqref{dist2} and \eqref{dist1}, we see that for all $u,v\in \Z^d$ with $u\neq v$ and for all $x\in V_v^{2^k}$ one has
	\begin{equation*}
		\p \left( \exists y \in V_u^{2^k}: \omega^\uparrow(\{x,y\}) = 1 \big|  V_u^{2^k} \sim V_v^{2^k} \text{ in } \omega^+ \right) \leq \frac{(2d)^{8d}(\beta+1)}{(\beta \wedge 1) 2^{kd}} \eqqcolon \frac{C_\beta}{2^{kd}} .
	\end{equation*}
	For $x\in V_v^{2^k}$ and $y\in V_u^{2^k}$ with $\omega^\uparrow(\{x,y\}) = 1 $ one already needs to have $\|u-v\|_\infty \geq 1$ (since $\|x-y\|\geq 2^k$) and $r(u) \sim r(v)$ in $G^\prime$. Further, conditioned on $G^\prime$, whenever $r(u) \sim r(v)$, the probability that there exists $y\in V_u^{2^k}$ with $\omega^\uparrow(\{x,y\}) = 1 $ is exactly $\p \left( \exists y \in V_u^{2^k}: \omega^\uparrow(\{x,y\}) = 1 \big|  V_u^{2^k} \sim V_v^{2^k} \text{ in } \omega^+ \right)$. Thus we get that for all $v\in \Z^d$ with $\|v\|_\infty=1$ and all $x\in V_v^{2^k}$
	\begin{align}\label{eq:upperenvironm bound}
		\notag & \p \left( x \in \mathcal{V}(\omega^{\uparrow}) \big|  G^\prime \right)
		=
		\p \left( \exists y \in \Z^d: \omega^\uparrow(\{x,y\}) = 1 \big|  G^\prime \right)
		\\
		& \notag
		\leq
		\sum_{u\in \Z^d : r(u) \sim r(v)} \p \left( \exists y \in V_u^{2^k}: \omega^\uparrow(\{x,y\}) = 1 \big|  G^\prime \right)\\
		\notag &
		=
		\sum_{u\in \Z^d : r(u) \sim r(v)} \p \left( \exists y \in V_u^{2^k}: \omega^\uparrow(\{x,y\}) = 1 \big|  V_u^{2^k} \sim V_v^{2^k} \text{ in } \omega^+ \right)
		\\
		&
		\leq
		\sum_{u\in \Z^d : r(u) \sim r(v)} \frac{C_\beta}{2^{kd}}
		=
		\deg(r(v)) \frac{C_\beta}{2^{kd}}.
	\end{align}
	For each $v\in \Z^d$ with $\|v\|_\infty=1$, the set $ (A_0^k \setminus A_s^k) \cap V_v^{2^k} = \left\{x \in V_v^{2^k} : \inf_{y\in V_\mz^{2^k}} \|x-y\|_\infty \leq s 2^k \right\}$ has size at most $s 2^{kd}$. Indeed, assume that $v=(v_1,\ldots,v_d)$ satisfies $\|v\|_\infty=1$. By symmetry, we can assume that $v_1=+1$. If $x = (x_1,\ldots,x_d)\in V_v^{2^k}$ satisfies $\inf_{y\in V_\mz^{2^k}} \|x-y\|_\infty \leq s 2^k$, then $x_1 \in \left( 2^k-1, 2^k-1+ s2^k \right]\cap \Z$ and $x_i \in \{v_i 2^k, \ldots, v_i 2^k + 2^k-1\}$. Thus, there are at most $s 2^k \cdot (2^k)^{d-1} = s2^{kd}$ many choices for $x=(x_1,x_2,\ldots,x_d)$.\\
	
	Using that $\left|(A_0^k \setminus A_s^k) \cap V_v^{2^k}\right| \leq s 2^{kd}$ and inequality \eqref{eq:upperenvironm bound}, we can bound the conditional probability of the event in \eqref{line3} by
	\begin{align}\label{line3bound}
		\notag & \p \left( \bigcup_{x \in A_0^k \setminus A_s^k} \{x\in \mathcal{V}(\omega^\uparrow)\} \Big| G^\prime \right)
		\leq
		\sum_{x \in A_0^k \setminus A_s^k}
		\p \left( x\in \mathcal{V}(\omega^\uparrow) \Big| G^\prime \right)
		\\
		& \notag
		=
		\sum_{v : \|v\|_\infty = 1} \
		\sum_{x \in (A_0^k \setminus A_s^k) \cap V_v^{2^k}}
		\p \left( x\in \mathcal{V}(\omega^\uparrow) \Big| G^\prime \right)
		\overset{\eqref{eq:upperenvironm bound}}{\leq}
		\sum_{v : \|v\|_\infty = 1} \
		\sum_{x \in (A_0^k \setminus A_s^k) \cap V_v^{2^k}}
		\deg(r(v)) \frac{C_\beta}{2^{kd}}
		\\
		& \notag
		=
		\sum_{v : \|v\|_\infty = 1}
		\left| \left(A_0^k \setminus A_s^k\right) \cap V_v^{2^k} \right|
		\deg(r(v)) \frac{C_\beta}{2^{kd}}
		\leq
		\sum_{v : \|v\|_\infty = 1}
		s 2^{kd}
		\deg(r(v)) \frac{C_\beta}{2^{kd}}
		\\
		&
		=
		s C_\beta \sum_{v : \|v\|_\infty = 1} \deg(r(v)) 
		\leq
		s C_\beta \deg^{\cN}(r(\mz)) .
	\end{align} 
	Further, note that the events of the form $\left\{D\left(x, B_{s 2^k}(x)^c ; \omega^\downarrow \right) \leq \delta 2^{k\theta(\beta)}\right\}$ are independent of both $G^\prime$ and $\omega^\uparrow$, as the graph $G^\prime$ depends almost surely just on edges that are introduced by the environment $\omega^\uparrow$. 
	Thus we get that
	\begin{align}\label{eq:lowerupperenvironments}
		& \notag \p \left( \bigcup_{x \in A_s^k} \left\{x \in \mathcal{V}(\omega^\uparrow), D\left(x, B_{s 2^k}(x)^c ;\omega^\downarrow \right) \leq \delta 2^{k\dxp(\beta)} \right\} \Big| G^\prime \right)\\
		& \notag
		\leq
		\sum_{x \in A_s^k}
		\p \left(  x \in \mathcal{V}(\omega^\uparrow), D\left(x, B_{s 2^k}(x)^c ;\omega^\downarrow \right) \leq \delta 2^{k\dxp(\beta)}  \Big| G^\prime \right)
		\\
		& \notag
		=
		\sum_{x \in A_s^k}
		\p\left( D\left(x, B_{s 2^k}(x)^c ;\omega^\downarrow \right) \leq \delta 2^{k\dxp(\beta)}\right)
		\p \left( x \in \mathcal{V}(\omega^\uparrow) \Big| G^\prime \right)\\
		&
		=
		\p\left( D\left(\mz, B_{s 2^k}(\mz)^c ;\omega^\downarrow \right) \leq \delta 2^{k\dxp(\beta)}\right)
		\sum_{v: \|v\|_\infty = 1} \
		\sum_{x \in A_s^k \cap V_v^{2^k}}
		\p \left( x \in \mathcal{V}(\omega^\uparrow) \Big| G^\prime \right).
	\end{align}
	The random variable $D\left(\mz, B_{s 2^k}(\mz)^c ;\omega^\downarrow \right)$ stochastically dominates the random variable $D\left(\mz, B_{s 2^k}(\mz)^c ;\omega \right)$, as all edges contained in the environment $\omega^{\downarrow}$ are also contained in the environment $\omega$.
	Let $q_1:\R_{\geq 0}\to \left[0,1\right]$ be a function such that $\lim_{\delta \searrow 0} q_1(\delta)=0$ and
	\begin{equation*}
		\p_{\beta}\left( D\left(\mz, B_m(\mz)^c \right) \leq \delta m^{\dxp(\beta)} \right) \leq q_1(\delta) .
	\end{equation*}
	for all $m\in \N$.
	Such a function exists since $ D\left(\mz, B_m(\mz)^c \right) \approx_P m^{\theta(\beta)}$, see \eqref{point to box}. Then we get that
	\begin{align*}
		& \p\left( D\left(\mz, B_{s 2^k}(\mz)^c ;\omega^\downarrow \right) \leq \delta 2^{k\dxp(\beta)}\right) \leq \p\left( D\left(\mz, B_{s 2^k}(\mz)^c ;\omega \right) \leq \delta 2^{k\dxp(\beta)}\right) \\
		&
		=
		\p_\beta \left( D\left(\mz, B_{s 2^k}(\mz)^c  \right) \leq \frac{\delta}{s^{\theta(\beta)}} \left(s2^{k}\right)^{\dxp(\beta)} \right)
		\leq
		q_1\left(\frac{\delta}{s^{\dxp(\beta)}} \right) .
	\end{align*}
	Plugging this into inequality \eqref{eq:lowerupperenvironments}, we get that
	\begin{align}\label{line2bound}
		&\notag \p \left( \bigcup_{x \in A_s^k} \left\{x \in \mathcal{V}(\omega^\uparrow), D\left(x, B_{s 2^k}(x)^c ;\omega^\downarrow \right) \leq \delta 2^{k\dxp(\beta)} \right\} \Big| G^\prime \right)\\
		& \notag
		\leq
		q_1\left(\frac{\delta}{s^{\dxp(\beta)}} \right)
		\sum_{v: \|v\|_\infty = 1} \
		\sum_{x \in A_s^k \cap V_v^{2^k}}
		\p \left( x \in \mathcal{V}(\omega^\uparrow) \Big| G^\prime \right)\\
		& \notag
		\overset{\eqref{eq:upperenvironm bound}}{\leq}
		q_1\left(\frac{\delta}{s^{\dxp(\beta)}} \right)
		\sum_{v: \|v\|_\infty = 1} \
		\sum_{x \in A_s^k \cap V_v^{2^k}}
		\deg(r(v)) \frac{C_\beta}{2^{kd}}
		\leq
		q_1\left(\frac{\delta}{s^{\dxp(\beta)}} \right)
		\sum_{v: \|v\|_\infty = 1} 
		\deg(r(v)) C_\beta\\
		&
		\leq
		q_1\left(\frac{\delta}{s^{\dxp(\beta)}} \right)  C_\beta \deg^{\cN}(r(\mz)) .
	\end{align}
	This bounds the probability of the event described in line \eqref{line2}. In order to bound the probability of the event described in line \eqref{line1}, let $q_2:\R_{\geq 0}\to \left[0,1\right]$ be a function so that $\lim_{\delta \searrow 0} q_2(\delta)=0$ and
	\begin{equation*}
		\p_{\beta}\left( D^\star\left(V_\mz^{m}, \bigcup_{u \in \Z^d : \| u \|_\infty \geq 2} V_u^{m} \right) \leq \delta m^{\dxp(\beta)} \right) \leq q_2(\delta)
	\end{equation*}
	for all $m\in \N$.
	Such a function exists by Lemma \ref{lem:all quantiles of boxtobox}. Since the environment $\omega^\downarrow$ and the graph $G^\prime$ are independent, we get that
	\begin{align}\label{line1bound}
		&\notag \p \left(  D^\star \left( V_\mz^{2^k} , \bigcup_{u \in \Z^d : \| u \|_\infty \geq 2} V_u^{2^k} ; \omega^\downarrow \right) \leq \delta 2^{k\dxp(\beta)} \Bigg| G^\prime \right)
		=
		\p \left(  D^\star \left( V_\mz^{2^k} , \bigcup_{u \in \Z^d : \| u \|_\infty \geq 2} V_u^{2^k} ; \omega^\downarrow \right) \leq \delta 2^{k\dxp(\beta)}  \right)
		\\
		&
		\leq
		\p_\beta \left(  D^\star \left( V_\mz^{2^k} , \bigcup_{u \in \Z^d : \| u \|_\infty \geq 2} V_u^{2^k} \right) \leq \delta 2^{k\dxp(\beta)} \right) \leq q_2(\delta).
	\end{align}
	Inequalities \eqref{line1bound}, \eqref{line2bound}, and \eqref{line3bound} imply that the conditional probability of the event $\widetilde{\mathcal{B}}(s,\delta)$ defined in \eqref{line1} is bounded by
	\begin{align*}
		\p \left( \widetilde{\mathcal{B}}(s,\delta) | G^\prime \right) & \leq q_2(\delta)
		+
		q_1\left(\frac{\delta}{s^{\dxp(\beta)}} \right)  C_\beta \deg^{\cN}(r(\mz))
		+
		s C_\beta \deg^{\cN}(r(\mz)).
	\end{align*}
	Assuming that $\deg^{\cN}(r(\mz)) \leq 9^d 100 \mu_{\beta+1}$, we can thus bound the conditional probability of the event $\widetilde{\mathcal{B}}(s,\delta)$ by	
	\begin{equation*}
		\p \left( \widetilde{\mathcal{B}}(s,\delta) | G^\prime \right) 
		\leq q_2(\delta)
		+
		C_\beta 9^d 100 \mu_{\beta+1} \left(
		q_1 \left(\frac{\delta}{s^{\dxp(\beta)}} \right) 
		+
		s \right) .
	\end{equation*}
	This expression can be made arbitrarily small for suitable choice of $s,\delta >0$, as $\lim_{\delta \searrow 0} q_1(\delta) = \lim_{\delta \searrow 0} q_2(\delta) = 0$.
\end{proof}

Finally, we prove Lemma \ref{lem:uniformovereps}. We prove the three inequalities (\eqref{eq:trivia}, \eqref{trivia2}, and \eqref{eq:uniformovereps}) separately. We start with the proof of inequality \eqref{eq:trivia}.

\begin{proof}[Proof of inequality \eqref{eq:trivia}]
	The inequality is clear when $n=k$, since the left hand side of the inequality equals $0$ when $n-k=0$. For $k<n$, it suffices to show that 
	\begin{equation}\label{reno}
		\E_{\beta<k}^{\beta+\eps \geq k} \left[D_{V_\mz^{2^n}} \left(\mz,(2^n-1)\mo\right)\right] \geq \E_{\beta+\eps} \left[D_{V_\mz^{2^{n-k}}} \left(\mz,(2^{n-k}-1)\mo\right)\right]  ,
	\end{equation}
	since it is known by \cite[Remark 4.3]{baeumler2022distances} that
	\begin{equation*}
		\E_{\beta+\eps} \left[D_{V_\mz^{2^{n-k}}} \left(\mz,(2^{n-k}-1)\mo\right)\right] \geq \frac{\Lambda\left(2^{n-k},\beta+\eps\right)-1}{36 d} \geq
		\frac{\Lambda\left(2^{n-k},\beta+\eps\right)}{72 d} .
	\end{equation*}
	In order to show inequality \eqref{reno}, we define the graph $G^\prime=(V^\prime, E^\prime)$ by starting with the random graph with vertex set $V_\mz^{2^n}$ and contracting the boxes $\left(V_u^{2^k}\right)_{u\in V_\mz^{2^{n-k}}}$. We write $r(u)$ for the vertex in $V^\prime$ that corresponds to the set $V_u^{2^k}$. This construction directly implies that $D_{V_\mz^{2^{n}}} \left(\mz,(2^{n}-1)\mo\right) \geq D_{G^\prime} (r(\mz),r((2^{n-k}-1)\mo))$. The graph $G^\prime$ has the exact same distribution as long-range percolation with measure $\p_{\beta+\eps}$ on $V_\mz^{2^{n-k}}$ and thus
	\begin{align*}
		\E_{\beta<k}^{\beta+\eps \geq k} \left[D_{V_\mz^{2^{n}}} \left(\mz,(2^{n}-1)\mo\right)\right] & \geq 
		\E_{\beta<k}^{\beta+\eps \geq k} \left[D_{G^\prime} (r(\mz),r((2^{n-k}-1)\mo))\right]
		\\
		&
		=
		\E_{\beta+\eps} \left[D_{V_\mz^{2^{n-k}}} \left(\mz,(2^{n-k}-1)\mo\right)\right].
	\end{align*}
\end{proof}

\begin{proof}[Proof of inequality \eqref{trivia2}]
	The proof of the inequality 
	\begin{align*}
		\E_{\beta < k}^{\beta+\eps \geq k} \left[D_{V_\mz^{2^n}}\left(\mz,(2^n-1) \mo \right)\right]  \leq
		C(\beta) \Lambda\left(2^k,\beta\right)  \Lambda\left(2^{n-k},\beta + \eps\right)
	\end{align*}
	directly follows from Lemma \ref{lem:uniform2ndmomentbound:renormalized} and the Cauchy-Schwarz-inequality. Indeed, let $C_\beta$ be the constant of Lemma \ref{lem:uniform2ndmomentbound:renormalized}. Then we get by the Cauchy-Schwarz inequality and inequality \eqref{eq:uniformoverbeta} that
	\begin{align*}
		&\E_{\beta < k}^{\beta+\eps \geq k} \left[D_{V_\mz^{2^n}}\left(\mz,(2^n-1) \mo \right)\right] 
		\leq \sqrt{\E_{\beta < k}^{\beta+\eps \geq k} \left[D_{V_\mz^{2^n}}\left(\mz,(2^n-1) \mo \right)^2 \right]} 
		\\
		&
		\overset{\eqref{eq:uniformoverbeta}}{\leq}\sqrt{C_\beta  \Lambda(2^k,\beta)^2 \Lambda(2^{n-k},\beta+\eps)^2}
		=
		\sqrt{C_\beta}  \Lambda(2^k,\beta) \Lambda(2^{n-k},\beta+\eps).
	\end{align*}
\end{proof}

\begin{proof}[Proof of inequality \eqref{eq:uniformovereps}]
	We are left to show that 
	\begin{align}\label{eq:lower bound}
		\E_{\beta < k}^{\beta+\eps \geq k} \left[D_{V_\mz^{2^n}}\left(\mz,(2^n-1) \mo \right)\right]  \geq
		c_\beta \Lambda\left(2^k,\beta\right)  \Lambda\left(2^{n-k},\beta + \eps\right) .
	\end{align}
	The proof of inequality \eqref{eq:lower bound} is divided into four parts. In part (A), we argue that it is sufficient to prove inequality \eqref{eq:lower bound} for $k,n-k$ large. In part (B), we define a notion of {\sl good blocks} in paths $P^\prime \subset G^\prime$. In part (C), we show that all paths $P^\prime \subset G^\prime$ of length $\ell$ starting at $r(\mz)$ contain many good blocks with probability approaching 1 exponentially fast, as $\ell$ increases. In part (D), we show how this implies inequality \eqref{eq:lower bound} for large $k,n-k$.\\
	
	\noindent
	\textbf{Part (A)}.
	We first argue that it suffices to consider $k,n-k$ large enough.
	When $k$ is small, this directly follows from inequality \eqref{eq:trivia}. To treat the case when $n-k$ is small, define the event
	\begin{equation*}
		\mathcal{D} = \left\{x\nsim y \text{ for all } x,y \in V_\mz^{2^n} \text{ with } |\{x,y\}| \geq 2^k \right\}.
	\end{equation*}
	Since
	\begin{align*}
		& \sum_{\substack{ x,y \in V_\mz^{2^n} : |\{x,y\}| \geq 2^k}} \p_{\beta < k}^{\beta+\eps \geq k} \left(x\sim y\right)
		=
		\sum_{\substack{ x,y \in V_\mz^{2^n} : |\{x,y\}| \geq 2^k}} \p_{\beta + \eps} \left(x\sim y\right)
		\\
		&
		\overset{\eqref{eq:connectionupper bound}}{\leq}
		\sum_{\substack{ x,y \in V_\mz^{2^n} : |\{x,y\}| \geq 2^k}} \frac{2^{2d}(\beta+\eps)}{(2^k)^{2d}}
		\leq
		\left(2^{nd}\right)^2  \frac{2^{2d}(\beta+\eps)}{(2^k)^{2d}} = 2^{2d(n-k)} 2^{2d} (\beta+1),
	\end{align*}
	and each term in the above sum is at most $0.5$ for $k$ large enough, we get that
	\begin{align*}
		& \p_{\beta < k}^{\beta+\eps \geq k} \left(\mathcal{D}\right) = \prod_{\substack{ x,y \in V_\mz^{2^n} : |\{x,y\}| \geq 2^k}} \left(1- \p_{\beta < k}^{\beta+\eps \geq k} \left(x\sim y\right)\right)
		\geq
		\exp\left(- 2 \sum_{\substack{ x,y \in V_\mz^{2^n} : |\{x,y\}| \geq 2^k}} \p_{\beta < k}^{\beta+\eps \geq k} \left(x\sim y\right)\right)
		\\
		&
		\geq
		\exp\left(-2^{2d(n-k)} 2^{2d+1} (\beta+1) \right),
	\end{align*}
	where we used the elementary inequality $1-s \geq e^{-2s}$ for $s\leq 0.5$. In particular, we see that the event $\mathcal{D}$ has a uniform positive probability for $n-k$ small and $k$ large enough. Next, observe that the conditional measure $ \p_{\beta < k}^{\beta+\eps \geq k}\left(\cdot \big| \mathcal{D}\right)$ and $\p_{\beta}\left(\cdot \big| \mathcal{D}\right)$ are the same, as all edges $e$ with length at least $2^k$ are closed almost surely under this measure, whereas all edges $e$ with length at most $2^{k}-1$ are open with the exact same probability. Further, as both the event $\mathcal{D}$ and the distance $D_{V_\mz^{2^n}}\left(\mz,(2^n-1)\mo\right)$ are decreasing in the percolation environment, we can apply the FKG inequality and get that
	\begin{align*}
		& \E_{\beta < k}^{\beta+\eps \geq k} \left[D_{V_\mz^{2^n}}\left(\mz,(2^n-1) \mo \right)\right] =
		\E_{\beta < k}^{\beta+\eps \geq k} \left[D_{V_\mz^{2^n}}\left(\mz,(2^n-1) \mo \right)\Big| \mathcal{D}\right] \p_{\beta < k}^{\beta+\eps \geq k} \left(\mathcal{D}\right)
		\\
		&
		=
		\E_{\beta} \left[ D_{V_\mz^{2^n}}\left(\mz,(2^n-1) \mo \right)\Big| \mathcal{D} \right] \p_{\beta < k}^{\beta+\eps \geq k} \left(\mathcal{D}\right)
		\geq
		\E_{\beta} \left[ D_{V_\mz^{2^n}}\left(\mz,(2^n-1) \mo \right) \right] \p_{\beta < k}^{\beta+\eps \geq k} \left(\mathcal{D}\right)
		\\
		&
		\overset{\eqref{36d}}{\geq} \frac{\Lambda(2^n,\beta)-1}{36d} 	\exp\left(-2^{2d(n-k)} 2^{2d+1} (\beta+1) \right)
		\geq \frac{\Lambda(2^n,\beta)}{72d} 	\exp\left(-2^{2d(n-k)} 2^{2d+1} (\beta+1) \right)
		\\
		&
		\overset{\eqref{2-comparison}}{\geq} \frac{\Lambda(2^{n-k},\beta)}{72d} 	\exp\left(-2^{2d(n-k)} 2^{2d+1} (\beta+1) \right).
	\end{align*}
	In the last inequality above, we used inequality \eqref{2-comparison}, and in the third to last inequality above, we used that
	\begin{equation}\label{36d}
		\E_{\beta} \left[ D_{V_\mz^{2^n}}\left(\mz,(2^n-1) \mo \right) \right] + 1 \leq 36 d \Lambda(2^n,\beta) ;
	\end{equation}
	This was proven in \cite[Remark 4.3]{baeumler2022distances}. Thus we see that inequality \eqref{eq:lower bound} holds for small $n-k$ with a sufficiently small constant $c_\beta>0$.\\
	
	\noindent
	\textbf{Part (B)}.
	For the rest of the proof, we always assume that both $k$ and $n-k$ are large.
	We split the graph $V_{\mz}^{2^n}$ into blocks of the form $V_v^{2^k}$, where $v\in V_{\mz}^{2^{n-k}}$. For each $v\in V_{\mz}^{2^{n-k}}$, we contract the block $V_{v}^{2^k} \subset V_{\mz}^{2^n}$ into one vertex, denoted by $r(v)$. We call the graph that results from contracting all these blocks $G^\prime = (V^\prime, E^\prime)$. The graph $G^\prime$ has the same distribution as long range percolation on $V_{\mz}^{2^{n-k}}$ under the measure $\p_{\beta+\eps}$. We also define an analogy of the infinity-distance on $G^\prime$ by $\|r(u) - r(v)\|_\infty= \|u-v\|_\infty$. Remember that we write $\cN(r(u))=\left\{r(v)\in V^\prime : \|r(u)-r(v)\|_\infty \leq 1\right\}$ for the neighborhood of $r(u)$. Our goal is to bound the expected distance between the vertices $\mz$ and $(2^n-1)\mo$ from below, conditioned on the graph $G^\prime$. For this, we consider all self-avoiding paths $P^\prime = \left(r(v_0), r(v_1),\ldots, r(v_\ell) \right)$ between $r(\mz)$ and $r((2^{n-k}-1)\mo)$ in $G^\prime$. In the following, we always work on a certain event $\mathcal{H}_t$, which is defined by
	\begin{align*}
		\mathcal{H}_t = \bigcap_{\ell \geq t} \left\{	\left| \mathcal{CS}_\ell \left( G^\prime \right) \right| \leq 10^\ell \mu_{\beta+\eps}^\ell \right\} \cap 
		\bigcap_{\ell\geq t} \left\{ \overline{\deg}(Z) \leq 20 \mu_{\beta+\eps} \forall Z \in \mathcal{CS}_\ell \left( G^\prime \right) \right\} \text .
	\end{align*}
	The graph $G^\prime$ has the exact same distribution as long-range percolation on $V_\mz^{2^{n-k}}$ under the measure $\p_{\beta+\eps}$.
	So by Lemma \ref{lem:connnectedsetsinLRP} and Markov's inequality one has
	\begin{align}\label{eq:Htbound}
		& \notag \p_{\beta < k}^{\beta+\eps \geq k} \left( \mathcal{H}_t^c \right) \\
		& \notag \leq \sum_{\ell=t}^{\infty} \p_{\beta < k}^{\beta+\eps \geq k} \left( 	\left| \mathcal{CS}_\ell \left( G^\prime \right) \right| > 10^\ell \mu_{\beta+\eps}^\ell \right)  + 
		\sum_{\ell=t}^{\infty} \p_{\beta < k}^{\beta+\eps \geq k} \left( \exists  Z \in \mathcal{CS}_\ell \left( G^\prime \right) : \overline{\deg}(Z) > 20 \mu_{\beta+\eps}  \right)\\
		& \notag = \sum_{\ell=t}^{\infty}  \p_{\beta+\eps} \left( 	\left| \mathcal{CS}_\ell \left( V_\mz^{2^{n-k}} \right) \right| > 10^\ell \mu_{\beta+\eps}^\ell \right)
		+  
		\sum_{\ell=t}^{\infty} \p_{\beta+\eps} \left( \exists  Z \in \mathcal{CS}_\ell \left( V_\mz^{2^{n-k}} \right) : \overline{\deg}(Z) > 20 \mu_{\beta+\eps}  \right)
		\\
		& \notag
		\leq \frac{\E_{\beta+\eps} \left[ 	\left| \mathcal{CS}_\ell \left( \Z^d \right) \right| \right] }{ 10^\ell \mu_{\beta+\eps}^\ell  }
		+  
		\sum_{\ell=t}^{\infty} \p_{\beta+\eps} \left( \exists  Z \in \mathcal{CS}_\ell \left( \Z^d \right) : \overline{\deg}(Z) > 20 \mu_{\beta+\eps}  \right)\\
		& \leq \sum_{\ell=t}^{\infty} 0.4^\ell +  \sum_{\ell=t}^{\infty} e^{-4\mu_{\beta+\eps} \ell}
		\leq \sum_{\ell=t}^{\infty} 0.4^\ell +  \sum_{\ell=t}^{\infty} e^{- 4 \ell}
		\leq  \sum_{\ell=t}^{\infty} 0.5^\ell = 2\cdot 2^{-t}.
	\end{align}
	Let $P^\prime=\left(r(\mz)=r(v_0),\ldots,r(v_\ell)\right)$ be a self-avoiding path in $G^\prime$ starting at $r(\mz)$. Assume that $\ell$ is large enough (which will be specified later) and let $\delta$ be small enough such that
	\begin{align}\label{eq:supermult eps }
		2 \left(	 9^d 50 \mu_{\beta+1} \left( 1-f_2(\delta)\right) + \left(1-f_3(\delta)\right) \right)^{\frac{1}{30^d 200 \mu_{\beta+1}}}
		\leq
		\frac{1}{20 \mu_{\beta+1}}, 
	\end{align}
	where $f_2$ and $f_3$ are the functions defined in Lemma \ref{lem:goodeventA} and Lemma \ref{lem:goodeventD}.
	We will see later on, where this condition on $\delta$ comes from. We will now describe what it means for a block $V_{v_i}^{2^k}$ to be {\sl good}; we will also say that the vertex $r(v_i) \in G^\prime$ is good in this case. Intuitively, a block being good ensures that a path in the original model that passes through this block needs to walk a distance of at least $\delta \Lambda(2^k,\beta)$. Let $P^\prime = \left(r(\mz)=r(v_0),\ldots,r(v_\ell) = r((2^{n-k}-1)\mo)\right)$ be a self-avoiding path in $G^\prime$.
	Suppose that $\|v_{i} - v_{i+1} \|_\infty \geq 2$. Then we call the block $V_{v_i}^{2^k}$ $($and the vertex $r(v_i))$ good $($for the path $P^\prime)$ if
	\begin{align*}
		D_{V_{v_{i}}^{2^k}}(x,y) \geq \delta 2^{k\theta(\beta)} \text{ for all } x,y \in V_{v_{i}}^{2^k} \text{ with } x \sim V_{v_{i+1}}^{2^k}, y \sim V_w^{2^k}, w \notin \{v_{i}, v_{i+1}\}.
	\end{align*}
	If $\|v_{i} - v_{i+1} \|_\infty = 1$, we call the block $V_{v_i}^{2^k}$ $($and the vertex $r(v_i))$ good $($for the path $P^\prime)$ if
	\begin{align*}
		D^\star_{V_\mz^{2^n}} \left( V_{v_{i}}^{2^k} , \bigcup_{r(w) \in G^\prime : \|w-v_{i} \|_\infty \geq 2 } V_w^{2^k}\right) \geq \delta 2^{k\theta(\beta)}  \text .
	\end{align*}
	
	\noindent
	\textbf{Part (C)}.
	We define the set 
	\begin{equation*}
		\widetilde{R}_\ell= \bigcup_{i=0}^{\ell-1}  \cN\left(r(v_i)\right) \text .
	\end{equation*}
	The set $\widetilde{R}_\ell$ is a connected set in $G^\prime$, containing the origin $r(v_0)=r(\mz)$, and its size is bounded from above and below by
	\begin{equation}\label{widetilde ub}
		\ell \leq |\widetilde{R}_\ell|\leq 3^d \ell  \text.
	\end{equation}
	Assuming that the event $\cH_{\ell}$ holds, the average degree of the set $\widetilde{R}_\ell$ is bounded by $20 \mu_{\beta+\eps}$. If $r(u) \in \widetilde{R}_\ell$, then at least one of $r(v_0),\ldots,r(v_{\ell-1})$ is in $\cN(r(u))$. Thus, there can be at most $3^d$ many indices $i \in \{0,\ldots,\ell-1\}$ for which $r(u) \in \cN(r(v_i))$. From this, we can directly deduce that
	\begin{align*}
		& \sum_{i=0}^{\ell-1} \deg^{\cN} \left(r(v_i)\right)
		=
		\sum_{i=0}^{\ell-1} \ \sum_{r(u) \in \cN(r(v_i))} \deg \left(r(u)\right)
		\\
		&
		=
		\sum_{r(u)\in \widetilde{R}_\ell} \Big|\left\{ i \in \{0,\ldots,\ell-1\} : r(u) \in \cN(r(v_i)) \right\}  \Big| \deg(r(u))
		\\
		&
		\leq
		3^d
		\sum_{r(u)\in \widetilde{R}_\ell} \deg(r(u))
		=
		3^d \left|\widetilde{R}_\ell\right| \overline{\deg} \left(\widetilde{R}_\ell\right)
		\leq 3^d \left|\widetilde{R}_\ell\right| 20 \mu_{\beta+\eps}  \overset{\eqref{widetilde ub}}{\leq} 9^d 20 \mu_{\beta+\eps} \ell .
	\end{align*}
	This implies that there can be at most $\tfrac{2\ell}{5}$ many indices $i \in \{0,\ldots, \ell-1\}$ with $\deg^{\cN} \left(r(v_i)\right) > 9^d 50 \mu_{\beta+\eps}$. Indeed if there were strictly more than $\tfrac{2\ell}{5}$ many indices $i \in \{0,\ldots, \ell-1\}$ with $\deg^{\cN} \left(r(v_i)\right) > 9^d 50 \mu_{\beta+\eps}$, then
	\begin{equation*}
		\sum_{i=0}^{\ell-1} \deg^{\cN} (r(v_i)) > \frac{2\ell}{5}  \cdot  9^d 50 \mu_{\beta+\eps} = 9^d 20 \mu_{\beta+\eps} \ell ,
	\end{equation*}
	which is a contradiction. Thus, there need to be at least $\tfrac{\ell}{2} \leq \ell - \tfrac{2 \ell}{5} -1$ many indices $i \in \{1,\ldots, \ell-1\}$ with $\deg^{\cN} \left(r(v_i)\right) \leq 9^d 50 \mu_{\beta+\eps}$. 
	Say that $\mathcal{I}\subset \{1,\ldots,\ell-1\}$ is the set of indices $i$ with $\deg^{\cN} \left(r(v_i)\right) \leq 9^d 50 \mu_{\beta+\eps}$. Then $\left|\mathcal{I}\right| \geq \tfrac{\ell}{2}$. \\
	
	We now define a set of special indices $\mathcal{IND}(P^\prime) \subset \mathcal{I} \subset \{1,\ldots,\ell-1\}$ via the algorithm below. For abbreviation, we will mostly just write $\mathcal{IND}$ for $\mathcal{IND}(P^\prime)$, but one should keep in mind that the indices really depend on the chosen path $P^\prime$.
	\begin{enumerate}\addtocounter{enumi}{-1}
		\item Start with $\mathcal{IND}_0 = \emptyset$.
		\item For $i=1,\ldots,\ell-1$: \\ If $\deg^{\cN} \left(r(v_i)\right) \leq 9^d 50 \mu_{\beta+\eps}$ and $\cN \left(r(v_i)\right) \nsim \bigcup_{j \in \mathcal{IND}_{i-1}} \cN \left( r(v_j) \right)$, then define $\mathcal{IND}_{i} = \mathcal{IND}_{i-1} \cup \{i\}$. Otherwise set $\mathcal{IND}_{i} = \mathcal{IND}_{i-1}$.
		\item Set $\mathcal{IND} \coloneqq \mathcal{IND}_{\ell-1}$.
	\end{enumerate}
	
	For all $i\in \mathcal{I}$ we have either $i \in \mathcal{IND}$ or $\cN(r(v_i)) \sim \cN(r(v_j))$ for some $j<i$ with $j\in \mathcal{IND}$. In particular, we see that
	\begin{align}\label{eq:inclusion}
		\mathcal{I} = \bigcup_{j\in \mathcal{IND}} \left(\{j\} \cup \left\{ i \in \mathcal{I} : \cN(r(v_i)) \sim \cN(r(v_j)) \right\}  \right).
	\end{align}
	For each fixed $j\in \mathcal{IND}$, there can be at most $\deg^{\cN}(r(v_j)) \leq 9^d 50 \mu_{\beta+\eps}$ many $r(u) \in G^\prime$ with $r(u) \sim \cN(r(v_j))$. Thus, there can also be at most $3^d \deg^{\cN}(r(v_j)) \leq 27^d 50 \mu_{\beta+\eps}$ many $r(u) \in G^\prime$ with $\cN(r(u)) \sim \cN(r(v_j))$. So in particular each of the sets $\left\{ i \in \mathcal{I} : \cN(r(v_i)) \sim \cN(r(v_j)) \right\}$ in the union above has size at most $27^d 50 \mu_{\beta+\eps}$. Taking the size of the sets in \eqref{eq:inclusion} and applying a union bound, we get that
	\begin{align*}
		\frac{\ell}{2} \leq \left|\mathcal{I}\right| & \leq \sum_{j\in \mathcal{IND}} \left|\left(\{j\} \cup \left\{ i \in \mathcal{I} : \cN(r(v_i)) \sim \cN(r(v_j)) \right\}  \right)\right|
		\leq
		\sum_{j\in \mathcal{IND}} \left(1+ 27^d 50 \mu_{\beta+\eps} \right)
		\\
		&
		=
		\left|\mathcal{IND}\right|  \left(1+ 27^d 50 \mu_{\beta+\eps} \right) .
	\end{align*}
	Rearranging this inequality, we see that
	\begin{align*}
		|\mathcal{IND}| \geq \frac{\frac{\ell}{2}}{27^d 50 \mu_{\beta+\eps}+1} \geq \frac{\ell}{30^d 100 \mu_{\beta+1}} \text .
	\end{align*}
	
	Next, we want to lower bound the probability that a block $r(v_i)$ with $i \in \mathcal{IND}$ is good, given the graph $G^\prime$. Remember that we had two different notions what it meant that $r(v_i)$ is good - depending on whether $\|v_i-v_{i+1}\|_\infty \geq 2$ or $\|v_i-v_{i+1}\|_\infty = 1$. 
	
	Assume that $\|v_{i}- v_{i+1} \|_\infty \geq 2$. If $r(v_i)$ is not good, then there exists $r(w) \in V^\prime \setminus \{r(v_i),r(v_{i+1})\}$ for which the event 
	\begin{align*}
		\mathcal{M}_w \coloneqq \bigcup_{\substack{ x \in V_{v_i}^{2^k}: \\ x \sim V_w^{2^k} }} \
		\bigcup_{\substack{ y \in V_{v_i}^{2^k}: \\ y \sim V_{v_{i+1}}^{2^k} }} \left\{D_{v_i}^{2^k} (x,y) < \delta 2^{k\theta(\beta)}  \right\}
	\end{align*}
	holds. By Lemma \ref{lem:goodeventA} and translation invariance, we have for all large enough $k$ and all $r(w) \in V^\prime \setminus \{r(v_i),r(v_{i+1})\}$ with $r(w) \sim r(v_i)$ that
	\begin{align*}
		\p_{\beta < k}^{\beta+\eps \geq k} \left( \mathcal{M}_w \big| G^\prime \right)
		=
		\p_{\beta < k}^{\beta+\eps \geq k} \left( \mathcal{M}_w \big| V_w^{2^k} \sim V_{v_i}^{2^k} \sim V_{v_{i+1}}^{2^k} \right) \leq (1-f_2(\delta)) .
	\end{align*}
	A union bound over all $r(w) \in V^\prime \setminus \{r(v_i),r(v_{i+1})\}$ with $r(w) \sim r(v_i)$ implies that for large enough $k$ the conditional probability that $r(v_i)$ is not good is bounded by
	\begin{align}\label{eq:ungleichung1}
		\notag & \p_{\beta < k}^{\beta+\eps \geq k} \left( r(v_i) \text{ not good} \Big| G^\prime \right)
		\leq
		\sum_{\substack{ r(w) \in V^\prime \setminus \{r(v_i),r(v_{i+1})\} : \\ r(w) \sim r(v_i) }} \p_{\beta < k}^{\beta+\eps \geq k} \left( \mathcal{M}_w \Big| G^\prime \right)
		\\
		&
		\leq
		\sum_{\substack{ r(w) \in V^\prime \setminus \{r(v_i),r(v_{i+1})\} : \\ r(w) \sim r(v_i) }} \left(1-f_2(\delta)\right)
		\leq 
		\deg(r(v_i)) \left( 1-f_2(\delta)\right) \leq 9^d 50 \mu_{\beta+1} \left( 1-f_2(\delta)\right),
	\end{align}
	where for the last inequality we used that $\deg(r(v_i)) \leq \deg^{\cN}(r(v_i)) \leq  9^d 50 \mu_{\beta+1}$, since we assumed that $i \in \mathcal{IND}$.\\
	
	Next, we estimate the probability that the vertex $r(v_i)$ is not good if $\|v_{i}- v_{i+1} \|_\infty = 1$.
	So assume that $\|v_{i}- v_{i+1} \|_\infty = 1$. Since $i \in \mathcal{IND}$, we have that $\deg^{\cN}\left(r(v_i)\right) \leq 9^d 100 \mu_{\beta+1}$. Thus - by translation-invariance - we can apply Lemma \ref{lem:goodeventD}. This implies that for all large enough $k$, the conditional probability that $r(v_i)$ is not good is bounded by
	\begin{align}\label{eq:ungleichung2}
		\notag \p_{\beta < k}^{\beta+\eps \geq k}&  \left(r(v_i) \text{ not good } \big| G^\prime \right)
		=
		\p_{\beta < k}^{\beta+\eps \geq k}   \left( D^{\star}_{V_\mz^{2^n}}\Big(V_{v_{i}}^{2^k} , \bigcup_{r(v) : \|v-v_{i}\|_\infty \geq 2} V_{v}^{2^k} \Big) < \delta 2^{k\theta(\beta)} \ \big| \ G^\prime \right)\\
		&
		\leq (1-f_3(\delta)).
	\end{align}
	
	Combining the results of inequalities \eqref{eq:ungleichung1} and \eqref{eq:ungleichung2} we see that for all $i\in \mathcal{IND}$, no matter whether $\|v_i-v_{i+1}\|_\infty = 1$ or $\|v_i-v_{i+1}\|_\infty \geq 2$, we always have that
	\begin{align*}
		\p_\beta \left( r(v_i) \text{ not good } | \ G^\prime \right) \leq  9^d 50 \mu_{\beta+1} \left( 1-f_2(\delta)\right) + \left(1-f_3(\delta)\right) \eqqcolon \widetilde{f}(\delta) \text .
	\end{align*}

	Whether a block $V_{v_i}^{2^k}$ is good in the path $P^\prime$ depends only on the edges with at least one end in $\cN \left(r(v_i)\right)$. So in particular for different indices $i\in \mathcal{IND}$, it is independent whether the underlying blocks $V_{v_i}^{2^k}$ are good. Using that $\left|\mathcal{IND}\right| \geq \frac{\ell}{30^d 100 \mu_{\beta+1}}$, we get that
	\begin{align}\label{eq:20mubeta}
		& \notag \p_{\beta < k}^{\beta+\eps \geq k} \left( \big|\left\{i \in \mathcal{IND}(P^\prime) : r(v_i) \text{ good }\right\}\big| \leq \frac{\ell}{30^d 200 \mu_{\beta+1}} \ \Big| \ G^\prime  \right)\\
		& \notag \leq \p_{\beta < k}^{\beta+\eps \geq k} \left( \bigcup_{U\subset \mathcal{IND}(P^\prime) : |U| \geq \frac{\left|\mathcal{IND}(P^\prime)\right|}{2} } \left\{ r(v_i) \text{ not good for all } i\in U \right\}  \ \Big| \ G^\prime  \right)
		\\
		& \notag
		\leq \sum_{U\subset \mathcal{IND}(P^\prime) : |U| \geq \frac{\left|\mathcal{IND}(P^\prime)\right|}{2} } \ \prod_{i\in U} \ \p_{\beta < k}^{\beta+\eps \geq k} \left(   r(v_i) \text{ not good }  \ \Big| \ G^\prime  \right)
		\\
		&
		\leq 2^{|\mathcal{IND}(P^\prime)|}
		\left(\widetilde{f}(\delta)\right)^{\frac{|\mathcal{IND}(P^\prime)|}{2}}
		\leq 2^{\ell} \left(\widetilde{f}(\delta)\right)^{\frac{\ell}{30^d 200 \mu_{\beta+1}}}
		\leq \left(20 \mu_{\beta+1}\right)^{-\ell}
	\end{align}
	where the last inequality holds because of our assumption on $\delta$ \eqref{eq:supermult eps }.\\
	
	So far, we only considered one fixed path $P^\prime$ in $G^\prime$ of length $\ell$, assuming that the event $\cH_{\ell}$ holds. Conditioned on the event $\cH_{\ell}$, there can be at most $10^\ell \mu_{\beta+\eps}^\ell$ many paths in $G^\prime$ which have length $\ell$ and start at the origin. For abbreviation, we define the event $\cG_\ell$ by
	\begin{align*}
		\cG_\ell^c = \left\{\exists P^\prime \text{ in } G^\prime \text{ of length } \ell \text{ s.t. } \left|\left\{i \in \mathcal{IND}(P^\prime) : r(v_i) \text{ good for } P^\prime\right\}\right| \leq \frac{\ell}{30^d 200 \mu_{\beta+1}}\right\} ,
	\end{align*}
	where we say $P^\prime \text{ in } G^\prime$ if a path $P^\prime$ starts at $r(\mz)$ and is contained in the graph $G^\prime$.
	So a union bound over the at most $(10\mu_{\beta+\eps})^\ell$ paths $P^\prime$ in $G^\prime$ of length $\ell$ starting at $r(\mz)$ implies that
	\begin{align*}
		&\p_{\beta < k}^{\beta+\eps \geq k} \left( \exists P^\prime \text{ in } G^\prime \text{ of length } \ell \text{ s.t.} \left|\left\{i \in \mathcal{IND}(P^\prime) : r(v_i) \text{ good for  } P^\prime\right\}\right| \leq \frac{\ell}{30^d 200 \mu_{\beta+1}} \ \Big| \ \mathcal{H}_\ell  \right)\\
		&
		= \p_{\beta < k}^{\beta+\eps \geq k} \left( \mathcal{G}_\ell^c \ \Big| \ \mathcal{H}_\ell  \right)
		\overset{\eqref{eq:20mubeta}}{\leq} (10 \mu_{\beta+\eps})^\ell \left(20 \mu_{\beta+1}\right)^{-\ell} \leq 2^{-\ell}.
	\end{align*}
	Using that $\p_{\beta < k}^{\beta+\eps \geq k} \left( \mathcal{H}_\ell^c \right) \leq 2 \cdot 2^{-\ell}$ \eqref{eq:Htbound}, we thus get that
	\begin{align}\label{glc0}
		& \p_{\beta < k}^{\beta+\eps \geq k} \left( \mathcal{G}_\ell^c \right)
		\leq
		\p_{\beta < k}^{\beta+\eps \geq k} \left( \mathcal{G}_\ell^c \ \Big| \ \mathcal{H}_\ell  \right) 
		+ \p_{\beta < k}^{\beta+\eps \geq k} \left( \mathcal{H}_\ell^c \right)
		\leq
		3 \cdot 2^{-\ell}.
	\end{align}
	We also define the event $\mathcal{G}_{\geq \ell}$ by $\mathcal{G}_{\geq \ell}^c = \bigcup_{t=\ell}^{\infty} \mathcal{G}_t^c = \left( \bigcap_{t=\ell}^{\infty} \mathcal{G}_t\right)^c$. A union bound over all $t\geq \ell$ directly implies that
	\begin{equation}\label{glc}
		\p_{\beta < k}^{\beta+\eps \geq k} \left( \mathcal{G}_{\geq \ell}^c \right) 
		\leq 
		\sum_{t=\ell}^{\infty} \p_{\beta < k}^{\beta+\eps \geq k} \left( \mathcal{G}_{t}^c \right)
		\leq
		\sum_{t=\ell}^{\infty} 3 \cdot 2^{-t} \leq 6 \cdot 2^{-\ell}
	\end{equation}
	
	\noindent
	\textbf{Part (D)}.
	Assume that the events $\cG_{\geq \ell}$ and $\left\{D_{G^\prime}\left(r(\mz),r((2^{n-k}-1)\mo)\right) = \ell\right\}$ both hold and that $P$ is a geodesic in $V_\mz^{2^n}$ between $\mz$ and $(2^n-1)\mo$, which goes through the boxes $V_{u_0}^{2^k},\ldots,V_{u_L}^{2^k}$ in this order. Define the path $P^\star = \left(r(\mz)=r(u_0), r(u_1),\ldots,r(u_L) = r((2^{n-k}-1)\mo) \right)$ and let $P^\prime = \left(r(\mz)=r(v_0), r(v_1),\ldots,r(v_{\ell^\prime}) = r((2^{n-k}-1)\mo) \right)$ be the loop-erasure of $P^\star$. The condition $\left\{D_{G^\prime}\left(r(\mz),r((2^{n-k}-1)\mo)\right) = \ell\right\}$ directly implies that $\ell^\prime \geq \ell$. For each index $i\in \mathcal{IND}(P^\prime)$, the path $P$ needs to walk a distance of at least $\delta 2^{k\theta(\beta)}$ in the set $\bigcup_{r(w) \in \cN(r(v_i))} V_{v_i}^{2^k}$, as long as $r(\mz),r((2^{n-k}-1)\mo) \notin \cN(r(v_i))$. This holds for all indices $i\in \mathcal{IND}(P^\prime)$ for which $r(\mz),r((2^{n-k}-1)\mo) \notin \cN(r(v_i))$. The sets $\left(\bigcup_{r(w) \in \cN(r(v_i))} V_{v_i}^{2^k}\right)_{i\in \mathcal{IND(P^\prime)}}$ are disjoint by construction. So in particular, we get that
	\begin{equation}\label{lengthind}
		D_{V_\mz^{2^n}} \left(\mz,(2^n-1)\mo\right) =  \text{length}(P) \geq \left(|\mathcal{IND}(P^\prime)|-2\right) \delta 2^{k\theta(\beta)} 
		.
	\end{equation}
	As we assumed that the event $\cG_{\geq \ell}$ holds, we have that $\left|\mathcal{IND}(P^\prime)\right| \geq \frac{\ell^\prime}{30^d 200 \mu_{\beta+1}} \geq \frac{\ell}{30^d 200 \mu_{\beta+1}}$. Inserting this into \eqref{lengthind} implies that that
	\begin{align*}
		&D_{V_\mz^{2^n}}\left(\mz,(2^n-1)\mo\right) 
		\geq
		\left(\frac{\ell}{30^d 200 \mu_{\beta+1}} - 2 \right) \delta 2^{k\theta(\beta)}
		\geq
		\left(\frac{\ell}{30^d 300 \mu_{\beta+1}}  \right) \delta 2^{k\theta(\beta)}
	\end{align*}
	for all large enough $\ell$. So in total we get that for large enough $\widetilde{\ell} \in \N$ one has
	\begin{align}\label{eq:finalsubmulti}
		\notag &\E_{\beta < k}^{\beta+\eps \geq k} \left[ D_{V_{\mz}^{2^n}}(\mz,(2^n-1)\mo) \right] \geq 
		\sum_{\ell=\widetilde{\ell}}^{\infty} 
		\E_{\beta < k}^{\beta+\eps \geq k} \left[ D_{V_{\mz}^{2^n}}(\mz,(2^n-1)\mo) \mathbbm{1}_{\mathcal{G}_{\geq \ell}} \mathbbm{1}_{\left\{ D_{G^\prime}(r(\mz),r((2^{n-k}-1)\mo)) = \ell \right\}} \right]\\
		& \notag
		\geq 
		\sum_{\ell=\widetilde{\ell}}^{\infty} 
		\E_{\beta < k}^{\beta+\eps \geq k} \left[ \frac{\ell}{30^d 300 \mu_{\beta+1}} \delta 2^{k\theta(\beta)} \mathbbm{1}_{\mathcal{G}_{\geq \ell}} \mathbbm{1}_{\left\{ D_{G^\prime}(r(\mz),r((2^{n-k}-1)\mo)) = \ell \right\}} \right]
		\\
		& =  \frac{\delta 2^{k \theta(\beta)}}{30^d 300 \mu_{\beta+1}} \sum_{\ell=\widetilde{\ell}}^{\infty} \ell \ \E_{\beta < k}^{\beta+\eps \geq k} \left[ \mathbbm{1}_{\mathcal{G}_{\geq \ell}} \mathbbm{1}_{\left\{ D_{G^\prime}(r(\mz),r((2^{n-k}-1)\mo)) = \ell \right\}} \right] .
	\end{align}
	We can further bound the last sum by
	\begin{align}\label{eq:bigineq}
		& \notag \sum_{\ell=\widetilde{\ell}}^{\infty} \ell \ \E_{\beta < k}^{\beta+\eps \geq k} \left[ \mathbbm{1}_{\mathcal{G}_{\geq \ell}} \mathbbm{1}_{\left\{ D_{G^\prime}(r(\mz),r((2^{n-k}-1)\mo)) = \ell \right\}} \right] \\
		& \notag =  
		\sum_{\ell=\widetilde{\ell}}^{\infty} \ell \ \E_{\beta < k}^{\beta+\eps \geq k} \left[ \mathbbm{1}_{\left\{ D_{G^\prime}(r(\mz),r((2^{n-k}-1)\mo)) = \ell \right\}} \right]
		-
		\sum_{\ell=\widetilde{\ell}}^{\infty} \ell \ \E_{\beta < k}^{\beta+\eps \geq k} \left[ \mathbbm{1}_{\left\{\mathcal{G}_{\geq \ell}^c \right\}} \mathbbm{1}_{\left\{ D_{G^\prime}(r(\mz),r((2^{n-k}-1)\mo)) = \ell \right\}} \right]  \\
		& \notag \geq  \sum_{\ell=\widetilde{\ell}}^{\infty} \ell \ \p_{\beta < k}^{\beta+\eps \geq k} \left( D_{G^\prime}(r(\mz),r((2^{n-k}-1)\mo)) = \ell \right)
		-
		\sum_{\ell=\widetilde{\ell}}^{\infty} \ell \ \E_{\beta < k}^{\beta+\eps \geq k} \left[ \mathbbm{1}_{\left\{\mathcal{G}_{\geq \ell}^c \right\}}  \right]  \\
		& \notag \overset{\eqref{glc}}{\geq} \sum_{\ell=\widetilde{\ell}}^{\infty} \ell \ \p_{\beta + \eps} \left(  D_{V_{\mz}^{2^{n-k}}}(\mz,(2^{n-k}-1)\mo) = \ell \right)
		-
		6\sum_{\ell=\widetilde{\ell}}^{\infty} \ell 2^{-\ell} \\
		& \notag \geq \sum_{\ell=1}^{\infty} \ell \ \p_{\beta + \eps} \left(  D_{V_{\mz}^{2^{n-k}}}(\mz,(2^{n-k}-1)\mo) = \ell \right)
		-
		\sum_{\ell=1}^{\widetilde{\ell}-1} \ell \ \p_{\beta + \eps} \left(  D_{V_{\mz}^{2^{n-k}}}(\mz,(2^{n-k}-1)\mo) = \ell \right)
		-
		12 \\
		&
		\geq \E_{\beta+\eps}\left[ D_{V_{\mz}^{2^{n-k}}}(\mz,(2^{n-k}-1)\mo) \right]
		-
		\widetilde{\ell} 
		-
		12 \geq \frac{1}{2} \E_{\beta+\eps}\left[ D_{V_{\mz}^{2^{n-k}}}(\mz,(2^{n-k}-1)\mo) \right],
	\end{align}
	where the last inequality holds for $n-k$ large enough.  Using further that
	\begin{equation*}
		\E_{\beta+\eps}\left[ D_{V_{\mz}^{2^{n-k}}}(\mz,(2^{n-k}-1)\mo) \right] \geq \frac{\Lambda(2^{n-k},\beta+\eps)-1}{36d } \geq \frac{\Lambda(2^{n-k},\beta+\eps)}{72d }
	\end{equation*}
	for $n-k>0$ (see \cite[Remark 4.3]{baeumler2022distances}) and that $2^{k\theta(\beta)} \geq c(\beta) \Lambda \left(2^k, \beta\right)$ for some $c(\beta) > 0$ \eqref{eq:subsuperimpli}, we can insert inequality \eqref{eq:bigineq} into \eqref{eq:finalsubmulti} and get that
	\begin{align*}
		& \E_{\beta < k}^{\beta+\eps \geq k} \left[ D_{V_{\mz}^{2^n}}(\mz,(2^n-1)\mo) \right]
		\overset{\eqref{eq:finalsubmulti}}{\geq}
		\frac{\delta 2^{k \theta(\beta)}}{30^d 300 \mu_{\beta+1}} \sum_{\ell=\widetilde{\ell}}^{\infty} \ell \ \E_{\beta < k}^{\beta+\eps \geq k} \left[ \mathbbm{1}_{\mathcal{G}_{\geq \ell}} \mathbbm{1}_{\left\{ D_{G^\prime}(r(\mz),r((2^{n-k}-1)\mo)) = \ell \right\}} \right]  
		\\
		&
		\overset{\eqref{eq:bigineq}}{\geq}
		\frac{\delta 2^{k \theta(\beta)}}{30^d 300 \mu_{\beta+1}} \cdot \frac{1}{2} \E_{\beta+\eps}\left[ D_{V_{\mz}^{2^{n-k}}}(\mz,(2^{n-k}-1)\mo) \right]
		\geq
		\frac{\delta c(\beta) \Lambda(2^k,\beta)}{30^d 300 \mu_{\beta+1}} \cdot \frac{1}{2} \frac{\Lambda(2^{n-k},\beta+\eps)}{72 d}\\
		&
		=
		c_\beta \Lambda(2^k,\beta) \Lambda(2^{n-k},\beta+\eps),
	\end{align*}
	which shows \eqref{eq:lower bound} and thus finishes the proof of Lemma \ref{lem:uniformovereps}. 
\end{proof}

Finally, we give the proof of Lemma \ref{lem:uniform2ndmomentbound:renormalized}.

\begin{proof}[Proof of Lemma \ref{lem:uniform2ndmomentbound:renormalized}]
	We use a renormalization argument for this proof. We first define the renormalized graph $G^\prime$ where we contract all vertices of the set $V_u^{2^k}$ to one vertex $r(u)$ and do this for all $u \in V_\mz^{2^{n-k}}$. In the graph $G^\prime$, there is an edge between $r(u)$ and $r(v)$ if and only if there is an edge between $V_u^{2^k}$ and $V_v^{2^k}$. Now, let $x,y\in V_\mz^{2^n}$ be arbitrary, say with $x \in V_u^{2^k}$ and $y \in V_v^{2^k}$. If $u=v$, note that the distance $D_{V_u^{2^k}}\left(x,y\right)$ has the same distribution under the measures $\p_{\beta < k}^{\beta+\eps \geq k}$ and $\p_{\beta}$. Thus we get that
	\begin{equation*}
		\E_{\beta < k}^{\beta+\eps \geq k} \left[D_{V_\mz^{2^n}}\left(x,y\right)^2 \right] 
		\leq
		\E_{\beta < k}^{\beta+\eps \geq k} \left[D_{V_u^{2^k}}\left(x,y\right)^2 \right]
		=
		\E_{\beta} \left[D_{V_u^{2^k}}\left(x,y\right)^2 \right]
		\leq
		\tilde{C_\beta} \Lambda(2^k,\beta)^2
	\end{equation*}
	by Lemma \ref{lem:uniform 2nd moment bound}, so inequality \eqref{eq:uniformoverbeta} holds. From here on, we consider the case $u\neq v$. Let $\left(r(u)=r(u_0),\ldots,r(u_K)=r(v)\right)$ be the shortest path between $r(u)$ and $r(v)$ in the graph $G^\prime$, where $K=D_{G^\prime}\left(r(u),r(v)\right)$. In case there are several shortest paths between $r(u)$ and $r(v)$ in the graph $G^\prime$, we choose one by some deterministic rule. Using this geodesic in $G^\prime$, we now construct a new path in the original graph $V_\mz^{2^n}$. Define $x_0=x$ and $y_K=y$ and for $i\in \{0,\ldots,K-1\}$, let $(y_{i},x_{i+1}) \in V_{u_i}^{2^k} \times V_{u_{i+1}}^{2^k}$ be such that $y_{i} \sim x_{i+1}$. For each $i\in \{0,\ldots,K\}$, there exists a path between $x_i$ and $y_i$ that stays entirely within $V_{u_i}^{2^k}$. Concatenating these $K+1$ path, we can inductively build a path via the relations
	\begin{equation*}
		x= x_0 \overset{V_{u_0}^{2^k}}{\longleftrightarrow} y_0
		\sim
		x_1 \overset{V_{u_1}^{2^k}}{\longleftrightarrow} y_1
		\sim
		x_2 \overset{V_{u_2}^{2^k}}{\longleftrightarrow} \ldots \ldots \ldots \ldots
		\overset{V_{u_{K-1}}^{2^k}}{\longleftrightarrow}
		y_{K-1}
		\sim
		x_K \overset{V_{u_K}^{2^k}}{\longleftrightarrow} y_K = y.
	\end{equation*}
	By concatenating the geodesics between $x_i$ and $y_i$ in the sets $V_{u_i}^{2^k}$, we get the upper bound on the graph distance between $x$ and $y$ 
	\begin{align}\label{eq:sum_diabound}
		D_{V_\mz^{2^n}}\left(x,y\right) & \leq \sum_{i=0}^K \left(  \dia\left( V_{u_i}^{2^k} \right) + 1 \right) \text .
	\end{align}
	For $i\in \{0,\ldots,K\}$, define $X_i \coloneqq \dia\left( V_{u_i}^{2^k} \right)$. 
	For $i>K$, let $ \left(X_i \right)_{i>K}$ be i.i.d. copies of $\dia\left( V_{u_0}^{2^k} \right)$ that are also independent of $K=D_{G^\prime}\left(r(u),r(v)\right)$ and of $ \left(\dia\left( V_{u_i}^{2^k} \right)\right)_{i \leq K}$. For $i\leq K$, the diameters $\dia\left( V_{u_i}^{2^k} \right)$ depend only on edges with both endpoints inside $ V_{u_i}^{2^k}$, whereas the distance $K=D_{G^\prime}\left(r(u),r(v)\right)$ depends only on edges that are between two different boxes. 
	So in particular the random variables $\left( X_i \right)_{i\in \N_0}$ are i.i.d. and furthermore independent of  $K=D_{G^\prime}\left(r(u),r(v)\right)$. Thus  we get that
	\begin{align}\label{eq:wald 2nd}
		& \notag \E_{\beta < k}^{\beta+\eps \geq k}\left[\left(\sum_{i=0}^{K} X_i\right)^2 \right] 
		\leq \E_{\beta < k}^{\beta+\eps \geq k}\left[\left(\sum_{i=0}^\infty \mathbbm{1}_{\{i\leq K\}}  X_i\right)^2 \right] 
		\\
		& \notag
		= \E_{\beta < k}^{\beta+\eps \geq k}\left[ \sum_{i=0}^\infty \sum_{j=0}^\infty \mathbbm{1}_{\{i\leq K\}} \mathbbm{1}_{\{j\leq K\}}  X_i X_j \right] 
		=  \sum_{i=0}^\infty \sum_{j=0}^\infty \E_{\beta < k}^{\beta+\eps \geq k}\left[ \mathbbm{1}_{\{i\leq K\}} \mathbbm{1}_{\{j\leq K\}} \right] \E_{\beta < k}^{\beta+\eps \geq k} \left[  X_i X_j \right] 
		\\
		&
		\leq \E_{\beta < k}^{\beta+\eps \geq k} \left[ \sum_{i=0}^\infty \sum_{j=0}^\infty \mathbbm{1}_{\{i\leq K\}} \mathbbm{1}_{\{j\leq K\}}  \right] \E_{\beta < k}^{\beta+\eps \geq k}\left[X_1^2\right] 
		= \E_{\beta < k}^{\beta+\eps \geq k}\left[(K+1)^2\right] \E_{\beta < k}^{\beta+\eps \geq k}\left[X_1^2\right] \text .
	\end{align}
	The distance $D_{G^\prime}\left(r(u),r(v)\right)$ only depends on the occupation status of edges with both ends in $V_\mz^{2^n}$ that have a length of at least $2^k$. Thus $K=D_{G^\prime}\left(r(u),r(v)\right)$ has exactly the same distribution as $D_{V_\mz^{2^{n-k}}}\left(u,v \right)$ under the measure $\p_{\beta+\eps}$, which implies that
	\begin{equation}\label{2nd n-k}
		\E_{\beta < k}^{\beta+\eps \geq k}\left[(K+1)^2\right] \leq 4 \E_{\beta < k}^{\beta+\eps \geq k}\left[K^2\right] = 4 \E_{\beta+\eps} \left [D_{V_\mz^{2^{n-k}}} \left(u,v \right)^2 \right]
		\overset{\eqref{eq:second moment bound}}{\leq} 4 \tilde{C_\beta} \Lambda\left(2^{n-k},\beta+\eps\right)^2.
	\end{equation}
	Further, we know that $\E_{\beta < k}^{\beta+\eps \geq k} \left[X_1^2\right] = \E_\beta \left[ \dia\left( V_{u_1}^{2^k} \right)^2\right] \leq \tilde{C_\beta} \Lambda\left(2^k,\beta\right)^2$, which follows from Lemma \ref{lem:uniform 2nd moment bound} applied to $\eps=0$.
	Together with \eqref{eq:sum_diabound} and \eqref{eq:wald 2nd}, we thus get that
	\begin{align*}
		&\E_{\beta < k}^{\beta+\eps \geq k} \left[D_{V_\mz^{2^n}}\left(x,y\right)^2 \right] 
		\overset{\eqref{eq:sum_diabound}}{\leq} \E_{\beta < k}^{\beta+\eps \geq k} \left[\left(\sum_{i=0}^K \left(  \dia\left( V_{u_i}^{2^k} \right) + 1 \right)\right)^2\right]
		\\
		&
		\leq
		\E_{\beta < k}^{\beta+\eps \geq k} \left[\left(\sum_{i=0}^K \left( X_i + 1 \right)\right)^2\right]
		\leq
		\E_{\beta < k}^{\beta+\eps \geq k} \left[\left( 2\sum_{i=0}^K X_i \right)^2\right]
		=
		4 \E_{\beta < k}^{\beta+\eps \geq k} \left[\left( \sum_{i=0}^K X_i \right)^2\right]\\
		&
		\overset{\eqref{eq:wald 2nd}}{\leq}
		4 \ \E_{\beta < k}^{\beta+\eps \geq k} \left[ (K+1)^2 \right] \E_{\beta < k}^{\beta+\eps \geq k} \left[X_1^2 \right] 
		\leq 4 \ \E_{\beta + \eps} \left[(K+1)^2 \right]  \tilde{C_\beta} \Lambda(2^k,\beta)^2
		\\
		&
		\overset{\eqref{2nd n-k}}{\leq}  \left(4  \tilde{C_\beta}\right)^2 \Lambda(2^{n-k},\beta)^2 \Lambda(2^k,\beta)^2 .
	\end{align*}
\end{proof}







\textbf{Acknowledgments.} I thank Noam Berger for  many useful discussions. I also thank Yuki Tokushige for helpful comments on an earlier version of this paper. This work is supported by TopMath, the graduate program of the Elite Network of Bavaria and the graduate center of TUM Graduate School. I thank two anonymous referees for very useful comments and suggestions.

\end{document}